\documentclass[11pt]{amsart}

\usepackage{amsfonts,epsfig,siunitx}
\usepackage{caption} 
\usepackage{latexsym}
\usepackage{amssymb}
\usepackage{amsmath}
\usepackage{amsthm}
\usepackage{graphics}
\usepackage{verbatim}
\usepackage{enumerate}
\usepackage[all]{xy}
\usepackage{array,booktabs,ctable}
\usepackage[backref]{hyperref}
\usepackage{color}
\usepackage{mathtools}

\definecolor{mylinkcolor}{rgb}{0.8,0,0}
\definecolor{myurlcolor}{rgb}{0,0,0.8}
\definecolor{mycitecolor}{rgb}{0,0,0.8}
\hypersetup{colorlinks=true,urlcolor=myurlcolor,citecolor=mycitecolor,linkcolor=mylinkcolor,linktoc=page,breaklinks=true}
\usepackage{bigstrut}

\newcommand\myeq{\stackrel{\mathclap{\normalfont\mbox{\small av}}}{=}}

\newtheorem{defn}{Definition}[section]

\newtheorem{lemma}[defn]{Lemma}

\newtheorem{thm}[defn]{Theorem}
\newtheorem{theorem}[defn]{Theorem}
\newtheorem{cor}[defn]{Corollary}
\newtheorem{prop}[defn]{Proposition}

\newtheorem{conj}[defn]{Conjecture}
\newtheorem{hypa}[defn]{Hypothesis}
\newtheorem{hypb}[defn]{Hypothesis}

\theoremstyle{definition}

\newtheorem*{ack}{Acknowledgements}
\newtheorem{remark}[defn]{Remark}
\newtheorem{question}[defn]{Question}
\newtheorem{example}[defn]{Example}

\newcommand{\R}{\mathbb R}
\newcommand{\Q}{\mathbb Q}

\newcommand{\Z}{\mathbb Z}
\newcommand{\N}{\mathbb N}
\newcommand{\F}{\mathbb F}

\newcommand{\s}{\mathcal{S}}
\newcommand{\rr}{\mathcal{R}}
\newcommand{\wee}{\widetilde{\mathcal{E}}}
\newcommand{\wss}{\widetilde{\mathcal{S}}}
\newcommand{\wrr}{\widetilde{\mathcal{R}}}
\newcommand{\wtt}{\widetilde{\mathcal{T}}}
\newcommand{\wts}{\widetilde{\mathcal{TS}}}
\newcommand{\wtr}{\widetilde{\mathcal{TR}}}
\newcommand{\TT}{\mathbf{T}}
\newcommand{\MW}{\operatorname{MW}}

\newcommand{\E}{\mathcal{E}}

\newcommand{\rank}{\operatorname{rank}}
\newcommand{\rk}{\operatorname{rk}}
\newcommand{\selrank}{\operatorname{selrank}}

\newcommand{\h}{\operatorname{ht}}
\newcommand{\Sel}{\operatorname{Sel}}
\newcommand{\Rr}{\mathcal{R}_r}
\newcommand{\Var}{\operatorname{Var}}
\newcommand{\Cov}{\operatorname{Cov}}

    \DeclareFontFamily{U}{wncy}{}
    \DeclareFontShape{U}{wncy}{m}{n}{<->wncyr10}{}
    \DeclareSymbolFont{mcy}{U}{wncy}{m}{n}
    \DeclareMathSymbol{\Sh}{\mathord}{mcy}{"58}

\begin{document}



\bibliographystyle{plain}
\title[Distribution of Ranks of Elliptic Curves]{A probabilistic model for the distribution of ranks of elliptic curves over $\Q$}

\author{\'Alvaro Lozano-Robledo}
\address{Department of Mathematics, University of Connecticut, Storrs, CT 06269, USA}
\email{alvaro.lozano-robledo@uconn.edu}
\urladdr{http://alozano.clas.uconn.edu/}



\begin{abstract} 
In this article, we propose a new probabilistic model for the distribution of ranks of elliptic curves in families of fixed Selmer rank, and compare the predictions of our model with previous results, and with the databases of curves over the rationals that we have at our disposal. In addition, we document a phenomenon we refer to as {\it Selmer bias} that seems to play an important role in the data and in our models.
\end{abstract}

\maketitle

\section{Introduction}\label{sec-intro}
Let $E/\Q$ be an elliptic curve. The Mordell--Weil theorem states that the group $E(\Q)$ of rational points on $E$ is finitely generated and, therefore, we have an isomorphism
$$E(\Q)\cong E(\Q)_\text{tors} \oplus \Z^{R_E},$$
where $E(\Q)_\text{tors}$ is the (finite) subgroup of points of finite order, and $R_E = \rank (E(\Q)) \geq 0$ is the rank of the elliptic curve. The torsion subgroups that arise over $\Q$ are well understood: Mazur's theorem settles what groups are possible (\cite{mazur1}, \cite{mazur2}), the parametrization of the corresponding modular curves are known (\cite{kubert}), and we know the distribution of elliptic curves with a prescribed torsion subgroup (\cite{harron}) as a function of the height of the curve. However, the  distribution of ranks of elliptic curves is unknown. Several conjectures can be found in the literature (e.g., on the average rank, see \cite{poonen}), and also some heuristic models (\cite{watkins}, \cite{ppvm}), but the basic questions about the distribution of the ranks remain unanswered. For instance, it is not known whether the rank can be arbitrarily large (currently, the largest rank known is $28$, due to Noam Elkies - see \cite{dujella} for Elkies' example, and other current records).

In this article, we propose a new probabilistic model for the distribution of ranks of elliptic curves (in families of fixed $2$-Selmer rank) and explore its possible consequences. The model itself is built on a probability space  of {\it test elliptic curves} and {\it test Selmer elements} in the spirit of Cram\'er's model for the prime numbers (see \cite{cramer}, \cite{granville}). As such, our model is a collection $\mathbf T$ of all possible sequences of (finite) sets of test elliptic curves of each height (with certain growth conditions as the height grows). The sequence of ordinary elliptic curves $\mathcal{E}$ over $\Q$ belongs to this class, and we make predictions about $\mathcal{E}$ from the asymptotic average behavior from sequences in $\mathbf T$ under the assumption of certain probabilistic hypotheses (see Sections \ref{sec-intromodel}, \ref{sec-cramermodel}, and \ref{sec-cramermodel2} for more details). We use the largest database of elliptic curves at our disposal (\cite{BHKSSW}, which we will refer to as the BHKSSW database) in order to test our model and to make predictions. We concentrate on elliptic curves over $\Q$ because there are no analogous databases for any other number field $K$ or function field $\mathbb{F}(T)$ to test the model, but the same ideas would apply more generally for $p$-Selmer groups of abelian varieties over $K$ or $\mathbb{F}(T)$, with suitable modifications of the probability functions $\Theta_n(X)$ and $\rho_n(X)$ defined below in Hypothesis $C$. 

\subsection{Notation and setup for elliptic curves over $\Q$} For an elliptic curve $E/\Q$ we define the ($2$-)Selmer rank of an elliptic curve by  $\selrank(E(\Q))=\dim_{\F_2}\Sel_2(E/\Q)-\dim_{\F_2}(E(\Q)[2])$, where $\Sel_2(E/\Q)$ is the $2$-Selmer group attached to $E/\Q$. 
\begin{itemize}
\item For fixed $n,r\geq 0$, and for any $1\leq X_1\leq X_2$, let $\mathcal{E}([X_1,X_2])$, $\s_n([X_1,X_2])$, and $\Rr([X_1,X_2])$ be, respectively, the sets of all elliptic curves, all curves with Selmer rank $n$, and curves of rank $r$, with naive height in the interval $[X_1,X_2]$. 
\item We will denote the set of elliptic curves of height exactly $X$ by  $\mathcal{E}^X=\mathcal{E}([X,X])$, and we will write $\mathcal{E}(X)=\mathcal{E}([1,X])$ for the set of all elliptic curves up to height $X$. We define similarly $\s_n^X$, $\s_n(X)$, and $\Rr^X$, $\Rr(X)$, for each $n,r\geq 0$.
\item If $\mathcal{C}\subseteq \mathcal{E}$ is a set of elliptic curves (say $\mathcal{C}=\mathcal{E}$, $\mathcal{S}_n$, $\Rr$, or  $\Rr\cap \s_n$), then we write $\pi_{\mathcal{C}}(X)$ for $\# \mathcal{C}([1,X])$, i.e., $\pi_{\mathcal{C}}$ is the counting function of elliptic curves in $\mathcal{C}$ up to height $X$.
\end{itemize} 
For a fixed rank $r\geq 0$ and a height $X$ such that the set  $\mathcal{E}^X$ is non-empty, we are interested in the probability that an elliptic curve of height $X$ belongs to $\Rr^X$, that is, $\operatorname{Prob}(E\in\Rr^X)=\# \Rr^X/\# \mathcal{E}^X$. Our model is aimed at giving meaning and estimating $\operatorname{Prob}(E\in\Rr^X)$ via the probability formula:
$$\operatorname{Prob}(E\in\Rr^X)=\sum_{j\geq 0} \operatorname{Prob}(E\in \s_{r+2j}^X)\cdot \operatorname{Prob}(E\in\Rr^X\ | \ E\in \s_{r+2j}^X),$$
where $\operatorname{Prob}(E\in \s_{r+2j}^X)=\# \s_{r+2j}^X/\# \mathcal{E}^X$, and we define the conditional probability $\operatorname{Prob}(E\in\Rr^X\ | \ E\in \s_{r+2j}^X)$ as $0$ if $\s_{r+2j}^X$ is empty, and by $\# \Rr^X\cap \s_{r+2j}^X/\# \s_{r+2j}^X$ otherwise. 
In Section \ref{sec-allcurves} we will discuss the known results about the number of elliptic curves up to height $X$. 

\subsection{Notation and setup for test elliptic curves} A {\it test elliptic curve} is a triple $E=(X,n,\Sel_2)$ consisting of:
\begin{itemize}
	\item a positive integer $X\geq 1$, the height of $E$, also denoted $X=\h(E)$,
	\item a non-negative integer $n$, the Selmer rank of $E$, also denoted $n=\selrank(E)$, and 
	\item a vector $\Sel_2(E)=(s_{E,1},s_{E,2},\ldots,s_{E,\lfloor n/2 \rfloor})$ of $\lfloor n/2 \rfloor$ {\it test Selmer elements}. Each Selmer element is a symbol, which is either a MW, or a $\Sh$ symbol. 
\end{itemize}
The set of all test elliptic curves will be denoted by $\widetilde{\E}$, those test curves with height $X$ will be $\widetilde{\E}^X$, and those test curves with height $X$ and Selmer rank $n$ will be denoted by $\widetilde{\s}_n^X$. We define $\widetilde{\mathcal{R}}_r^X$ similarly. We let $\mathbf{T}$ be a space of sequences of (finite) subsets of $\wee^X$ with certain growth conditions, defined as follows: 
	$$\mathbf{T} = \left\{ \left(\wtt^X\right)_{X\geq 1}  : \wtt^X \subseteq \wee^X, \ \sum_{N=1}^X \# \wtt^N = \kappa X^{5/6}+O(X^{1/2})\right\},$$
	where $\kappa = 2^{4/3}\cdot(\zeta(10)\cdot 3^{3/2})^{-1}$. To each elliptic curve $E$ we can associate a test elliptic curve (Remark \ref{rem-testellipticcurve}) and the sequence $\mathcal{E}=(\mathcal{E}^X)$ of ordinary elliptic curves belongs to $\mathbf T$. Thus, the goal is to predict the behaviour of $\E$ from the average asymptotic behaviour of sequences in $\wtt$.

	We also need to introduce counting notation for test elliptic curves: if $I$ is a finite interval in $[1,\infty)$, and $\wtt \in \mathbf{T}$, we will write $\wtt(I)= \bigcup_{X\in I} \wtt^X$, and $\wts_n(I) = \bigcup_{X\in I} \wtt^X \cap \wss_n^X$. Finally, we define 
	$$\pi_{\wtt}(I) = \#\wtt(I)=\sum_{X\in I} \# \wtt^X\ \text{ and }\ \pi_{\wts_n}(I) = \# \wts_n(I)= \sum_{X\in I} \# (\wtt^X\cap \wss_n^X).$$

\subsection{Probability spaces}\label{sec-intromodel} In Sections \ref{sec-cramermodel} and \ref{sec-cramermodel2}, and for fixed $n\geq 0$ and $X\geq 1$, we state two probabilistic hypotheses, $H_A$ and $H_B$ stated formally in Hypotheses \ref{hypa} and \ref{hypb}, respectively. These hypotheses make $\widetilde{\mathcal{E}}^X$ and $\widetilde{\s}_n^X$ into probability spaces:
\begin{enumerate} 
\item[($H_A$)] {\bf Hypothesis A}: Informally, the probability of drawing a test elliptic curve of Selmer rank $n$ out of the bin $\widetilde{\E}^X$ is given by a function $\theta_n(X)$. Formally, the function  $Y_{\Sel,n,X}\colon\widetilde{\mathcal{E}}^X\to \{0,1\}$ such that  $Y_{\Sel,n,X}(E)=1$ if $E\in\widetilde{\s}_n^X$, and $Y_{\Sel,n,X}(E)=0$ otherwise, is a random variable with Bernoulli distribution $B(1,\theta_n(X))$, where $\theta_n(X)$ is a function that depends on $n$ and $X$. In particular, this implies that the expected value $\mathbb{E}(Y_{\Sel,n,X})$ is $\operatorname{Prob}(E\in \widetilde{\s}_{n}^X)=\theta_n(X)$. 
\item[($H_B$)] {\bf Hypothesis B}: Let $E\in \wss_n^X$ be chosen at random. Informally, the probability that the $i$-th coordinate of $\Sel_2(E)=(s_{E,1},\ldots,s_{E,\lfloor n/2 \rfloor})$ is a MW element is given by a function $\rho_n(X)$ (that does not depend on $i$ or $E$). Formally, for each $1\leq i \leq \lfloor n/2\rfloor$, the function  $Y_{i}\colon\wss_n^X \to \{0,1\}$ that takes the value $1$ whenever $s_{E,i}$ is a MW element, and $0$ otherwise, is a random variable with Bernoulli distribution $B(1,\rho_n(X))$, where $\rho_n(X)$ is a function that depends on $n$ and $X$, but not on $i$ (however, the variables $Y_i$ are not independent in general). From the distribution of the variables $Y_i$  we shall recover the conditional probability $\operatorname{Prob}(E\in\widetilde{\mathcal{R}}_r^X\ | \ E\in \widetilde{\s}_{n}^X)$  for any $0\leq r\leq n$ with $n\equiv r \bmod 2$ (see Corollary \ref{cor-hasse1}).
\end{enumerate} 
After taking all the available data under consideration (mainly  \cite{BHKSSW}), we formulate a refinement of the model which specifies the shape of $\theta_n(X)$ and $\rho_n(X)$ up to some constants (which are Hypotheses \ref{conj-selmerratio} and \ref{conj-hasseratio}):

\begin{enumerate}
	\item[($H_C$)] {\bf Hypothesis C}: Assume $H_A$ and $H_B$. Then, there are constants $C_n$, $D_n$, $e_n$, $f_n$, for each $n\geq 1$, such that 
	$$\theta_n(X) = \frac{s_n}{1+C_nX^{-e_n}},\ \text{ and }\ \rho_n(X) = \frac{D_n}{X^{f_n}},$$
	where the limit values $s_n$ of $\theta_n(X)$ are those given by a conjecture of Poonen and Rains, and all constants are positive except $C_1<0$.
\end{enumerate}

The data suggest that for the family of all elliptic curves over $\Q$ the values of the constants of Hypothesis C, for $n=1,\ldots,5$, are as given in Tables \ref{tab3} and  \ref{tab-rhomodel}, and the limit values $s_n$ are discussed in Section \ref{sec-selmer-ratio} (as in   \cite{poonen}). We have also investigated the suitability of the model in the subfamily of curves with $j=1728$ (see Remark \ref{rem-j1728}).

\subsection{Summary of results} In our results, we give the expected value and asymptotic behavior of a random sequence $\wtt$ in $\mathbf{T}$ under the probabilistic hypotheses $A$, $B$, and $C$. In our main Theorems \ref{thm-predictrank} and \ref{thm-averank}, under the assumption of $H_A$ and $H_B$, we provide formulas for $\pi_{\wtr_r\cap \widetilde{\s}_n}(X)$, i.e., the number of test elliptic curves in $\wtt$ of rank $r$ and Selmer rank $n$ up to height $X$, and also for the contribution to the average rank coming from test elliptic curves of Selmer rank $n$. Note that our results are stated ``on average'' (denoted by $\myeq$, a concept that we define precisely in Definition \ref{defn-inave}).
\begin{theorem}[also Theorem \ref{thm-predictrank}]\label{thm-predictrankintro}
Let $\wtt\in\mathbf{T}$ be arbitrary, and let $X,r\geq 0,j\geq 0$ be fixed, such that $n(j)=r+2j\geq 1$. If we assume Hypothesis C, then  the expected value of $\pi_{\wtr_r\cap \widetilde{\s}_{n(j)}}(X)$ is given on average by
\begin{eqnarray*} \mathbb{E}\left( \pi_{\wtr_r\cap \widetilde{\s}_{n(j)}}(X)\right) &\myeq  & 
 \frac{5\kappa}{6} \binom{\lfloor \frac{r}{2} \rfloor+j}{j} \int_{0}^X \cfrac{\theta_{n(j)}(H)}{H^{1/6}}\cdot \mathbb{E}_{\lfloor \frac{r}{2} \rfloor,j}^{n(j)}(H)\, dH + O(X^{1/2})\\ 
 &\myeq & \frac{5\kappa}{6}  \binom{\lfloor \frac{r}{2} \rfloor+j}{j} \int_{0}^X \cfrac{ s_{n(j)}\cdot \mathbb{E}_{\lfloor \frac{r}{2} \rfloor,j}^{n(j)}(H)}{(1+C_{n(j)}H^{-e_{n(j)}})\cdot H^{1/6}}\, dH + O(X^{1/2}),
 \end{eqnarray*} 
where $\kappa = 2^{4/3}\cdot(\zeta(10)\cdot 3^{3/2})^{-1}$, and $\mathbb{E}_{\lfloor \frac{r}{2} \rfloor,j}^{n(j)}(H)$ is the expected value defined in Remark \ref{rem-notation}.
\end{theorem}
In Corollary \ref{cor-predictrank2} we specialize the formulas of $\pi_{\wtr_r\cap \widetilde{\s}_{n(j)}}(X)$ for $0\leq r\leq n \leq 5$ (see also Table \ref{tab-predictrank}). Using our formulas, we have computed approximations of $\pi_{\Rr}(X)$ for $1\leq r\leq 5$ in the range $[0,2.7\cdot 10^{10}]$, and plotted them in Figures \ref{fig-predictrank123} and \ref{fig-predictrank45}. The error in our approximations is less than $0.7\%$ in this range, which is within the order of magnitude of the error predicted by the model (see Table \ref{tab-predictrank2}).

Our second theorem gives  formulas for the contribution to the average rank of test elliptic curves coming from curves of each Selmer rank $n\geq 1$. Then, the contributions are added up to estimate the behavior of the average rank.
\begin{thm}[also Theorem \ref{thm-averank} and Corollary \ref{cor-averank}]\label{thm-averankintro}
Let $\wtt\in\mathbf{T}$ be arbitrary. Assume $H_A$ and $H_B$, and let $n\geq 1$ be fixed. Then, the expected value of $\operatorname{AvgRank}_{\wts_n}(X)=\cfrac{1}{\pi_{\wtt}(X)}\cdot \sum_{E\in \wts_n(X)} \rank(E)$
is given on average by
\begin{eqnarray*} \frac{5\kappa}{6\pi_{\wtt}(X)} \cdot\int_1^X  \frac{\theta_n(H)}{H^{1/6}} \left((n\bmod 2) + 2\left\lfloor{\frac{n}{2}}\right\rfloor \rho_n(H)\right)\, dH + O(X^{-1/3}),
\end{eqnarray*}
where the implied error in the approximation is bounded by $C\cdot \theta_{n}^\text{sup}/X^{1/3}$, for some constant $C$ that does not depend on $n$. Moreover, the error in approximating $\operatorname{AvgRank}_{\wts_n}(X)$ by its expected value is given, on average, by 
$$\sqrt{\frac{5\kappa \lfloor n/2 \rfloor}{6\pi_{\wtt}(X)^2} \int_1^X \frac{\theta_n(H)}{H^{1/6}} (\rho_n(H)(1-\rho_n(H)) + (\lfloor n/2 \rfloor -1) C_{1,1}^n(H))\, dH + O(X^{-7/6})},$$
where $C_{1,1}^n(X)$ is the covariance function defined in Proposition \ref{prop-equicorr}. Further, there are constants $\tau_n$ such that the expected value of $\operatorname{AvgRank}_{\wtt}(X)=\sum_{n=1}^\infty 
\operatorname{AvgRank}_{\wts_n}(X)$ is given on average by 
$$
 \sum_{n=1}^\infty s_n \cdot  \left(\frac{\tau_n}{X^{5/6}} +  \sum_{m=0}^\infty  \left(\frac{(n\bmod 2)(-C_n)^m}{1-(6/5)me_n} +  X^{-f_n} \frac{2\left\lfloor{\frac{n}{2}}\right\rfloor D_n (-C_n)^m}{1-(6/5)(f_n+me_n)}\right)X^{-me_n}\right) + O(X^{-1/3}).$$
In particular,
$$\lim_{X\to \infty} \operatorname{AvgRank}_{\wtt}(X) \myeq \sum_{k=0}^\infty s_{2k+1}=\frac{1}{2},$$
in the sense that the expected value goes on average to $1/2$ and the standard error goes to $0$ on average as $X\to \infty$.
\end{thm}
In particular, Theorem \ref{thm-averankintro} says that our assumptions imply the so-called ``$50\%-50\%$ conjecture'' (see Conjecture \ref{conj-5050}) and, moreover, it predicts not only the $1/2$ limit of the average rank, but also a rate of convergence to said limit. We have used our formulas to compute  an approximation of $\operatorname{AvgRank}_{\mathcal{E}}(X)$ by approximating $\sum_{n=1}^5 
\operatorname{AvgRank}_{\s_n}(X)$ in the range $[0,2.7\cdot 10^{10}]$ and plotted it in Figure \ref{fig-averank}. The error in our approximation of $\operatorname{AvgRank}_{\mathcal{E}}(2.7\cdot 10^{10})$ is $0.0523\%$ of the actual value, which is again in agreement with the error predicted by the model (see Remark \ref{rem-averank}). In Table \ref{tab-averank}, and under the assumption of Hypothesis C, we have computed approximate values of $\operatorname{AvgRank}_{\wtt}(X)$ by approximatin $\sum_{n=1}^5 
\operatorname{AvgRank}_{\wts_n}(X)$ using numerical integration of the formulas of Theorem \ref{thm-averankintro}.

Our Hypothesis A also implies a formula for the average $2$-Selmer rank of a test elliptic curve.

\begin{thm}[Also Prop. \ref{prop-aveselrank}]\label{thm-aveselrank}
Let $\wtt\in\mathbf{T}$ be arbitrary.  Let $\operatorname{AvgSelRank}(X)$ be defined by
$$\operatorname{AvgSelRank}(X) = \frac{1}{\pi_{\wtt}(X)}\sum_{E\in \wtt(X)} \selrank(E).$$
If we assume $H_A$ and we assume that $0\leq \theta_n(X)\leq s_n$ for all $n\geq 2$ and all $X>0$, then the expected value of the average Selmer rank is given by
$$\mathbb{E}(\operatorname{AvgSelRank}(X)) \myeq \frac{5/6}{X^{5/6}}\int_1^X \frac{\sum_{n\geq 1}n\cdot \theta_n(H)}{H^{1/6}}\, dH + O\left(X^{-1/3}\right).$$
Moreover, $\lim_{X\to\infty} \mathbb{E}(\operatorname{AvgSelRank}(X)) = \sum_{n\geq 1} n \cdot s_n =  1.26449978\ldots$. 
\end{thm}

\begin{table}[h!]
\centering
\def\arraystretch{2}
\begin{tabular}{c|c||c|c}
$X$ & $\sum_{n=1}^5\operatorname{AvgRank}_{\wts_n}(X)$ & $X$ & $\sum_{n=1}^5\operatorname{AvgRank}_{\wts_n}(X)$ \\
\hline
$10^{10}$ & $0.905665$ & $10^{50}$ & $0.548880$\\
$10^{15}$ & $0.846828$ & $10^{75}$ & $0.512531$\\
$10^{20}$ & $0.766868$ & $10^{100}$ & $0.503256$\\
$10^{30}$ & $0.649901$ & $10^{150}$ & $0.500215$\\
$10^{40}$ & $0.585108$ & $10^{200}$ & $0.500006$\\
\end{tabular}
\caption{Approximate values of $\sum_{n=1}^5\operatorname{AvgRank}_{\wts_n}(X)$ obtained using numerical integration of the formulas of Theorem \ref{thm-averank}. The integration was done with SageMath, which reports an absolute error in the numerical integration less than $4\cdot 10^{-7}$ in all cases. By Theorem \ref{thm-averank}, the limit should be $s_1+s_3+s_5=0.49999965\ldots $.}
\label{tab-averank}
\end{table}

 Finally,  a question on {\it Selmer rank bias} arises in our work:
\begin{question}
Does the expected value of the random variables  $Y_i$ of Hypothesis B depend on $n$? In other words, does the probability that $s_E\in \Sel_2(E/\Q)$ is globally solvable depend on $n=\selrank(E(\Q))$?
\end{question}
The answer, surprisingly, seems to be that the probability does depend not only on the parity of $n$, but also on the value of $n$ itself (see Fig. \ref{fig-rhonx}). For instance, the data suggest that an element of $\Sel_2(E/\Q)$ is significantly more likely to be globally solvable for $n=5$ than for $n=3$. However, the probabilities for $n=2$ and $n=4$ are quite similar in the height interval $[0,2.7\cdot 10^{10}]$ (but they do not behave identically).

\begin{remark}
In this article we work with elliptic curves over $\Q$ and $2$-Selmer groups because the database we have to test our models (\cite{BHKSSW}) only contains $2$-Selmer information. However, the same probabilistic model could be derived for $p$-Selmer groups over a global field $K$.
\end{remark}

\begin{remark}
	The accuracy and the validity of the model to predict the distribution of ranks of elliptic curves is verified in two ways: (1) the probabilistic nature of the model allows for error formulas to be derived (see for instance Theorem \ref{thm-averankintro} or Corollary \ref{cor-predictrank}), and we compare the predictions of the model against the theoretical errors in several examples (see Tables \ref{tab-errors}, \ref{tab-errors2}, \ref{tab-predictrank2}), and (2) in order to give approximate values of the constants that appear in Hypothesis C, we have used the data of the BHKSSW database up to height $2.7\cdot 10^{10}$, but Balakrishnan et al. have also computed large height sample sets of elliptic curves, that we use to test our model at larger heights (see, for instance, Remark \ref{rem-sellargeheight} and  Table \ref{tab-errorslargeheight}). 
	
	In order to be able to improve the accuracy of our model, and extend our model to higher ranks ($r>5$), we would need a new massive amount of data in the form of a much larger database of elliptic curves which, at this time, is unavailable and far from the realm of our computational reach.
\end{remark}

\subsection{Structure of the article} 
The structure of the paper is as follows: in Section \ref{sec-notation} we settle the notation for the rest of the paper (including a summary Table \ref{tab-notation} of symbols) and discuss some basic probability notions. In Section \ref{sec-allcurves} we expand on a result of Brumer to estimate the number of elliptic curves up to a given height. In Section \ref{sec-Selmergroups} we review some basics about Selmer groups, and in Section \ref{sec-cramermodel} we begin the construction of the Cr\'amer-like random model by setting up the part of the probability space that models the Selmer rank of a test elliptic curve. In Section \ref{sec-selmer-ratio} we prove several consequences of the probabilistic model defined in Section \ref{sec-cramermodel}, and in particular count the number of test elliptic curves of each Selmer rank up to a certain height bound. In Section \ref{sec-cramermodel2} we continue the construction of the random model, now concentrating on the pieces of the model that will contribute to the Mordell--Weil and Tate--Shafarevich groups. In Section \ref{sec-hasse-ratio} we show more consequences of the model, and find formulas for the average rank of test elliptic curves up to a certain bound. Finally, in Sections \ref{sec-predictrank} and \ref{sec-predictave} we put everything together to give predictions on the number of elliptic curves of each rank, and the average rank.

\begin{ack}
The author would like to thank Jennifer Balakrishnan, Iddo Ben-Ari, Keith Conrad, Harris Daniels, Wei Ho, Jennifer Park, Ari Shnidman, Drew Sutherland, and John Voight for their helpful conversations, comments, and suggestions. The author would express his gratitude to the referees for a very meticulous reading of earlier drafts of the paper, and providing many detailed comments and suggestions to improve the article.
\end{ack}

\section{Notation and Probability}\label{sec-notation} 
\begin{table}[h!]
\centering
\def\arraystretch{1.5}
\begin{tabular}{lll}
 $\mathcal{E}$ & Set of elliptic curves over $\Q$ up to isomorphism & \S \ref{sec-allcurves}\\
 $\selrank(E(\Q))$ & $2$-Selmer rank, equal to $\dim_{\F_2}\Sel_2(E/\Q)-\dim_{\F_2}(E[2])$ & \S \ref{sec-intro}, \ref{sec-selmer-ratio}\\
 $\mathcal{S}_n$ & For $n\geq 0$, curves $E\in \mathcal{E}$ with $\selrank(E(\Q))=n$ & \S \ref{sec-selmer-ratio}\\
 $\mathcal{R}_r$ & For $r\geq 0$, curves $E\in \mathcal{E}$ with $\rank(E(\Q))=r$ & \S \ref{sec-predictrank}\\
 $\widetilde{\E}$, $\widetilde{\s}_n$, $\widetilde{\rr}_r$ & Test elliptic curves ($\widetilde{\E}$), of Selmer rank $n$ ($\widetilde{\s}_n$), of MW rank $r$ ($\widetilde{\rr}_r$) & \S \ref{sec-cramermodel}\\
$\h(E)$ & The naive height of an elliptic curve & \S \ref{sec-allcurves}\\
$\mathcal{C}(X)$ & For $X\geq 0$, curves in  $C=\mathcal{E}$, $\mathcal{R}_r$, or $\mathcal{S}_n$, with (naive) height $\leq X$ &\S \ref{sec-allcurves}, \ref{sec-selmer-ratio}, \ref{sec-predictrank} \\
$\mathcal{C}(I)$ & For an interval $I$, curves in $\mathcal{C}$ with height in $I$ &\S \ref{sec-allcurves}, \ref{sec-selmer-ratio}, \ref{sec-predictrank} \\
$\mathcal{C}^X$ &  For $X\geq 0$, curves in $\mathcal{C}$ with height exactly $X$& \S \ref{sec-allcurves}, \ref{sec-selmer-ratio}, \ref{sec-predictrank}\\ 
 $\mathbf{T}$ & Space of sequences of (finite) subsets of $\wee^X$, for each $X\geq 1$ & \S \ref{sec-intromodel}, \ref{sec-cramermodel}\\
 $\wtt$ & A sequence of (finite) subsets of $\wee^X$ for each $X\geq 1$, i.e., an element of $\mathbf T$ & \S \ref{sec-intromodel}, \ref{sec-cramermodel}\\
$\pi_\mathcal{C}(X)$ & For a set $\mathcal{C}\subseteq \mathcal{E}$, the size of $\mathcal{C}\cap \mathcal{E}(X)$, where $\mathcal{C}=\mathcal{E}$, $\mathcal{R}_r$, or $\mathcal{S}_n$ & \S  \ref{sec-allcurves}, \ref{sec-selmer-ratio}, \ref{sec-predictrank}\\
$\pi_\mathcal{C}(I)$ & For a set $\mathcal{C}\subseteq \mathcal{E}$ and an interval $I$, the size of $\mathcal{C}\cap \mathcal{E}(I)$ & \S \ref{sec-allcurves}, \ref{sec-selmer-ratio}, \ref{sec-hasse-ratio}\\
$\kappa$ & Constant equal to $2^{4/3}\cdot(\zeta(10)\cdot 3^{3/2})^{-1}= 0.484462004349\ldots$ & Thm. \ref{thm-brumer}\\
$s_n$ & $\lim_{X\to \infty} \pi_{\s_n}(X)/\pi_{\mathcal{E}}(X)$, given by a conjectural formula by \cite{poonen} & \S \ref{sec-selmer-ratio}\\
$B(m,p)$ & Binomial distribution with $m$ experiments and probability $p$ & \S \ref{sec-selmer-ratio}\\
$Y_{\Sel,n,X}(E/\Q)$ & Random variable with value $1$ if $\selrank(E(\Q))=n$, and $0$ otherwise & Hyp. \ref{hypa}\\
$\theta_n(X)$ & The function giving the expected value of $Y_{\Sel,n,X}(E/\Q)$& Hyp. \ref{hypa}\\
$\theta_n(X,N)$ & Moving ratio defined by $\pi_{\wts_n}((X,X+N])/\pi_{\wtt}((X,X+N])$ & Cor. \ref{cor-travelratio}\\
$Y_{\operatorname{Hasse},n,X}(s_E)$ & Random variable with value $1$  if $s_E\equiv 0\in \Sh(E/\Q)$, and $0$ otherwise & Hyp. \ref{hypb}\\
$\rho_n(X)$ & The function giving the expected value of $Y_{\operatorname{Hasse},n,X}(s_E)$ & Hyp. \ref{hypb}\\
$\rho_n(X,N)$ & Moving ratio approximating $\rho_n(X)$ & Def. \ref{defn-rhonx}\\
$C_{s,t}^n(X)$ & Covariance function of a certain products of random variables & Prop. \ref{prop-equicorr}\\
$\mathbb{E}_{s,t}^n(X)$ & Expected value of a certain product of random variables & Rem. \ref{rem-notation}\\
\end{tabular}
\caption{Notation defined and used throughout the paper.}
\label{tab-notation} 
\end{table} 

In Table \ref{tab-notation} we include a glossary of notation defined throughout the paper, together with a reference. We also recall here a few definitions of probability concepts for the convenience of the reader. We say that a random variable $Y$ follows a Bernoulli distribution $B(1,p)$, or $Y\sim B(1,p)$, if $Y$ takes the value $1$ with success probability of $p$ and the value $0$ with probability $1-p$. The binomial distribution $B(n,p)$ is the discrete probability distribution of the number of successes in a sequence of $n$ independent yes/no experiments, each of which yields success with probability $p$. The expected value and variance of a discrete random variable $Y$ that takes values $y_1,\ldots,y_k$ with probability $p_1,\ldots,p_k$ are defined respectively by
$$\mathbb{E}(Y)=\sum_{i=1}^k y_i\cdot p_i, \quad \Var(Y)=\mathbb{E}(Y^2) - (\mathbb{E}(Y))^2.$$
The covariance of two random variables $V,W$ is given by 
$$\Cov(V,W)=\mathbb{E}(VW)-\mathbb{E}(V)\cdot \mathbb{E}(W).$$
If $\Cov(V,W)=0$, then we say that $V$ and $W$ are uncorrelated random variables. If $V$ and $W$ are independent random variables, then $\mathbb{E}(VW)=\mathbb{E}(V)\mathbb{E}(W)$ and, in particular, $\Cov(V,W)=0$. Also, we note here that if $a$ and $b$ are constants, then
$$\Var(aV+bW) = a^2\Var(V)+b^2\Var(W)+2ab\Cov(V,W).$$
Finally, the standard error of the mean (SEM) of random variables $Y_1,\ldots,Y_m$ is an estimator for the accuracy of the approximation of $\frac{1}{m}\sum Y_i$ by $\frac{1}{m}\sum \mathbb{E}(Y_i)$, and it is defined as the square root of the variance of the mean of the variables. In other words, the standard error is given by 
$$ \sqrt{\Var\left(\frac{1}{m}\sum_{i=1}^m Y_i\right)}.$$
If $Y_1,\ldots,Y_m$ are $m$ independent random variables following the same distribution with mean $\mu$ and standard deviation $\sigma$, then $\operatorname{SEM}(Y_1,\ldots,Y_m)=(\frac{1}{m^2}\sum \Var(Y_i))^{1/2}=(\Var(Y_1)/m)^{1/2}=\sigma/\sqrt{m}$. 

\section{The number of elliptic curves with (naive) height $\leq X$}\label{sec-allcurves}

Let $E/\Q$ be an elliptic curve. We shall write each elliptic curve in a short Weierstrass model of the form $y^2=x^3+Ax+B$ with $A,B\in\Z$ and $0\neq 4A^3+27B^2$ such that $\Delta_E$ is minimal in absolute value (minimal among all  short Weierstrass models isomorphic to $E$ over $\Q$). In other words, we will be working with the set of elliptic curves 
$$\mathcal{E} = \{E_{A,B} : y^2=x^3+Ax+B\ |\ A,B\in\Z, 4A^3+27B^2\neq 0, \text{ and if } d^4|A,\ d^6|B, \text{ then } d=\pm 1\}.$$ 
 Then, the (naive) height of $E=E_{A,B}\in\mathcal{E}$ is defined by
$$\h(E_{A,B}) = \max\{ 4|A|^3,27B^2\},$$
as used in \cite{BHKSSW},  \cite{brumer}, and \cite{ppvm}. The BHKSSW database (\cite{BHKSSW}) contains data for all $\num[group-separator={,}]{238764310}$ elliptic curves up to height $\num[group-separator={,}]{26998673868}\approx 2.7\cdot 10^{10}$. While working on this project, we have gathered data for the curves $y^2=x^3+Ax$, for all fourth-power-free integers $A\in [1,10^6]$, that is, about a million curves with $j=1728$, up to height $4\cdot 10^{18}$.

For each positive real number $X$, we define $\E(X) = \{ E\in\mathcal{E}: \h(E)\leq X\}$, and $\pi_{\mathcal{E}}(X)=\#\E(X)$. Similarly, if $0\leq X_1\leq X_2$, we shall write $\E([X_1,X_2])$ for the set $\{E\in \mathcal{E} :X_1\leq \h(E)\leq X_2 \}$ and $\pi_{\mathcal{E}}([X_1,X_2])=\#\E([X_1,X_2])$ for its size (in particular, $\mathcal{E}^X=\E([X,X])$ denotes the elliptic curves of height {\it exactly} $X$, a set that can be empty depending on the value of $X$). We cite a result of Brumer (\cite{brumer}) that estimates the value of $\pi_{\mathcal{E}}(X)$ to our choice of height function.
\begin{theorem}[{\cite[Lemma 4.3]{brumer}}]\label{thm-brumer}  The number of elliptic curves of height up to $X$ satisfies  $\displaystyle \pi_{\mathcal{E}}(X) = \kappa X^{5/6} +  O(X^{1/2})$  
where the constant $\kappa = 2^{4/3}\cdot(\zeta(10)\cdot 3^{3/2})^{-1}= 0.484462004349\ldots$. 
\end{theorem}

\begin{remark} Using the BHKSSW database, we have calculated the values of $\pi_{\mathcal{E}}(X)$ up to $2.7\cdot 10^{10}$ in $0.25\cdot 10^9$ intervals. We have  found (using SageMath, \cite{sage}) the best-fit model of the form $C\cdot X^{5/6}$ for these data points, and found that the best constant is $C= 0.48447036\ldots$ in agreement with Brumer's constant ($C$ and $\kappa$ differ by $8.35\cdot 10^{-6}$).

\begin{center}
\begin{figure}[h!]
\includegraphics[width=6.6in]{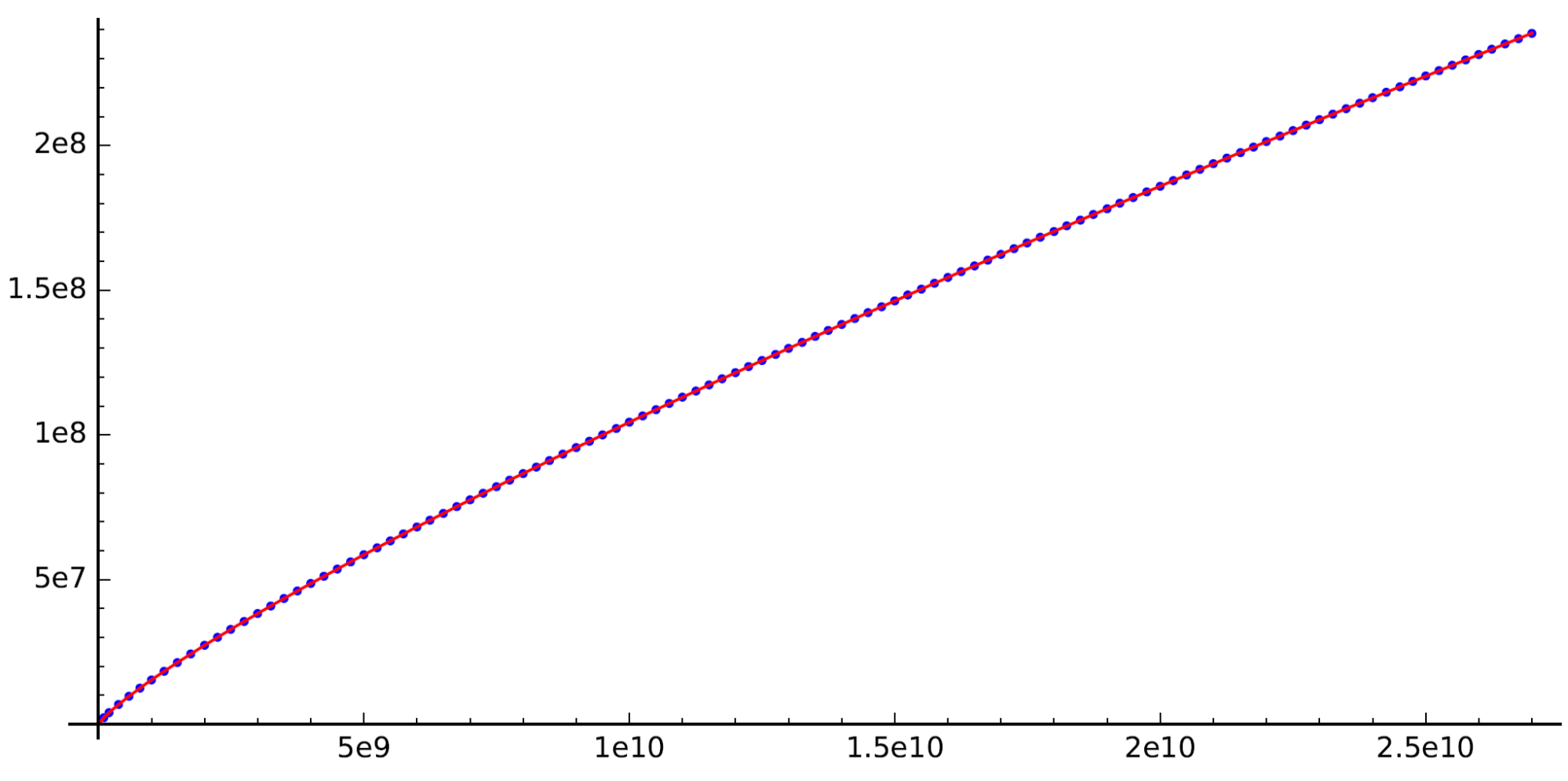}
\caption{Values of $\pi_{\mathcal{E}}(X)$ from the BHKSSW database (blue dots), and the function $0.48447036\cdot X^{5/6}$ (in red).}
\end{figure}
\end{center}
\end{remark}

\begin{remark}\label{rem-numbercurves}
According to Theorem \ref{thm-brumer}, the number of curves in the height interval $(X,X+N]$ is, approximately, 
\begin{eqnarray*}
\pi_{\mathcal{E}}((X,X+N]) &=& \pi_{\mathcal{E}}(X+N)-\pi_{\mathcal{E}}(X)\\
& \approx & \kappa\cdot ((X+N)^{5/6} - X^{5/6}) \\
 &=& \frac{5\kappa}{6}\int_{X}^{X+N} \frac{1}{H^{1/6}}\, dH \approx \frac{5\kappa}{6}\cdot \frac{N}{X^{1/6}},
\end{eqnarray*}
where $5\kappa/6= 0.403718336957\ldots$, and the last approximation is valid for large $X$ such that $X\gg N\geq 0$. However, the error in this approximation is still of the order $O(X^{1/2})$, so the error can be quite large, and it can oscillate from positive to negative. For instance, 
\begin{eqnarray*}
\pi([2\cdot 10^{10},2\cdot 10^{10}+0.25\cdot 10^9]) &=& \num[group-separator={,}]{1955593} \ \text{ while }\ \num[group-separator={,}]{1937225.394}\ldots = \frac{5\kappa}{6}\cdot \frac{0.25\cdot 10^9+1}{(2\cdot 10^{10})^{1/6}},\\
\pi([2.5\cdot 10^{10},2.5\cdot 10^{10}+0.25\cdot 10^9]) &=& \num[group-separator={,}]{1852352} \ \text{ while }\  \num[group-separator={,}]{1866502.107}\ldots = \frac{5\kappa}{6}\cdot \frac{0.25\cdot 10^9+1}{(2.5\cdot 10^{10})^{1/6}}.
\end{eqnarray*}
Nonetheless, we shall prove below (Corollary \ref{cor-countcurvesaverage}) that the approximation of $\pi_{\mathcal{E}}((X,X+N])$ by $\frac{5\kappa}{6}\cdot \frac{N}{X^{1/6}}$ works {\it on average} with error going to zero as $X$ goes to infinity, but before we do so, we will write a formal definition of what we mean by ``on average'' (see Definition \ref{defn-inave}). We also point out here that if we want $\pi_{\mathcal{E}}((X,X+N])$ to be approximately constant as $X\to \infty$, then we need $N=N(X) \asymp C\cdot X^{1/6}$. For instance, if we want $\pi_{\mathcal{E}}((X,X+N(X)])\approx 10^t$, then we should have $N=N(X)=(6\cdot 10^t/5\kappa)\cdot X^{1/6}$, where $6/5\kappa = 2.476974436029\ldots$.
\end{remark}

\begin{defn}\label{defn-inave}
	Let $f,g\colon \N \to \R$ and $h\colon \R \to \R$ be functions. We say that $f(X)\myeq g(X) +O(h(X))$ if the following condition is satisfied:
	$$\frac{1}{X} \sum_{k=1}^{\lfloor X \rfloor} (f(X)-g(X)) = O(h(X))$$
	where $O(h(X))$ is the standard big-O notation.
\end{defn}

\begin{prop}\label{cor-countcurvesaverage0}
	Let $f(X)= g(X) + O(X^{a})$ with $0<a<1$ and $g$ differentiable, and with bounded derivative in $[1,\infty)$. 
	Let $N=N(X)$ be a function of $X$ (possibly a constant) such that $N=O(X^a)$. Then, 
	$$f(X+N)-f(X)\myeq \int_X^{X+N} g'(H)\, dH + O\left(\frac{N}{X^{1-a}}\right) \myeq N\cdot g'(X+N) + O\left(\frac{N}{X^{1-a}}\right).$$ 
	In particular, $f(X)-f(X-1) \myeq g'(X) + O(X^{a-1})$.
\end{prop}
\begin{proof}
	Let $N=N(X)$ be a function of $X$ (possibly a constant function) such that $N^{1/a}=O(X)$. Then
	\begin{eqnarray*}
		& & \frac{1}{X}\cdot \sum_{i=1}^{\lfloor X\rfloor} \left(f(i+N)-f(i) -  \int_{i}^{i+N} g'(H)\, dH\right)\\
		&= & \frac{1}{X}\cdot \sum_{i=1}^{\lfloor X\rfloor} \left(f(i+N)-\int_{1}^{i+N} g'(H)\, dH+\int_{1}^{i} g'(H)\, dH - f(i) \right)\\
		&= & \frac{1}{X}\cdot \sum_{j=1}^{N} \left(f(\lfloor X \rfloor+j)-\int_{1}^{\lfloor X \rfloor+j} g'(H)\, dH+\int_{1}^{j} g'(H)\, dH - f(j) \right)\\
		&= & \frac{1}{X}\cdot \sum_{j=1}^{N} \left(O((\lfloor X \rfloor +j)^{a}) + g(1) - g(1) - O(j^{a})  \right)\\
		&= & \frac{1}{X}\cdot \left(O(N\cdot X^{a}) + O(N\cdot N^{a})\right) = O\left(\frac{N}{X^{1-a}}\right),
	\end{eqnarray*}
	since $f(Y)=g(Y)+O(Y^{a})$, and $N^{1/a}=O(X)$. Now it follows that
	\begin{eqnarray*}
		& & \frac{1}{X}\cdot \sum_{i=1}^{\lfloor X\rfloor} \left(f(i+N)-f(i) - Ng'(i+N)\right) \\
		&= & \frac{1}{X}\cdot \sum_{i=1}^{\lfloor X\rfloor} \left(f(i+N)-f(i) - \int_{i}^{i+N} g'(H)\, dH+ \int_{i}^{i+N} g'(H)\, dH- Ng'(i+N)\right),\\
		&=& O\left(\frac{N}{X^{1-a}}\right) + \frac{1}{X}\cdot \sum_{i=1}^{\lfloor X\rfloor} \left( \int_{i}^{i+N} g'(H)\, dH- Ng'(i+N)\right),
	\end{eqnarray*}
\begin{eqnarray*}
		&\leq & O\left(\frac{N}{X^{1-a}}\right) + \frac{N}{X}\cdot \sum_{i=1}^{\lfloor X\rfloor} \left( g'(i) - g'(i+N)\right),\\
		&=& O\left(\frac{N}{X^{1-a}}\right) + \frac{N}{X}\cdot \sum_{j=1}^{N} \left(g'(j) - g'(\lfloor X\rfloor +j)\right) = O\left(\frac{N}{X^{1-a}}\right) + O\left(\frac{N^2}{X}\right)=O\left(\frac{N}{X^{1-a}}\right),\\
	\end{eqnarray*}
	where we have used the fact that $N=O(X^a)$, and if $h(x)$ is decreasing, then $$0\leq \int_a^b h(x) dx - (b-a)h(b)\leq (b-a)(h(a)-h(b)),$$ with  $b-a=(i+N)-i=N$ and $h(x)=g'(x)$. 
\end{proof}

When we apply Prop \ref{cor-countcurvesaverage0} to a function that grows like $\kappa X^{5/6}$, we obtain the following corollary.

\begin{cor}\label{cor-countcurvesaverage}
Let $N=N(X)$ be a function of $X$ (possibly a constant) such that $N^2=O(X)$. Let $f(X)$ be a function such that $f(X)=\kappa X^{5/6}+O(X^{1/2})$. Then, 
$$f(X+N)-f(X)\myeq \int_X^{X+N} \frac{5\kappa/6}{H^{1/6}}\, dH + O\left(\frac{N}{X^{1/2}}\right) \myeq \frac{5\kappa}{6}\cdot \frac{N }{(X+N)^{1/6}} + O\left(\frac{N}{X^{1/2}}\right).$$ 
In particular, $f(X)-f(X-1) \myeq (5\kappa/6)X^{-1/6} + O(X^{-1/2})$.
\end{cor}

Next we consider an in-average result for the growth of a function $f(X)$ times a weight function $\theta(X)$.

\begin{prop}\label{prop-ftimestheta}
	Suppose $f(X)= g(X)+O(X^{a})$ and $\theta(X)=s+O(X^{-e})$ for some constants $s\in [0,1)$ and $a,e>0$, where $g(X)$ is a positive, increasing, and differentiable function on $[0,\infty)$, and its derivative is bounded in $(1,\infty)$. Also assume that $\theta(X)\in (0,1)$ for all $X\geq 1$, and it is monotonic, such that $|\theta'(X)|$ is decreasing for all $X\geq 0$. Then,
	\begin{align*}(f(X)-f(X-1))\cdot \theta(X) & \myeq \int_{X-1}^X g'(t)\theta(t)\, dt + O\left(\frac{\theta^{\text{sup}}}{X^{1-a}}+\frac{1}{X}\left(\int_{1}^X g'(t)\cdot |\theta'(t)| \, dt \right)\right).
	\end{align*} 
\end{prop}
\begin{proof}Suppose $f(X)= g(X)+O(X^{a})$. By Prop.  \ref{cor-countcurvesaverage0}, we have that $f(x)-f(x-1)\myeq \int_{x-1}^x g(t)\, dt + O(x^{a-1})$. Thus,
	there are $x_1\geq 0$ and $C_1>0$ such that 
	$$\left|\frac{1}{X}\sum_{j=2}^{\lfloor X \rfloor} \left(f(j)-f(j-1) - \int_{j-1}^j g'(t)dt\right)\right|\leq C_1X^{a-1}$$ for $X\geq x_1$. Let   $\theta^{\text{sup}}$ be the supremum of $\theta(X)$ on $[0,\infty)$. Then, for values $X\geq x_1$, we have
	\begin{align}
		\nonumber & \left|\frac{1}{X}\sum_{j=2}^{\lfloor X \rfloor} \left((f(j)-f(j-1))\theta(j) - \int_{j-1}^j g'(t)\theta(t)\, dt  \right)\right| = \diamondsuit\\
		\nonumber &= \left| \frac{1}{X}\sum_{j=2}^{\lfloor X \rfloor}\left(\theta(j)\left(f(j)-f(j-1) - \int_{j-1}^j g'(t)dt\right) - \int_{j-1}^j g'(t)(\theta(t)-\theta(j))\, dt \right)\right|\\
		\nonumber &\leq \frac{C_1\theta^{\text{sup}}}{X^{1-a}}+\frac{1}{X}\sum_{j=2}^{\lfloor X \rfloor}\left(\int_{j-1}^j g'(t)\cdot |\theta(t)-\theta(j)|\, dt \right) \\
		\label{eq-decrease}& \leq  \frac{C_1\theta^{\text{sup}}}{X^{1-a}}+\frac{1}{X}\left(\int_{1}^X g'(t)\cdot |\theta'(t)| \, dt \right),
		\end{align}
		where we have used the fact that $|\theta(t)-\theta(j)| = |\theta'(t^\ast)|\cdot |t-j|$ for some $t^\ast \in (t,j)$, by the mean value theorem. Thus, $|\theta(t)-\theta(j)|\leq |\theta'(t^\ast)|\leq |\theta'(t)|$, because $j-1\leq t\leq j$, and $|\theta'(t)|$ is decreasing as $t$ increases.
\end{proof}

\begin{cor}\label{cor-countcurvesaverage1}
 Let $f(X)$ be a function such that $f(X)=\kappa X^{5/6}+O(X^{1/2})$, and let $\theta(X)=s+O(X^{-e})$ for some constants $e>0$, and $s\in [0,1)$, such that $\theta(X)\in (0,1)$ for all $X\geq 1$, and $|\theta'(X)|$ decreases for all $X\geq 1$, and $|\theta'(X)|\leq (\theta^{\text{sup}}\cdot C)/X^{1+e}$, for all $X\geq 0$, for some constant $C$. Then, 
\begin{align*} (f(X)-f(X-1))\cdot \theta(X) & \myeq \frac{5\kappa}{6}\int_{X-1}^X \frac{\theta(t)}{t^{1/6}}\, dt + O\left(\frac{1}{X^{1/2}}\right)\\
& \myeq  \frac{5\kappa}{6}\frac{\theta(X)}{X^{1/6}} + O\left(\frac{1}{X^{1/2}}\right),\end{align*}
where the implied error in the approximation is bounded by $C_1\cdot \theta^\text{sup}/X^{1/2}$, for some constant $C_1$ that depends on $f$, and $\theta^\text{sup}$ is the supremum of $\theta(X)$ in $(0,\infty)$.
\end{cor}
\begin{proof}
	Now suppose that $f(X)$ is a function such that $f(X)=\kappa X^{5/6}+O(X^{1/2})$, so that $g(X)=\kappa X^{5/6}$ and $a=1/2$. Further, assume that $|\theta'(X)|\leq (\theta^{\text{sup}}\cdot C)/X^{1+e}$, for all $X\geq 0$, and a constant $C$. Thus, it follows from Eq. (\ref{eq-decrease}) that 
	\begin{align*}
	\left|\frac{1}{X} \left( \sum_{j=2}^{\lfloor X \rfloor}(f(j)-f(j-1))\cdot \theta(X) - \frac{5\kappa}{6}\int_{j-1}^j \frac{\theta(t)}{t^{1/6}}\, dt \right)\right| 
	& \leq  \frac{C_1\theta^{\text{sup}}}{X^{1/2}}+\frac{(5/6)\kappa \cdot \theta^{\text{sup}}\cdot C}{X}\left(\int_{1}^X t^{-7/6-e} \, dt \right)\\
	& \leq \frac{C_1\theta^{\text{sup}}}{X^{1/2}}+\frac{(5/6)\kappa \cdot \theta^{\text{sup}}\cdot C}{(1/6+e)X}\cdot \left(1-\frac{1}{X^{1/6+e}}\right)\\ &\leq \frac{C_2   \theta^{\text{sup}}}{X^{1/2}}= O\left(\frac{1}{X^{1/2}}\right),
	\end{align*}
	where the inequalities are valid for all $X\geq x_1$, with $x_1$ and $C_2=\max\{C_1, 5\kappa C \}$, and $C_1$ as defined in the proof of Prop \ref{prop-ftimestheta}.
	
	The second part follows similarly, when we note that
	\begin{align*} \left|\int_{X-1}^X\frac{\theta(t)}{t^{1/6}}\, dt - \frac{\theta(X)}{X^{1/6}}\right| & \leq \left|\frac{\theta(X)}{X^{1/6}} - \frac{\theta(X-1)}{(X-1)^{1/6}}\right|\leq \frac{|\theta(X)-\theta(X-1)|}{X^{1/6}}\\
	&\leq \frac{|\theta'(X-1)|}{X^{1/6}}\leq \frac{\theta^{\text{sup}}\cdot C}{X^{1/6}(X-1)^{1+e}}\leq \frac{2C\theta^{\text{sup}}}{X^{7/6}}=O(X^{-(7/6)}),\end{align*}
	for $X\geq 3$, so the difference can be subsumed into the error term  $O(X^{-1/2})$. 
\end{proof}

\begin{lemma}\label{lem-inaveweights}
	Suppose $f(X)\myeq g(X) + O(X^a)$ for some $a$, and let $\theta(X)$ be a function such that $0<\theta(X)<1$ for all $X>0$, and $\theta(X)$ is decreasing, with $\theta(X)=O(X^b)$ for some $b<0$. Then, $f(X)\theta(X) \myeq g(X)\theta(X) +  O(X^{a+b})$ if $a+b\neq -1$, and $O(\log X)$ otherwise.
\end{lemma}
\begin{proof}
	Suppose $\varphi_k$, for $k\geq 1$, are real numbers such that $|\frac{1}{X}\sum_{k=1}^{\lfloor X \rfloor} \varphi_k | \leq C_1\cdot X^a$ for large enough $X$, and suppose $\theta_k$ is a sequence of numbers in $[0,1]$ such that $\theta_k \leq  C_2 \cdot k^b$, for large enough $k$, and $\theta_k$ is decreasing. Put $C_3=C_1\cdot C_2$. Then, we can use Abel's lemma (summation by parts) to obtain
	\begin{eqnarray*}
		\left|\frac{1}{X}\sum_{k=1}^{\lfloor X \rfloor} \theta_k\varphi_k \right| &\leq & \left|\frac{\theta_{\lfloor X \rfloor}}{X}\sum_{k=1}^{\lfloor X \rfloor} \varphi_k \right| + \frac{1}{X}\sum_{k=1}^{\lfloor X \rfloor - 1} \left|\theta_{k+1}-\theta_k\right|\cdot k\cdot \left|\frac{1}{k}\sum_{j=1}^{k} \varphi_j  \right|\\
		&\leq& C_3 \cdot X^{a+b} + \frac{1}{X}\sum_{k=1}^{\lfloor X \rfloor - 1} C_3\cdot |b|\cdot k^{a+b}\\
		&\leq& C_3 \cdot X^{a+b} + \begin{cases} \frac{C_3\cdot |b|}{|a+b+1|}\cdot X^{a+b}, & \text{ if } a+b\neq -1,\\
			\frac{C_3\cdot |b|}{|a+b+1|}\cdot \log X, &\text{ if } a+b=-1
		 \end{cases}  = 	\begin{cases} O(X^{a+b}), & \text{ if } a+b\neq -1,\\
		 O(\log X), &\text{ if } a+b=-1.
	 \end{cases}
	\end{eqnarray*}
	where we have used $\sum_{k=1}^{n-1} k^c =\leq \int_{1}^n  t^{c}\, dt$, and $|\theta_{k+1}-\theta_k|\leq |(C_2\cdot k^b)'|$ because $\theta_k$ is decreasing and $\theta_k\leq C_2k^b$. 
\end{proof}

\begin{lemma}\label{lem-sumofthetas}
	Let $\theta(X)$ be a function such that $0<\theta(X)<1$ for all $X>0$, such that   $\theta(X)=s+O(X^b)$ with $b<0$ and $0\leq s<1$. Let $N=N(X)$ be a function of $X$ (possibly constant), and let $\alpha_1,\ldots,\alpha_N$ be positive real numbers such that $\sum_{i=1}^N \alpha_i = 1$. Then,
	$$\sum_{i=1}^N \alpha_i \cdot \theta(X+i)=\theta(X) + O(X^{b}), \text{ and } \sum_{i=1}^N \alpha_i \cdot \theta(X+i)(1-\theta(X+i))=\theta(X)(1-\theta(X)) + O(X^{b}).$$
\end{lemma}
\begin{proof} Let $\alpha_i$ and $\theta(X)$ be as in the statement. Then,
	\begin{eqnarray*}
	\left(\sum_{i=1}^N \alpha_i\cdot  \theta(X+i)\right) - \theta(X) &=& \left(\sum_{i=1}^N \alpha_i\cdot  \theta(X+i)\right) - \sum_{i=1}^N\alpha_i\cdot \theta(X) = \sum_{i=1}^N \alpha_i\cdot \left( \theta(X+i)-\theta(X)\right)\\
	&=& \sum_{i=1}^N \alpha_i \cdot O(X^b) = O(X^b),
\end{eqnarray*}
where we have used the fact that $\theta = s + O(X^b)$ with $b<0$, so 
$$|\theta(X+y)-\theta(X)|\leq |\theta(X+y)-s| + |s-\theta(X)|\leq  C\cdot(X+y)^b+C\cdot X^b\leq 2C\cdot X^b$$ for all $y\geq 0$, for sufficiently large $X$, and some constant $C>0$. Similarly,
\begin{eqnarray*}
	 \left(\sum_{i=1}^N \alpha_i \cdot \theta(X+i)(1-\theta(X+i)) \right) - \theta(X)(1-\theta(X)) &=& \\
	\sum_{i=1}^N \alpha_i \cdot \left(\theta(X+i)(1-\theta(X+i))  - \theta(X)(1-\theta(X)) \right) &=&\\
	\sum_{i=1}^N \alpha_i\cdot \left( \theta(X+i)-\theta(X)\right) - \sum_{i=1}^N \alpha_i\cdot \left( \theta(X+i)^2-\theta(X)^2\right) &=& O(X^b) + O(X^{\mathbf{b}}) = O(X^b),
\end{eqnarray*}
where we have used the first part of the proof, and also the fact that $$(s+O(X^b))^2 = s^2 + 2sO(X^b) + O(X^{2b})= s^2+O(X^{\mathbf{b}}),$$ for $b<0$, where $\mathbf{b}=b$ if $s>0$, and $\mathbf{b}=2b$ if $s=0$.  
\end{proof}

\section{Selmer groups}\label{sec-Selmergroups}

Let $\Sel_2(E/\Q)$ be the $2$-Selmer group of $E/\Q$ and let $\Sh(E/\Q)[2]$ be the $2$-torsion subgroup of the Tate-Shafarevich group of $E/\Q$ (as defined in \cite{silverman}, Chapter X) which fit in a short exact sequence
\begin{eqnarray}\label{eq-shortseq}
\xymatrix{ 0 \ar[r] & E(\Q)/2E(\Q) \ar[r]^{\delta_E} & \Sel_2(E/\Q) \ar[r]&\Sh(E/\Q)[2]\ar[r] & 0 }
\end{eqnarray}
As in \cite{heath1}, we shall refer to the quantity
$$\operatorname{rank}_{\Z/2\Z} \left(\Sel_2(E/\Q)/(E(\Q)_\text{tors}/2E(\Q)_\text{tors}) \right)= \operatorname{rank}_{\Z/2\Z}(\Sel_2(E/\Q))-\operatorname{rank}_{\Z/2\Z}(E(\Q)[2])$$
as the $2$-Selmer rank (or, simply, Selmer rank) of $E/\Q$, and will denote it by $\selrank(E(\Q))$. We note here that the exact sequence above implies that $\rank(E(\Q))\leq \selrank(E(\Q))$ for all elliptic curves. We define
$$\s_n = \{ E\in\mathcal{E} : \selrank(E(\Q))=n\},$$
and we will denote by $\s_n(X)$ those curves in $\s_n$ of height up to $X$, and $\pi_{\s_n}(X)=\# \s_n(X)$. Imitating the notation in the previous section, we shall also write $\s_n([X_1,X_2])$ and $\pi_{\s_n}([X_1,X_2])$ when referring to curves in $\s_n$ in the height interval $[X_1,X_2]$, and will abbreviate $\s_n^X=\pi_{\s_n}([X,X])$. Poonen and Rains (\cite{poonen}) have conjectured a value for the limit $s_n=\lim_{X\to \infty} \pi_{\s_n}(X)/\pi_{\mathcal{E}}(X)$, namely
$$s_n=\operatorname{Prob}(\selrank(E(\Q))=n) = \left(\prod_{j\geq 0} \frac{1}{1+2^{-j}} \right)\cdot \left(\prod_{k=1}^n \frac{2}{2^k - 1} \right),$$
and, in fact, they conjecture a similar distribution for $p$-Selmer groups of rank $n$, and any prime $p$. This probability has been shown to hold for quadratic twists of certain elliptic curves (see \cite{heath1}, \cite{heath2}, \cite{swin}, and \cite{kane}). The value of the constant $s_0=\prod_{j\geq 0}(1+2^{-j})^{-1}$ is approximately $0.20971122$, and we have included approximations of $s_n$ for $n=1,\ldots, 6$ for future reference in Table 1.
\begin{table}[h!]
\centering
\def\arraystretch{1.5}
\begin{tabular}{c|c|c|c|c|c|c}
\hline
$s_0$ & $s_1$ &  $s_2$ & $s_3$ & $s_4$ & $s_5$ & $s_6$ \\
\hline
$0.20971122$ & $0.41942244$ & $0.27961496$ & $0.07988998$ & $0.01065199$ & $0.00068722$ & $0.00002181$\\
\hline 
\end{tabular}
\caption{Values of $s_n=\operatorname{Prob}(\selrank(E(\Q))=n)$}\label{tab1}
\end{table} 
\begin{remark}
	Note that $s_n\to 0$ as $n\to \infty$. In fact, $\sum_{n\geq 0} s_n =1$ (see also Lemma \ref{lem-seriesconverge}).
\end{remark}

For our purposes, we are interested in the behavior of the function  $\s_n(X)$, but we are even more interested in the conditional probability
$$\operatorname{pSel}_n(X)=\operatorname{Prob}(\selrank(E(\Q))=n | \h(E)=X)=\# \s_n^X/\#\mathcal{E}^X,$$
when $\#\mathcal{E}^X \neq 0$. In other words, we would like to know the probability that a curve $E$ of height $X$ has Selmer rank $n$. We will model this probability with a random model, described in the following section.

\section{Random model, following Cram\'er, part 1} \label{sec-cramermodel}

In this section we define a space $\widetilde{\E}$ of ``test elliptic curves'' and ``test Selmer elements'', which will become a probability space when taking into account our probabilistic hypotheses $H_A$ and $H_B$ (and their refinement $H_C$). 

\begin{defn}
	\label{defn-testellipticcurve} A test elliptic curve is a triple $E=(X,n,\Sel_2)$ consisting of:
	\begin{itemize}
		\item a positive integer $X\geq 1$, the height of $E$, also denoted $X=\h(E)$,
		\item a non-negative integer $n$, the Selmer rank of $E$, also denoted $n=\selrank(E)$, and 
		\item a vector $\Sel_2(E)=(s_{E,1},s_{E,2},\ldots,s_{E,\lfloor \frac{n}{2} \rfloor} )$ of $\lfloor \frac{n}{2} \rfloor =(n-(n\bmod 2))/2$ test Selmer elements. Each test Selmer element is either a {\rm MW} element, or a $\Sh$ element.
	\end{itemize}
	The set of all test elliptic curves will be denoted by $\widetilde{\E}$, and the subset of those test curves with height $X$ will be denoted by $\widetilde{\E}^X$. 
\end{defn}

It follows from the definition of the space $\wee$ of test elliptic curves that $\wee = \bigcup_{X\geq 1} \wee^X$.

\begin{remark} \label{rem-testellipticcurve}	If $E/\Q$ is an elliptic curve, then we can associate a test elliptic curve $(X,n,\Sel_2)$ to $E$ as follows. Clearly, $X=\h(E)$ is the naive height of $E$, and the non-negative integer $n=\selrank(E)$ is the $2$-Selmer rank of $E$ (defined as in Section \ref{sec-Selmergroups}). Let $\Sel_2(E/\Q)$ be the $2$-Selmer group of $E/\Q$. Then, $\Sel_2(E/\Q)\cong (\Z/2\Z)^{n+t}$, where $t=\rank_{\Z/2\Z}(E(\Q)[2])$. Further, if we assume the finiteness of $\Sh(E/\Q)$, then $\rank_{\Z/2\Z}(\Sh(E/\Q)[2])=2s$ is even, and therefore $n=R_{E/\Q}+2s$, where $R_{E/\Q}$ is the $\Z$-rank of $E(\Q)$. In particular, $n\equiv R_{E/\Q}\bmod 2$. It follows that if $n$ is odd, then $1\leq R_{E/\Q}\leq n$ and $R_{E/\Q}$ is odd, so there is always an element of $\Sel_2(E/\Q)/(E(\Q)_\text{tors}/2E(\Q)_\text{tors})$ that comes from a point of infinite order from the Mordell--Weil group. Hence, when $n$ is odd, we are interested in the other $n-1$ generators of the Selmer group, to see if they come from the Mordell--Weil group, or generate non-trivial Sha elements. Moreover, there are $2m=n-1-(2s)$ generators of the Selmer group that come from the Mordell--Weil group. Thus, we define the set of symbols $\Sel_2$ by 
	$$
	\begin{cases} \text{ $s$ elements of the form }\Sh \text{ and $m=(n-2s)/2$ elements of the form MW} & \text{ if $n$ is even},\\
	 \text{ $s$ elements of the form }\Sh \text{ and $m=(n-1-2s)/2$ elements of the form MW} & \text{ if $n$ is odd}. 
	\end{cases}$$
	so that $n-(n\bmod 2)=2(m+s)$. We note here that $m=R_{E/\Q}/2$ when $R_{E/\Q}$ is even, and $m=(R_{E/\Q}-1)/2$ if the rank is odd. We will come back and explain in more detail why $\Sel_2(E)$ should have $\lfloor \frac{n}{2} \rfloor$ elements in Section \ref{sec-cramermodel2}.
	
\end{remark} 

A note about notation: if $E/\Q$ is an elliptic curve, then $\Sel_2(E/\Q)$ denotes the traditional $2$-Selmer group of $E/\Q$. However, if $E$ is a test curve, then $\Sel_2(E)$ denotes the vector of test Selmer elements of Definition \ref{defn-testellipticcurve}.

\begin{example}\label{ex-testellipticcurves}
	Let $E/\Q$ be the elliptic curve $y^2=x^3+2993x$, with height $4\cdot 2993^3=107245762628$. A $2$-descent shows that the Selmer group $\Sel_2(E/\Q)\cong (\Z/2\Z)^5$. Since $E(\Q)_\text{tors} \cong \Z/2\Z$, it follows that $\selrank(E)=4$. Further, a $4$-descent (using Magma) shows that $E(\Q)/2E(\Q)\cong  (\Z/2\Z)^3$, and $\Sh(E/\Q)[2]\cong (\Z/2\Z)^2$. Hence, this elliptic curve would be represented as a test elliptic curve by the triple
	$$(107245762628,4,(\text{MW},\Sh )).$$
	Similarly, the curve $E'\colon y^2=x^3-1679x$ has Selmer rank $3$, Mordell--Weil rank $1$, and so it would correspond to the triple 
	$$(18932679356,3,(\Sh )).$$
\end{example}

\begin{defn}\label{defn-TT}  We let $\mathbf{T}$ be a space of sequences of subsets of $\wee^X$, defined as follows: 
	$$\mathbf{T} = \left\{ \left(\wtt^X\right)_{X\geq 1}  : \wtt^X \subseteq \wee^X, \sum_{N=1}^X \# \wtt^N =\kappa X^{5/6}+O(X^{1/2}) \right\}.$$
If $I$ is a finite interval in $[1,\infty)$, and $\wtt \in \mathbf{T}$, we will write $\wtt(I)= \bigcup_{X\in I} \wtt^X$, and $\wts_n(I) = \bigcup_{X\in I} \wtt^X \cap \wss_n^X$. Finally, we define 
$$\pi_{\wtt}(I) = \#\wtt(I)=\sum_{X\in I} \# \wtt^X\ \text{ and }\ \pi_{\wts_n}(I) = \# \wts_n(I)= \sum_{X\in I} \# (\wtt^X\cap \wss_n^X).$$
\end{defn} 

\begin{remark}\label{rem-actualellipticcurves} 
	The sequence of test elliptic curves $\left(\E^X\right)_{X\geq 1}$ associated to ordinary elliptic curves over $\Q$ (as in Remark \ref{rem-testellipticcurve}), belongs to $\wtt$, by Theorem \ref{thm-brumer}. Thus, the goal is to predict the behaviour of $\left(\E^X\right)_{X\geq 1}$ from the average asymptotic behaviour of sequences in $\wtt$. 
\end{remark}

In the next definition, we fix a Selmer rank $n$, and we make $\wee^X$ into a probability space by defining a probability measure $P_n^X$ for each $X\geq 1$. 

\begin{hypa}[Hypothesis A, or $H_A$]\label{hypa}
	Let $n\geq 0$, and $X\geq 1$ be fixed. Let $\theta_n(X)$ be a function $[1,\infty)\to (0,1)$ such that $\lim_{X\to\infty} \theta_n(X)=s_n$.  We define a probability space $\left(\wee^X,\mathcal{F}_n^X,P_n^X\right)$ by defining a probability measure $P_n^X$ as follows:
	\begin{itemize}
		\item $\wee^X$ is an infinite discrete space of test elliptic curves of height $X$, and $\wss_n^X$ is the (infinite) subset of test elliptic curves $E$ of height $X$ with $\selrank(E)=n$.
		\item $\mathcal{F}_n^X=\left\{\emptyset,\wss_n^X,\wee^X\setminus \wss_n^X, \wee^X \right \}$.
		\item $P_n^X\left(\wss_n^X\right) = \theta_n(X)$, and $P_n^X\left(\wee^X \setminus \wss_n^X\right)=1-\theta_n(X)$.
	\end{itemize}
	If $X_1,\ldots,X_m$ are natural numbers, then we endow $\prod_{i=1}^m \wee^{X_i}$ with the product measure  $$\prod_{i=1}^m \left(\wee^{X_i},\mathcal{F}_n^{X_i},P_n^{X_i}\right).$$
\end{hypa} 

\begin{lemma} \label{lem-cramer1}
	Let $\left(\wee^X,\mathcal{F}_n^X,P_n^X\right)$ be the probability space defined by Hypothesis A, and let $Y_{\Sel,n,X}\colon\wee^X\to \{0,1\}$ be the function that takes values
	$$Y_{\Sel,n,X}(E)=\begin{cases}
	1 & \text{ if } \selrank(E(\Q))=n,\\
	0 & \text{ otherwise.}
	\end{cases}$$
	Then 
	\begin{enumerate}
		\item  $Y_{\Sel,n,X}$ is a random variable that  follows a Bernoulli distribution with probability $\theta_n(X)$, such that $\lim_{X\to \infty} \theta_n(X)=s_n$.
		\item Let $m\geq 1$, and let $E_1,\ldots, E_m$ be a sample of size $m$ test elliptic curves picked independently from $\wee$, and suppose $E_i$ is a test elliptic curves of height $X_i\geq 1$. Then, the events $$Y_{\Sel,n,X_1}(E_1)=1,\ \ldots ,\ Y_{\Sel,n,X_m}(E_m)=1$$ are mutually independent.
	\end{enumerate}
\end{lemma} 
\begin{proof}
	For (1), from the definitions, if $E\in \wee^X$, then the probability that $Y_{\Sel,n,X}(E)=1$ is given by
	$$ P_n^X(\{E \in \wee^X: Y_{\Sel,n,X}(E)=1 \}) = P_n^X(\wss_n^X)=\theta_n(X),$$
	and similarly, 
	$$ P_n^X(\{E \in \wee^X: Y_{\Sel,n,X}(E)=0 \}) = P_n^X(\wee^X\setminus\wss_n^X)=1-\theta_n(X).$$
	Thus, $Y_{\Sel,n,X}(E)$ follows a Bernoulli distribution $B(1,\theta_n(X))$. 
	
	For (2), consider $(E_1,\ldots,E_m)\in \prod_{i=1}^m\wee^{X_i}$. By Hypothesis A, the probability measure on the product space is the product measure $\prod_{i=1}^m \left(\wee^{X_i},\mathcal{F}^{X_i},P_n^{X_i}\right)$. Consider the events $$A_i=\{(E_1,\ldots, E_m) : E_i \in \wss_n^{X_i} \}$$ for $i=1,\ldots,m$. Then,
	$$ P_n^X(A_1)\cdots P_n^X(A_m) = \prod_{i=1}^m \left(P_n^{X_i}(\wss_n^{X_i})\cdot \prod_{j\neq i} (P_n^{X_j}(\wee^{X_j}))\right) = \prod_{i=1}^m \theta_n(X_i).$$
	On the other hand $A_1\cap  \cdots \cap A_m = \{(E_1,\ldots, E_m) : E_i  \in \wss_n^X \text{ for all } i=1,\ldots,m\}$ so
	$$P_n^X(A_1\cap  \cdots \cap A_m) = P_n^{X_1}(\wss_n^{X_1})\cdots P_n^{X_m}(\wss_n^{X_m}) = \prod_{i=1}^m \theta_n(X_i).$$
	Thus, $P_n^X(A_1)\cdots P_n^X(A_m)=P_n^X(A_1\cap  \cdots \cap A_m)$. Similarly, one can show that if $i\neq j$, then $P_n^X(A_i)P_n^X(A_j)=P_n^X(A_i\cap A_j)$. Thus, the events are mutually independent, as desired.
\end{proof}

\begin{remark}\label{rem-isog1}
	Suppose $E/\Q$ and $E'/\Q$ are two non-isomorphic elliptic curves. If $E$ and $E'$ happen to be in the same isogeny class, then their Selmer ranks will not be independent events. However, by a theorem of Kenku (\cite{kenku}), an elliptic curve $E/\Q$ is isogenous to at most $8$ non-isomorphic elliptic curves over $\Q$. Thus, we may disregard the possibility of isogenous curves, as it would only contribute a negligible error that would vanish as we increase the sample size. Indeed, for a given height $X$, there are about $\kappa X^{5/6}$ elliptic curves of height up to $X$, and if we sample $m$ curves up to height $X$, the probability that none of them are isogenous is given by, at least by,
	$$ 1\cdot \frac{\kappa X^{5/6} - 8}{\kappa X^{5/6}}\cdot \frac{\kappa X^{5/6} - 8\cdot 2}{\kappa X^{5/6}}\cdot \frac{\kappa X^{5/6} - 8\cdot 3}{\kappa X^{5/6}}\cdots \cdot \frac{\kappa X^{5/6} - 8\cdot (m-1)}{\kappa X^{5/6}} $$
	which goes to $1$ as $X$ goes to infinity.
\end{remark}

\section{The number of curves with Selmer rank $n$ up to height $X$}\label{sec-selmer-ratio}

In this section we prove several consequences of the probabilistic model defined in Section \ref{sec-cramermodel}, and we investigate the properties of $\theta_n(X)$.

\begin{cor}\label{cor-selbinomial}
Assume $H_A$, and let $\mathfrak{E}=\{ E_1,\ldots,E_m \}$ be a sample of size $m$ of test elliptic curves in $\wee^X$ chosen independently. Then, the number of curves in $\mathfrak{E}$ of Selmer rank $n$ follows a binomial distribution $B(m,\theta_n(X))$. In particular the expected value of $\#(\mathfrak{E}\cap \wss_n)/\# \mathfrak{E}$ is  $\theta_n(X)$ with standard error $\sqrt{\frac{1}{m}\theta_n(X)(1-\theta_n(X))}$. More generally, if $\{E_1,\ldots,E_m\}$ are test elliptic curves in $\wee$, with $E_i$ of height $X_i$ for $i=1,\ldots,m$, and chosen at random, then 
$$\mathbb{E}(\#(\mathfrak{E}\cap \wss_n)/\# \mathfrak{E})=  \frac{1}{m}\sum_{i=1}^m\theta_n(X_i)$$ with standard error $\sqrt{\frac{1}{m^2}\sum \theta_n(X_i)(1-\theta_n(X_i))}$.
\end{cor}
\begin{proof}
Let us assume $H_A$ and let us first show the most general case. Let $E_1,\ldots,E_m$ be test elliptic curves in $\wee$  of height $X_1,\ldots,X_m$, respectively, chosen at random. In particular, by $H_A$, each random variable $Y_{\Sel,n,X_i}(E_i)\sim B(1,\theta_n(X_i))$, and since the curves are chosen at random, $H_A$ says that the events $Y_{\Sel,n,X_i}(E_i)=1$ are mutually independent. Then, the number of elements in $\# \mathfrak{E}\cap \wss_n$ can be expressed as 
$$t =\# \mathfrak{E}\cap \wss_n= \sum_{i=1}^m Y_{\Sel,n,X_i}(E_i).$$
It follows that the expected value of $t$ is
$$\mathbb{E}(t) = \mathbb{E}\left( \sum_{i=1}^m Y_{\Sel,n,X_i}(E_i)\right) = \sum_{i=1}^m \mathbb{E}(Y_{\Sel,n,X_i}(E_i))=\sum_{i=1}^m \theta_n(X_i),$$
and so the expected value of $\# \mathfrak{E}\cap \wss_n/\#\mathfrak{E}$ is $\frac{1}{m}\sum \theta_n(X_i)$. The standard error of the approximation of $t/m$ by $\frac{1}{m}\sum \theta_n(X_i)$ is given by the square root of the variance of $t/m$. We compute
$$\operatorname{Var}\left(\frac{1}{m}\sum_{i=1}^m Y_{\Sel,n,X_i}(E_i) \right) = \frac{1}{m^2}\sum_{i=1}^m \operatorname{Var}(Y_{\Sel,n,X_i}(E_i))=\frac{1}{m^2}\sum_{i=1}^m \theta_n(X_i)(1-\theta_n(X_i)),$$
were we have used the fact that $Y_{\Sel,n,X_i}(E_i)$ are independent, which implies they are uncorrelated, and therefore the covariance terms vanish. Thus, the standard error is $ \sqrt{\frac{1}{m^2}\sum \theta_n(X_i)(1-\theta_n(X_i))}$ as claimed.

Now, if $X=X_1=\cdots =X_m$, then $t =\# \mathfrak{E}\cap \wss_n= \sum_{i=1}^m Y_{\Sel,n,X}(E_i)$ follows a binomial $B(m,\theta_n(X))$, with mean $m\cdot \theta_n(X)$ and variance $\frac{1}{m}\theta_n(X)(1-\theta_n(X))$, so the expected value of $t/m$ is $\theta_n(X)$ with standard error $\sqrt{\frac{1}{m}\theta_n(X)(1-\theta_n(X))}$, as desired.
\end{proof}

\begin{cor}\label{cor-travelratio}
If we assume $H_A$, then 
$\sum_{n=0}^\infty \theta_n(X) = 1$. Moreover, if $\wtt\in\TT$ and we define 
$$\theta_n(X,N) = \frac{\pi_{\wts_n}((X,X+N])}{\pi_{\wtt}((X,X+N])},$$ for each $X\geq 1$ and $N\geq 1$ if $\pi_{\wtt}((X,X+N])>0$, and  $\theta_n(X,N) =0$ otherwise. Then \begin{enumerate}
\item The expected value of $\pi_{\wts_n}((X,X+N])$ is given by the formula
$$\mathbb{E}(\pi_{\wts_n}((X,X+N])) \myeq \frac{5\kappa}{6}\cdot\int_{X}^{X+N}  \frac{\theta_n(H)}{H^{1/6}}\, dH + O\left(\frac{N}{X^{1/2}}\right),$$
on average (see Definition \ref{defn-inave}), where $\theta_n(X)=s_n+O(X^{-e_n})$.
\item The expected value of $\displaystyle \theta_n(X,N)$ is $\theta_n(X)+O(X^{-e_n})$, with a standard error given on average by $$\sqrt{\frac{1}{\pi_{\wtt}((X,X+N])}\cdot (\theta_n(X)\cdot (1-\theta_n(X)) + O(X^{-e_n}))}.$$
\item Let $N=N(X)$ be a function of $X$, such that $\pi_{\wtt}((X,X+N(X)])$ goes to $\infty$ as $X\to \infty$. Then, $\theta_n(X,N(X)) \to s_n$ in probability,  in the sense that the expected value of $\theta_n(X,N(X))$ goes to $s_n$ and its variance goes to zero. 
\end{enumerate}
\end{cor}
\begin{proof}
Let $E=(X,n,\Sel_2)$ be a fixed test elliptic curve of height $X$. Since $\selrank(E)$ takes precisely one value (a non-negative number $n\geq 0$), it follows that
$$\sum_{k=0}^\infty \theta_k(X) =\sum_{k\geq 0} \operatorname{Prob}\left(Y_{\Sel,k,X}(E) =1\right)=1,$$
by the laws of probability. For part (1) of the statement, let $\wtt\in\TT$ be arbitrary (as in Definition \ref{defn-TT}). We note that
$$\pi_{\wts_n}((X,X+N]) = \sum_{E\in\wtt((X,X+N])} Y_{\Sel,n,X}(E).$$
and therefore we may use Corollary \ref{cor-selbinomial} to obtain the expected value.
\begin{eqnarray*}\mathbb{E}(\pi_{\wts_n}((X,X+N])) &=& \sum_{E\in\wtt((X,X+N])} \theta_n(\h(E))\\
& =&\sum_{H=X+1}^{X+N} \sum_{E\in\wtt([H,H])} \theta_n(H) = \sum_{H=X+1}^{X+N} \pi_{\wtt}([H,H])\cdot \theta_n(H)
\end{eqnarray*}
Since $\wtt\in\TT$, Corollary \ref{cor-countcurvesaverage} shows that $\pi_{\wtt}(H)-\pi_{\wtt}(H-1)=\pi_{\wtt}([H,H]) \myeq \frac{5\kappa}{6}\cdot\int_{H-1}^{H}  \frac{1}{T^{1/6}}\, dT + O(H^{-1/2})$  (see Definition \ref{defn-inave} for the definition of $\myeq$ in this context). Since $|\theta_n(X)|\leq 1$ for all $X$, and if we write $\theta_n(X)=s_n+O(X^{-e_n})$ and assume that $|\theta'(X)|\leq (\theta^{\text{sup}}\cdot e)/X$, for all $X\geq 0$, then Cor. \ref{cor-countcurvesaverage1} implies that $\pi_{\wtt}([H,H])\theta_n(H) \myeq \frac{5\kappa}{6}\cdot\int_{H-1}^{H}  \frac{\theta_n(T)}{T^{1/6}}\, dT + O(H^{-1/2})$. Thus, 
\begin{eqnarray*}
\mathbb{E}(\pi_{\wts_n}((X,X+N])) &\myeq &  \sum_{H=X+1}^{X+N} \left(\frac{5\kappa}{6}\cdot\int_{H-1}^{H}  \frac{\theta_n(T)}{T^{1/6}}\, dT +  O(H^{-1/2})\right) \\
&=& \frac{5\kappa}{6}\cdot\int_{X}^{X+N}  \frac{\theta_n(H)}{H^{1/6}}\, dH +  O\left(\frac{N}{X^{1/2}}\right).\end{eqnarray*}
For part (2), 

\begin{align*} \mathbb{E}(\theta_n(X,N)) & =\mathbb{E}\left(\frac{\pi_{\wts_n}((X,X+N])}{\pi_{\wtt}((X,X+N])} \right) = \mathbb{E}\left(\frac{\sum_{E\in\wtt((X,X+N])} Y_{\Sel,n,X}(E)}{\pi_{\wtt}((X,X+N])} \right)\\
&= \sum_{H=X+1}^{X+N}  \frac{\pi_{\wtt}([H,H])}{\pi_{\wtt}((X,X+N])} \cdot\theta_n(H).
\end{align*}
Putting $\alpha_H = \pi_{\wtt}([H,H])/\pi_{\wtt}((X,X+N])$, we note that $\sum_{H=X+1}^{X+N} \alpha_H = 1$. Since $\theta_n(X) = s_n + O(X^{-e_n})$, Lemma \ref{lem-sumofthetas} shows that $\mathbb{E}(\theta_n(X,N))=\theta_n(X) + O(X^{-e_n})$, as claimed.  Similarly, the variance is given by
\begin{eqnarray*}
\operatorname{Var}(\theta_n((X,X+N])) &=& \frac{1}{(\pi_{\wtt}((X,X+N]))^2}\sum_{E\in\wtt((X,X+N])} \theta_n(\h(E))(1-\theta_n(\h(E)))\\
&=& \frac{1}{\pi_{\wtt}((X,X+N])}\sum_{H=X+1}^{X+N} \frac{\pi_{\wtt}([H,H])}{\pi_{\wtt}((X,X+N])}\cdot  \theta_n(H)(1-\theta_n(H))\\
&=&  \frac{1}{\pi_{\wtt}((X,X+N])}\cdot (\theta_n(X)\cdot (1-\theta_n(X)) + O(X^{-e_n})),
\end{eqnarray*}
by Lemma \ref{lem-sumofthetas}. Finally, the maximum value of the function $f(x)=x(1-x)$ in $[0,1]$ is $1/4$, and therefore standard error can be bounded by $\sqrt{(1/4+O(X^{-e_n}))/\pi_{\wtt}((X,X+N])}$. If we choose $N=N(X)$ such that $\pi_{\wtt}((X,X+N(X)])$ goes to $\infty$ as $X\to \infty$, then the variance of $\theta(X,N)$, and the standard error of the approximation $\theta_n(X,N)$ by $\theta_n(X)$ go to zero as $X\to \infty$. Hence,
$$\theta_n(X,N(X)) \to s_n $$
as $X\to \infty$, in probability, in the sense that the expected value of $\theta_n(X,N(X))$ goes to $s_n$ and the variance goes to zero. This concludes the proof of the corollary.
\end{proof}

\subsection{Refining Hypothesis A}

 In order to refine Hypothesis A, we need candidates for our functions $\theta_n(X)$. We shall use the sequence $(\E^X)_{X\geq 1}$ as a representative of $\TT$ (see Remarks \ref{rem-testellipticcurve} and \ref{rem-actualellipticcurves}). The BHKSSW data (Section \ref{sec-allcurves}, \cite{BHKSSW}) is thus used to estimate the function $\theta_n(X)$ using the  moving ratios $\theta_n(X,N)$ of Corollary \ref{cor-travelratio}. We have plotted approximate values of  $\theta_n(X,0.025\cdot 10^9)$ for $n=1,\ldots,5$ using the BHKSSW database, and the graphs can be found in Figure \ref{fig2}.

\begin{figure}[h!]
\includegraphics[width=6.6in]{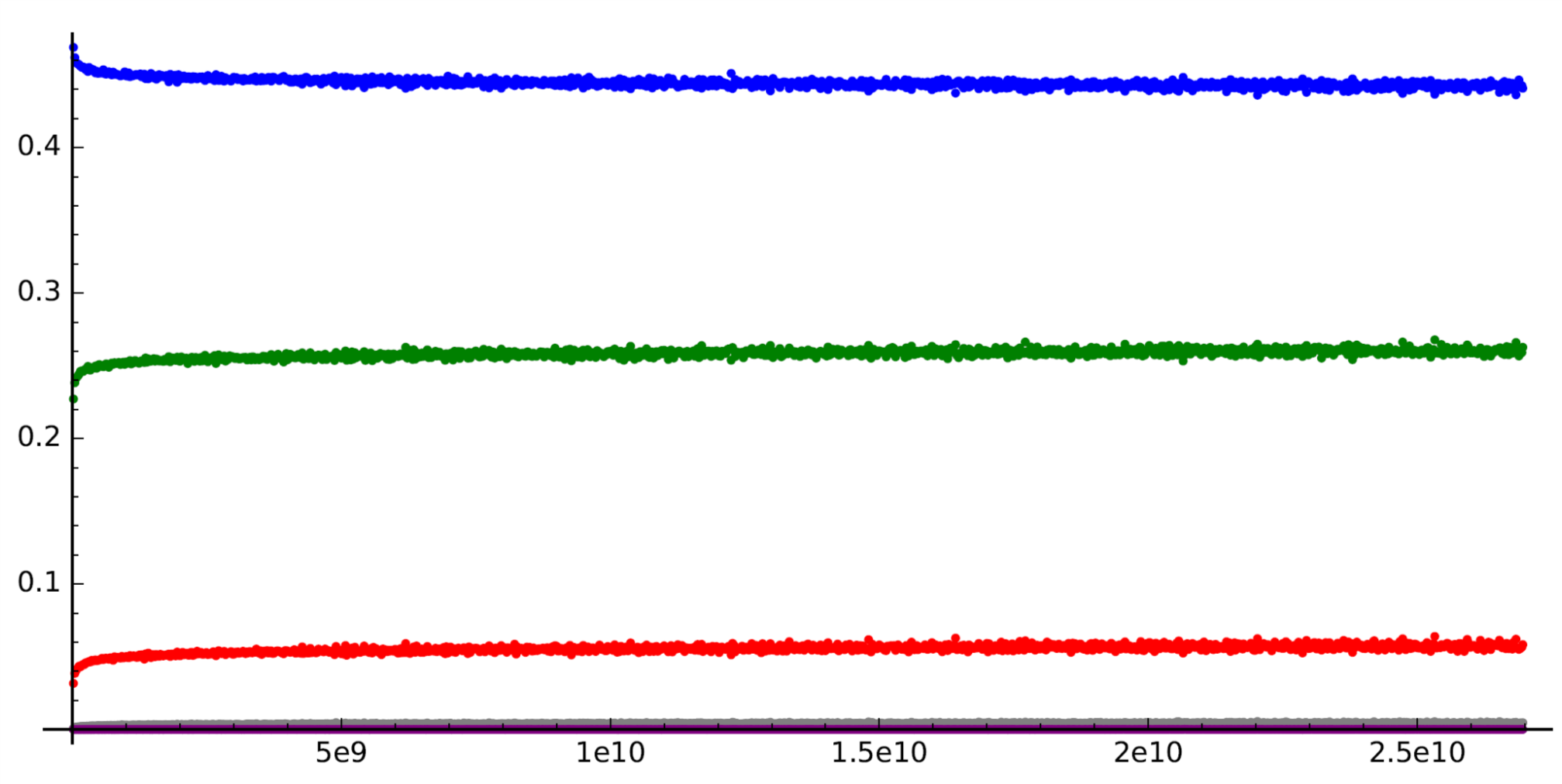}
\caption{Graphs of the moving ratios $\theta_n(X,0.025\cdot 10^9)$ for $n=1$ (blue), $2$ (green), $3$ (red), $4$ (gray), $5$ (purple).}
\label{fig2}
\end{figure}
In Table \ref{tab2} we record the last values of $\theta_n(X,0.025\cdot 10^9)$ that appear in the graphs (which correspond to $X\approx 2.6975\cdot 10^{10}$). We also record the values of $\pi_{\s_n}$ in $[2.6975\cdot 10^{10},2.7\cdot 10^{10}]$. The total number of elliptic curves in the same interval is $\num[group-separator={,}]{182823}$.

\begin{table}[h!]
\centering
\def\arraystretch{1.5}
\begin{tabular}{c||c|c|c|c|c}
\hline
$n$ & $1$ &  $2$ & $3$ & $4$ & $5$\\
\hline 
$\pi_{\s_n}([2.6975\cdot 10^{10},2.7\cdot 10^{10}])$ & $\num[group-separator={,}]{80996}$ & $\num[group-separator={,}]{47427}$ & $\num[group-separator={,}]{10556}$ & $821$ & $29$\\
\hline 
$\theta_n(2.6975\cdot 10^{10},0.025\cdot 10^9)$ & $0.44083621$ &
 $0.26278969$ & $0.05835152$ & $0.00463836$ & $0.00008751$\\
 \hline
 $s_n$ & $0.41942244$ & $0.27961496$ & $0.07988998$ & $0.01065199$ & $0.00068722$ \\
 \hline 
\end{tabular}
\caption{The number of curves of Selmer rank $1\leq n\leq 5$, and the values of $\theta_n(X,N)$ in the interval $[2.6975\cdot 10^{10},2.7\cdot 10^{10}]$, together with the values of $s_n$.}
\label{tab2}
\end{table} 

Finally, we have tried to model the graphs of $\theta_n(X,N)$ using simple rational functions (and assuming the conjectural limit values given by Poonen--Rains \cite{poonen}), and we have found using SageMath best-fit models for the data of $\theta_n(X,N)$ of the form 
$$\theta_n(X,N) \approx \frac{s_n}{1+C_nX^{-e_n}}.$$
In other words, we used a linear regression to find the best coefficients $C_n$ and $e_n$ to fit the data we have. We provide the best-fit values of $C_n$ and $e_n$ in Table \ref{tab3}. 

\begin{table}[h!]
\centering
\def\arraystretch{1.5}
\begin{tabular}{c||c|c|c|c|c}
\hline
$n$ & $1$ &  $2$ & $3$ & $4$ & $5$\\
\hline 
$C_n$ & $-0.40116957$ & $1.41108621$ & $11.18222736$ & $179.71749981$ & $95474.85098037$\\
 \hline
 $e_n$ & $0.08540201$ & $0.12348659$ & $0.14061542$ & $0.20339670$ & $0.39937065$ \\
 \hline 
\end{tabular}
\caption{The parameters of the best-fit models $\theta_n(X,N)\approx s_n/(1+C_nX^{-e_n})$.}
\label{tab3}
\end{table} 

The models constructed above are remarkably good approximations of the values of $\theta_n(X,0.025\cdot 10^9)$, at least up to height $2.7\cdot 10^{10}$. See Figures \ref{fig-ratiomodels123} and \ref{fig-ratiomodels45}. Thus, we refine our Cram\'er-like model of Section \ref{sec-cramermodel} by specifying $\theta_n(X)$ up to two constants $C_n$ and $e_n$.

\begin{hypa}[Hypothesis $H_A'$]\label{conj-selmerratio}
Hypothesis $H_A$ holds and, for each $n\geq 1$, there are constants $C_n$ and $e_n$ such that $\displaystyle \theta_n(X)= \frac{s_n}{1+C_nX^{-e_n}}.$  Moreover, if $n=1,\ldots, 5$, then
the approximate values of $C_n$ and $e_n$ are as in Table \ref{tab3}. Further, we shall assume that there is a constant $C_e>0$ such that $|e_n|\leq C_e$, and $C_e$ does not depend on $n$.
\end{hypa}

\begin{remark}\label{rem-derivativeoftheta}
	If $\displaystyle \theta_n(X)= \frac{s_n}{1+C_nX^{-e_n}},$ then $\displaystyle \theta_n'(X) = \frac{ s_n\cdot e_n \cdot C_n}{X^{1+e_n}(1+C_nX^{-e_n})^2} =\frac{ s_n\cdot e_n \cdot (C_n X^{-e_n})}{X(1+C_nX^{-e_n})^2} $. Since the function $y/(1+y)^2$ has a maximum value of $1/4$ in $[0,\infty)$, it follows that we have a bound $|\theta_n'(X)|\leq s_n\cdot e_n/X^{1+e_n}\leq s_n\cdot C_e/X^{1+e_n}$ for all $X\geq 0$.
\end{remark}

\begin{lemma}\label{lem-seriesconverge}
	Let $s_n$ and $e_n$ be the quantities defined, respectively, in Section \ref{sec-Selmergroups} and Hypothesis \ref{conj-selmerratio} (thus, we assume that $|e_n|\leq C_e$). Then, $\sum_{n\geq 1} s_n$, $\sum_{n\geq 1} ns_n$, and $\sum_{n\geq 1} ns_ne_n$ are convergent.
\end{lemma}
\begin{proof}
	 Let us define $t_1=s_1$ and 
	$t_n = t_1/(2^{\frac{n(n-1)}{2}-1})$
	for $n\geq 2$. Then, the definition of $s_n$ (in Section \ref{sec-Selmergroups}) implies that $s_n\leq t_n$, and so,
	$$\sum_{n=1}^N s_n \leq \sum_{n=1}^N n s_n \leq  \sum_{n=1}^N n t_n\leq \sum_{n=1}^N  \frac{n t_1}{2^{\frac{n(n-1)}{2}-1}}$$
	for any $N\geq 2$. In particular, $\sum_{n\geq 1} s_n$ and $\sum_{n\geq 1} ns_n$ converge (in fact, it can be shown that $\sum_{n\geq 1} s_n = 1$ because they are probabilities that add up to $1$). Moreover, if $|e_n|\leq C_e$, then
	$$\sum_{n\geq 1}^N ns_ne_n \leq C_e \sum_{n\geq 1}^N ns_n$$
	and it follows that $\sum ns_ne_n$ is also convergent.
\end{proof}

\begin{remark}
Let us assume Hypothesis $H_A'$, and let us use Corollary \ref{cor-travelratio} to estimate the error in the approximation $\theta_n(X)\approx \theta_n(X,N)$. The error should be given by the expression
\begin{align*} \text{err}_n(X,N) &=\sqrt{\frac{6\cdot (X+N)^{1/6}\cdot \theta_n(X)(1-\theta_n(X))}{5N\kappa}}\\
&=\sqrt{\frac{6\cdot (X+N)^{1/6}}{5N\kappa}\cdot \frac{s_n}{1+C_nX^{-e_n}}\left(1-\frac{s_n}{1+C_nX^{-e_n}}\right)},
\end{align*}
where we are using the fact that $\pi_{\wtt}((X,X+N]) \myeq (5N\kappa/6)/(X+N)^{1/6} + O(X^{-1/2})$, by Cor. \ref{cor-countcurvesaverage}. 
In Table \ref{tab-errors} we include the values of $\theta_n(X,N)$ based on the BHKSSW data, our model of $\theta_n(X)$ with the constants from Table \ref{tab3}, the error $|\theta_n(X,N)-\theta_n(X)|$, and the predicted standard error $\text{err}_n(X,N)$, for $X=2.6975\cdot 10^{10}$ and $N=0.025\cdot 10^9$.
\begin{table}[h!]
\centering
\def\arraystretch{1.5}
\begin{tabular}{c||c|c|c|c|c}
\hline
$n$ & $1$ &  $2$ & $3$ & $4$ & $5$\\
\hline 
$\theta_n(2.6975\cdot 10^{10},0.025\cdot 10^9)$ & $0.44083621$ &
 $0.26278969$ & $0.05835152$ & $0.00463836$ & $0.00008751$\\
 \hline
 $\theta_n(2.6975\cdot 10^{10})$ & $0.44223400$ & $0.26066727$ & $0.05781814$ & $0.00451697$ & $0.00009141$ \\
 \hline 
 $|\text{Error}|$ & $0.00139779$ & $0.00212241$ & $0.00053337$ & $0.00012138$ & $0.00000390$\\
 \hline 
 $\text{err}_n(2.6975\cdot 10^{10},0.025\cdot 10^9)$ & $0.00115688$ & $0.00102258$ & $0.00054367$ & $0.00015619$ & $0.00002227$\\
 \hline 
\end{tabular}
\caption{Values of: $\theta_n(X,N)$, our model of $\theta_n(X)$, the error $|\theta_n(X,N)-\theta_n(X)|$, and the predicted standard error $\text{err}_n(X,N)$, for $X=2.675\cdot 10^{10}$ and $N=0.025\cdot 10^9$.}
\label{tab-errors}
\end{table} 
\end{remark}

\begin{remark}\label{rem-sellargeheight}
The BHKSSW database (\cite{BHKSSW}) also includes small databases of random samples of elliptic curves at larger heights. In particular, for each $k \in [11, 16]$, they calculated the Selmer rank and rank of a set $\mathcal{L}_k$ consisting of about  $\num[group-separator={,}]{100000}$ curves from a uniform distribution of
all curves in the height range $[10^k, 2 \cdot 10^k)$. We have tested $H_A'$ on these sets $\mathcal{E}_k$ of curves of large height. In Table \ref{tab-errorslargeheight}, we include the value of the moving ratio for $\mathcal{E}_{16}\cap \s_n$, the value of $\theta_n(10^{16})$, the error, and the predicted error $\text{err}_n$. The predicted error for $n=5$ is too large (similarly for $n=4$ to a lesser degree), so the sample is just too small to provide significant evidence. Otherwise, the data for $n=1,2,3$ shows that Hypothesis $H_A'$ seems to be a very good match for the data, even for large heights.
\begin{table}[h!]
\centering
\def\arraystretch{1.5}
\begin{tabular}{c||c|c|c|c|c}
\hline
$n$ & $1$ &  $2$ & $3$ & $4$ & $5$\\
\hline 
$\#\mathcal{E}_{16}\cap \s_n$ & $\num[group-separator={,}]{42631}$ & $\num[group-separator={,}]{27543}$ &
 $\num[group-separator={,}]{7327}$ & $836$ & $38$ \\
\hline 
$\#\mathcal{E}_{16}\cap \s_n/\#\mathcal{E}_{16}$ & $0.42636116$ & $0.27546305$ &
 $0.07327879$ & $0.00836100$ & $0.00038004$ \\
 \hline
 $\theta_n(10^{16})$ & $0.42678631$ & $0.27550444$ & $0.07516196$ & $0.00968314$ & $0.00066148$  \\
 \hline 
 $|\text{Error}|$ & $0.00042515$ & $0.00004138$ & $0.00188317$ & $0.00132213$ & $0.00028144$\\
 \hline 
 $\text{err}_n(\mathcal{E}_{16}\cap \s_n)$ & $0.00239552$ & $0.00269200$ & $0.00308012$ & $0.00338681$ & $0.00275683$ \\
 \hline 
\end{tabular}
\caption{Values of: the moving ratio $\theta_n$, our model of $\theta_n(X)$, the error, and the predicted standard error $\text{err}_n(X,N)$, for the database $\mathcal{E}_{16}$ of height $X\approx 10^{16}$.}
\label{tab-errorslargeheight}
\end{table} 
\end{remark}

\begin{figure}[h!]
\includegraphics[width=6.6in]{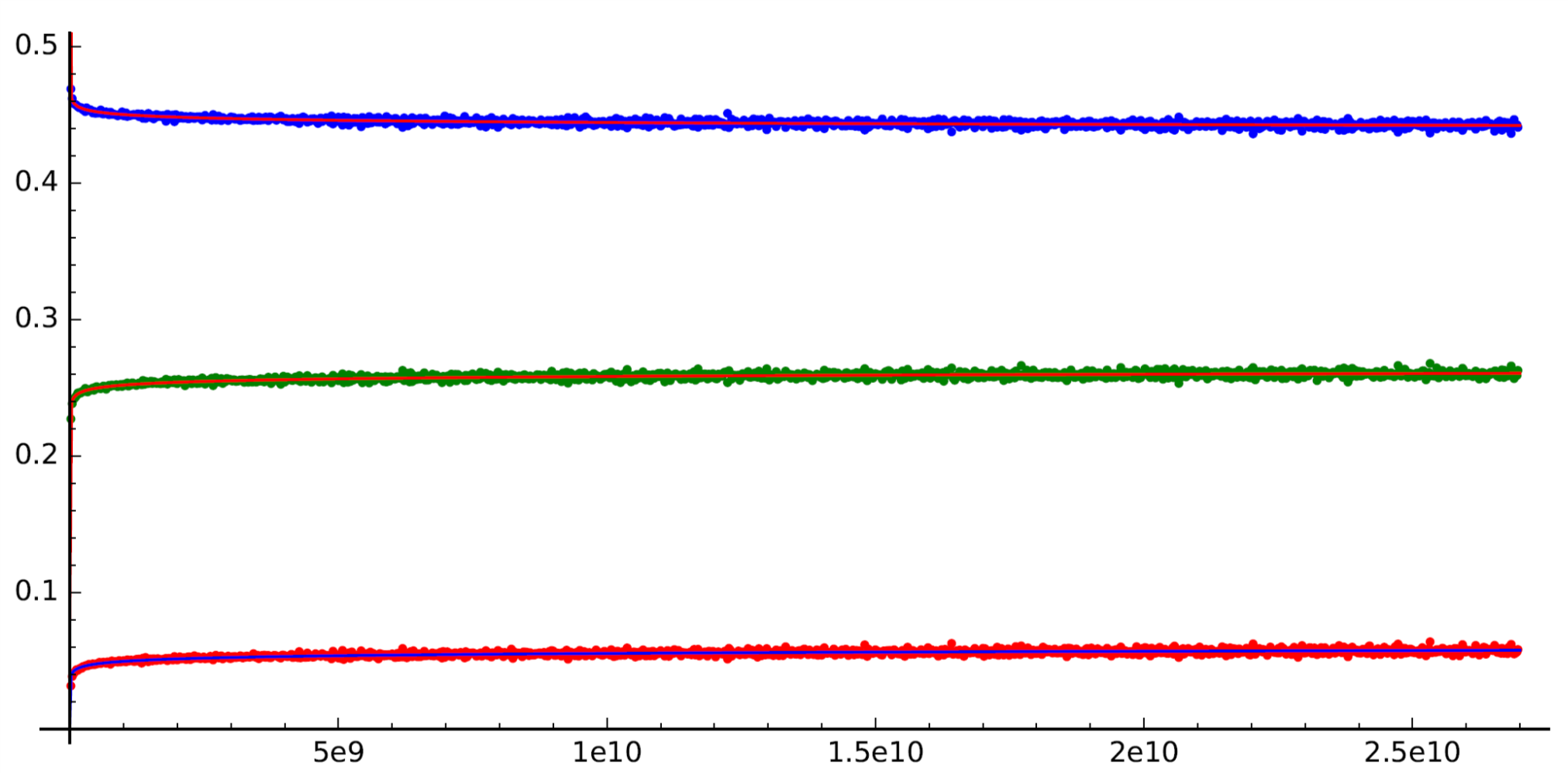}
\caption{Graphs of the moving ratios $\theta_n(X,0.025\cdot 10^9)$ for $n=1$ (blue), $2$ (green), $3$ (red), and the corresponding models of the form $s_n/(1+C_nX^{-e_n})$ (in red).}
\label{fig-ratiomodels123}
\end{figure}

\begin{figure}[h!]
\includegraphics[width=6.6in]{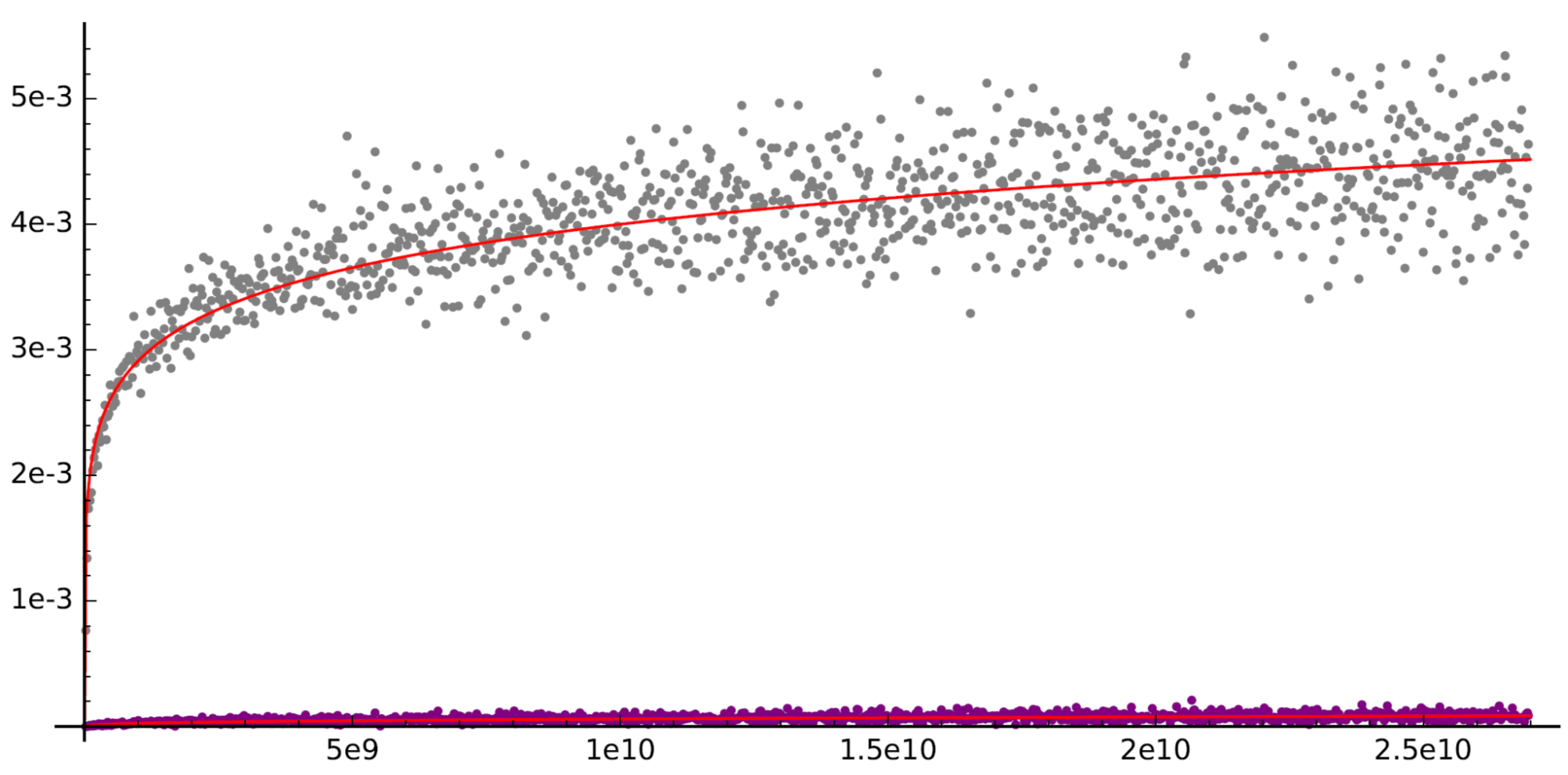}
\caption{Graphs of the moving ratios $\theta_n(X,0.025\cdot 10^9)$ for $n=4$ (gray), $5$ (purple), and the corresponding models of the form $s_n/(1+C_nX^{-e_n})$ (in red, and blue).}
\label{fig-ratiomodels45}
\end{figure}

\begin{example}\label{ex-histselmer}
By Corollary \ref{cor-selbinomial}, Hypothesis A implies that if $E_1,\ldots,E_m$ are test elliptic curves with height $\approx X$, then the number of test curves $E_i$ of Selmer rank $n$ would follow a binomial distribution $B(m,\theta_n(X))$. We have tested this against the BHKSSW database and the data and have always found the result to be in nice agreement with the predictions. For instance, let $T$ be the first $\num[group-separator={,}]{100000}$ elliptic curves with height $\geq 9\cdot 10^{9}$.  We let $m=100$, and pick $100$ curves at random in $T$, and repeat this process $10000$ times. For a fixed $n$ and for each of the $10000$ trials, the distribution of the number $0\leq t\leq 100$ of curves with Selmer rank $n$ would follow a binomial $B(100,\theta_n(X))$, where $X$ is in the interval $[9000573228,9012972924]$. We use our models $\theta_n(X)= s_n/(1+C_nX^{-e_n})$ in order to approximate values, for $n=1,\ldots,4$. We obtain:
$$\theta_1(X)\approx 0.444608,\ \theta_2(X)\approx 0.258128,\ \theta_3(X)\approx 0.055273,\ \theta_4(X)\approx 0.003948$$
for any $X$ in the given interval. If our event of picking $100$ curves follows a binomial $B(100,\theta_n(X))$, then it must be approximately a normal $N(100\cdot \theta_n(X),100\cdot \theta_n(X)\cdot (1-\theta_n(X)))$, where $100\cdot \theta_n(X)$ and $100\cdot \theta_n(X)\cdot (1-\theta_n(X))$ are, respectively, the mean and the variance of the binomial distribution. We have plotted the result of the $10000$ experiments in Figure \ref{fig-hist}, together with the normal distributions predicted by $H_A'$.
\end{example}
 
\begin{center}
\begin{figure}[h!]
\includegraphics[width=6.6in]{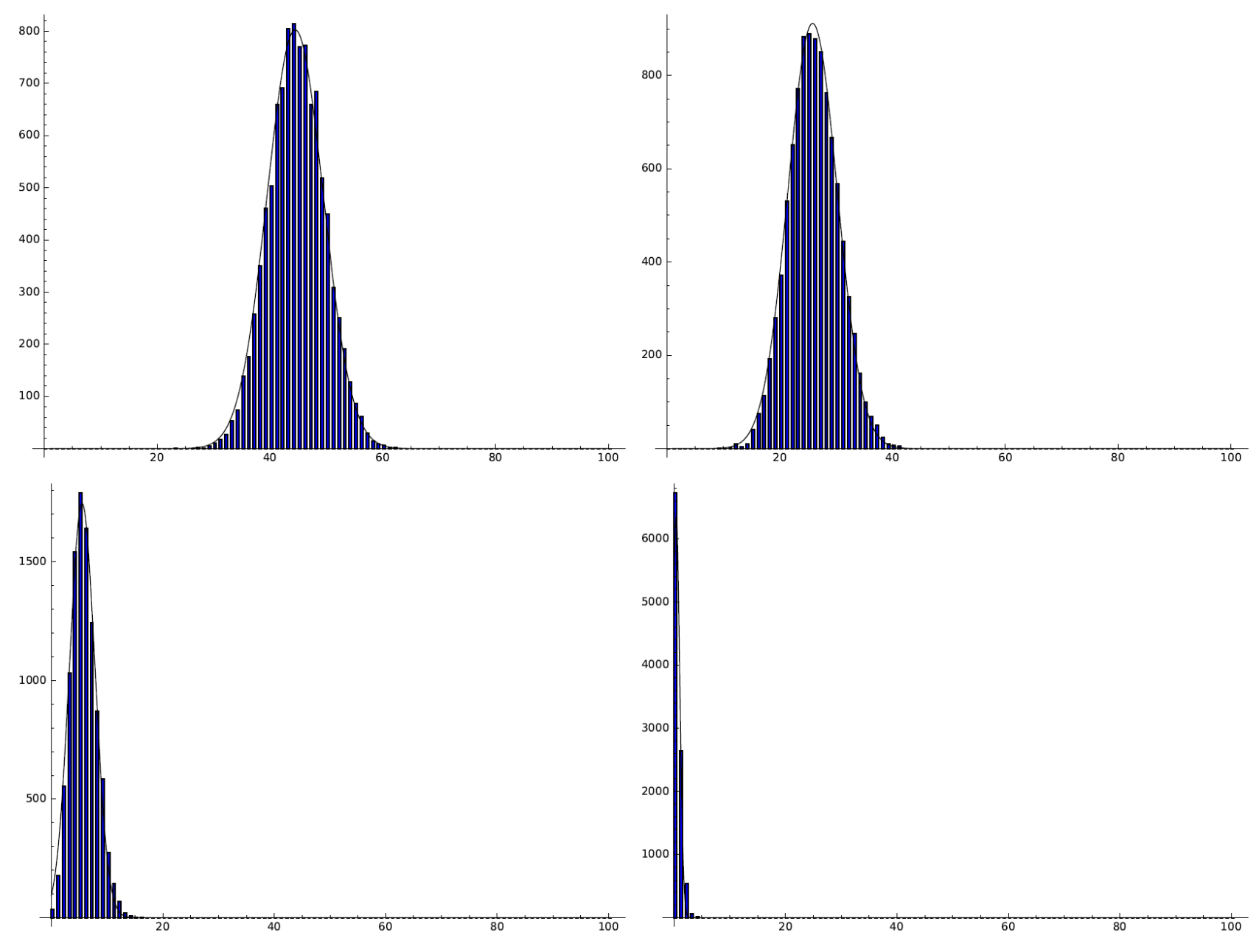}
\caption{Histogram of the distribution of $10000$ experiments of picking $100$ elliptic curves at random of height $\approx 9\cdot 10^{9}$, and counting the number of Selmer ranks equal to $n=1,2,3,4$. The graph is that of the normal distribution predicted by $H_A'$.}
\label{fig-hist}
\end{figure}
\end{center}

Now, we can put our results together to estimate the number of curves of Selmer rank $1,\ldots,5$ up to height $X$.

\begin{prop}\label{prop-numberselcurves}
If we assume $H_A$, and $\wtt\in \TT$ is arbitrary, then:
\begin{enumerate}  
\item The expected value of $\pi_{\wts_n}(X)$ is given on average by 
$$\mathbb{E}(\pi_{\wts_n}(X))\myeq \frac{5\kappa}{6}\int_1^X \frac{\theta_n(H)}{H^{1/6}}\, dH +O\left(X^{1/2}\right) ,$$
where $\kappa = 2^{4/3}\cdot(\zeta(10)\cdot 3^{3/2})^{-1}$. If in addition we assume Hypothesis \ref{conj-selmerratio} ($H_A'$) , then 
$$\mathbb{E}(\pi_{\wts_n}(X))\myeq  \frac{5 \kappa s_n}{6}\int_1^X \frac{1}{H^{1/6}(1+C_nH^{-e_n})}\, dH + O\left(X^{1/2}\right) ,$$
where the constants $s_n$, $C_n$, and $e_n$ are given in Tables \ref{tab1} and \ref{tab3} for $1\leq n\leq 5$. 
\item If we assume Hypothesis \ref{conj-selmerratio} ($H_A'$), then the expected value, on average, is   \begin{align*} \mathbb{E}(\pi_{\wts_n}((X,X+N])) &\myeq  \frac{5\kappa}{6} \frac{ N \theta_n(X)}{X^{1/6}} +O\left(\frac{N}{X^{1/2}}\right)\\
&=  \cfrac{5\kappa s_n N}{6X^{1/6}(1+C_nX^{-e_n})} +O\left(\frac{N}{X^{1/2}}\right).\end{align*}
\end{enumerate} 
\end{prop}
\begin{proof} 
If we assume $H_A$, and $\wtt\in\TT$ is arbitrary, then the expected value of $$\pi_{\wts_n}(X)=\sum_{H=1}^X \pi_{\wts_n}([H,H])$$ is given by $\sum_{H=1}^X \pi_{\wtt}([H,H])\cdot \theta_n(H)$. Thus, Corollaries \ref{cor-countcurvesaverage}  and \ref{cor-countcurvesaverage1}  imply that 
\begin{eqnarray*} \mathbb{E}(\pi_{\wts_n}(X)) =  \sum_{H=1}^X \pi_{\wtt}([H,H])\cdot \theta_n(H)
	&\myeq& \frac{5\kappa}{6}\int_1^X \frac{\theta_n(H)}{H^{1/6}}\, dH + O\left(X^{1/2}\right) \\
& = & \frac{5 \kappa s_n}{6}\int_1^X \frac{1}{H^{1/6}(1+C_nH^{-e_n})}\, dH + O\left(X^{1/2}\right) ,
\end{eqnarray*}
where the last line assumes Hypothesis \ref{conj-selmerratio}. For part (2), Corollary \ref{cor-countcurvesaverage1}, implies that  
\begin{eqnarray*}
\mathbb{E}(\pi_{\wts_n}((X,X+N])) &=& \sum_{H=X+1}^{X+N} \pi_{\wtt}([H,H])\cdot \theta_n(H) \myeq  \frac{5\kappa}{6} \sum_{H=X+1}^{X+N} \left(\frac{\theta_n(H)}{H^{1/6}} + O\left(H^{-1/2}\right) \right)\\
&= & \frac{5\kappa N \theta_n(X)}{6X^{1/6}} +O\left(\frac{N}{X^{1/2}}\right) =  \cfrac{5\kappa s_n N}{6X^{1/6}(1+C_nX^{-e_n})} +O\left(\frac{N}{X^{1/2}}\right),
\end{eqnarray*}
as claimed, where again, the last line assumes Hypothesis \ref{conj-selmerratio}.
\end{proof} 

We have used SageMath to do numerical integration and approximation of the expected values of $\pi_{\wts_n}(X)$ using the formula of Proposition \ref{prop-numberselcurves}, part (1), and we have plotted the graphs against actual data from the BHKSSW database in Figures \ref{fig-numberselcurves123} and \ref{fig-numberselcurves45}.

\begin{figure}[h!]
\includegraphics[width=6.6in]{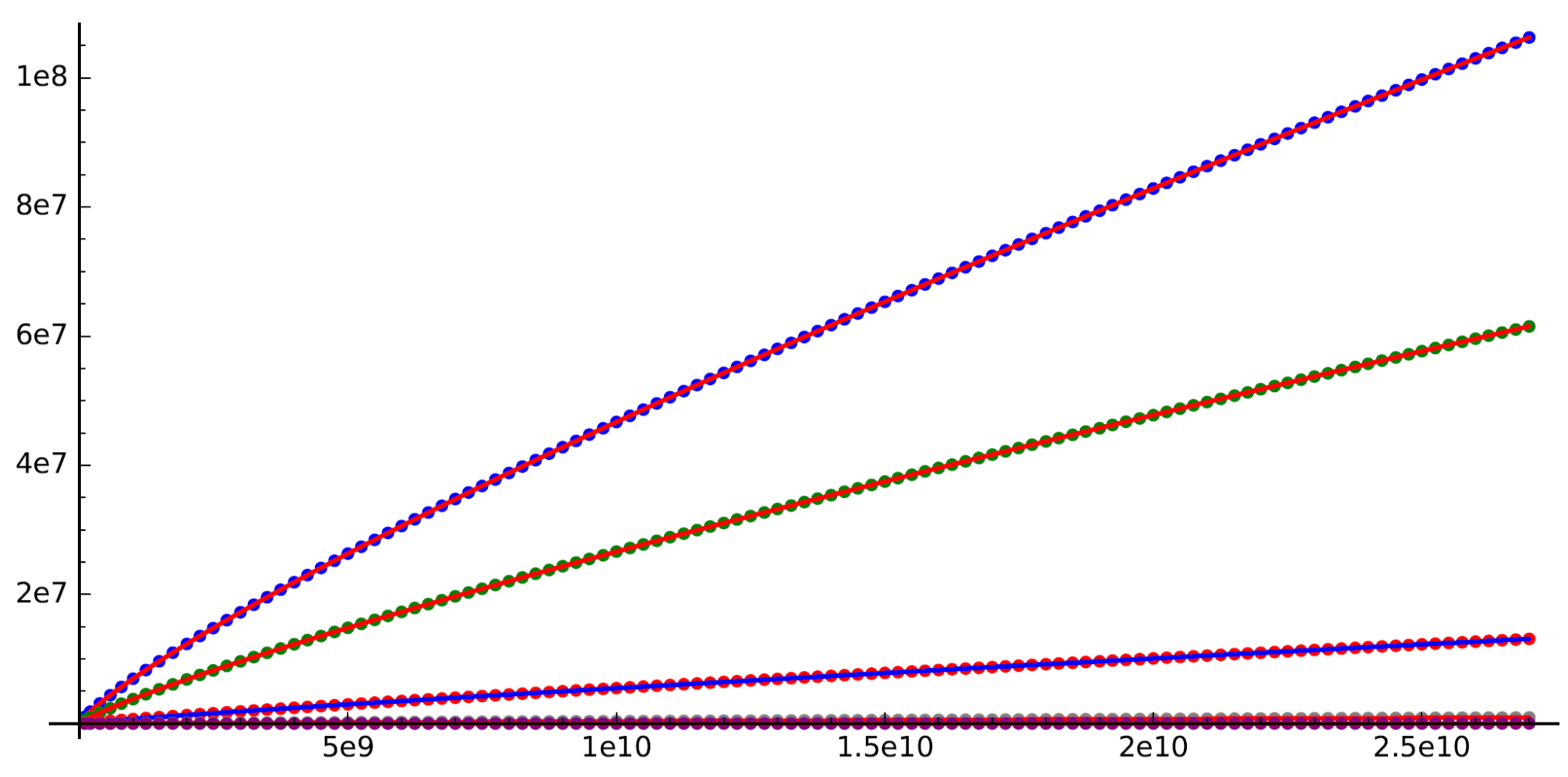}
\caption{Values of $\pi_{\s_n}(X)$ using the BHKSSW database are represented by dots for $n=1$ (blue), $2$ (green), $3$ (red), and the corresponding predictions from Proposition \ref{prop-numberselcurves} (curves in red, except for $n=3$ in blue).}
\label{fig-numberselcurves123}
\end{figure}

\begin{figure}[h!]
\includegraphics[width=6.6in]{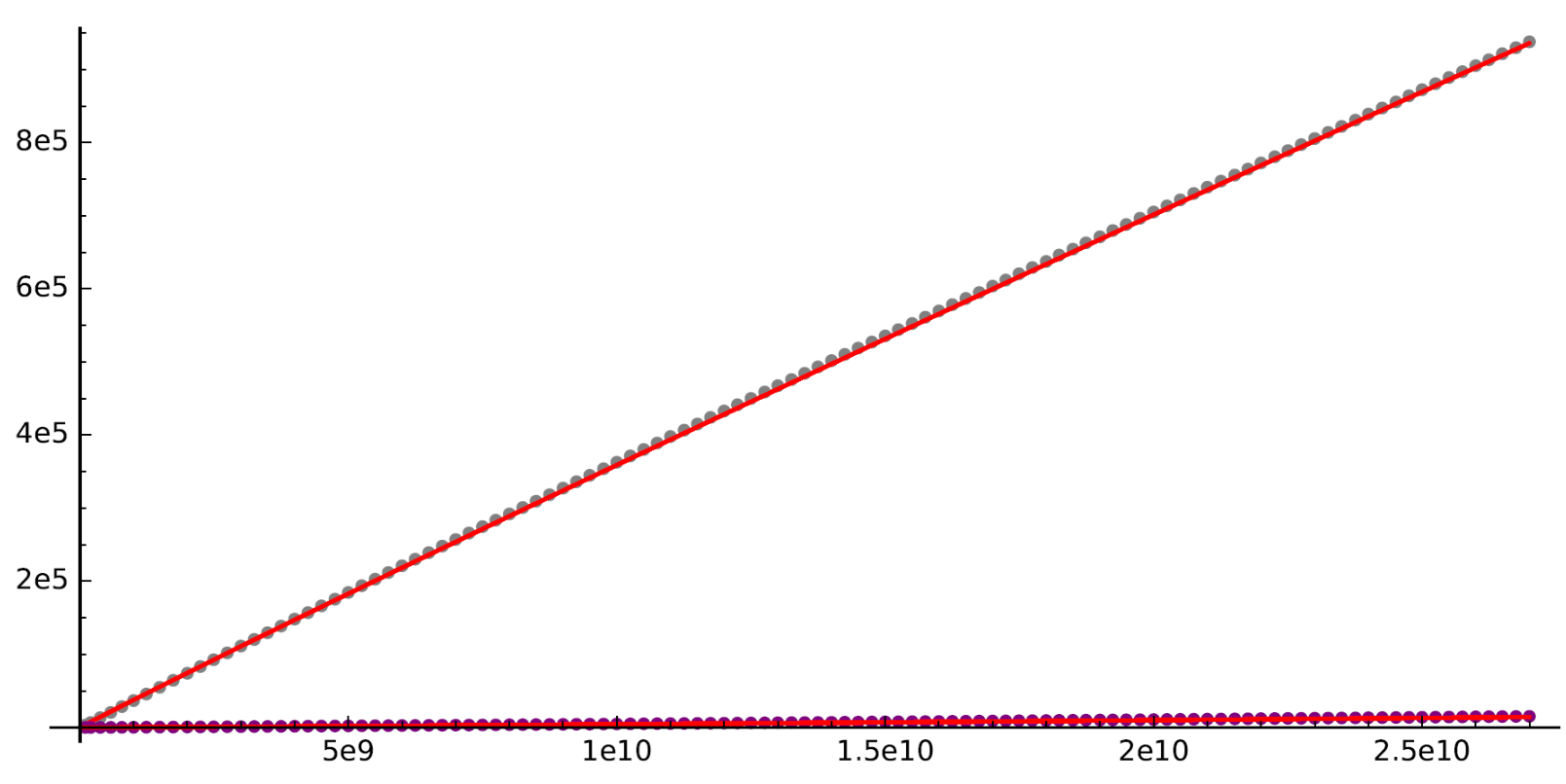}
\caption{Values of $\pi_{\s_n}(X)$ using the BHKSSW database are represented by dots for $n=4$ (gray), $5$ (purple), and the corresponding predictions from Proposition \ref{prop-numberselcurves} (curves in red).}
\label{fig-numberselcurves45}
\end{figure}

Finally, Proposition \ref{prop-numberselcurves} will allow us to write formulas for the average $2$-Selmer rank of a test elliptic curve up to height $X$. We plot our conjectural formula in Figure \ref{fig-aveselrank}.

\begin{prop}\label{prop-aveselrank}
Let $\wtt\in\TT$ be arbitrary, and let $\operatorname{AvgSelRank}_{\wtt}(X)$ be defined by
$$\operatorname{AvgSelRank}_{\wtt}(X) = \frac{1}{\pi_{\wtt}(X)}\sum_{E\in \wtt(X)} \selrank(E).$$
If we assume $H_A$ and we assume that $0\leq \theta_n(X)\leq s_n$ for all $n\geq 2$ and all $X>0$, then the expected value of the average Selmer rank is given by
$$\mathbb{E}(\operatorname{AvgSelRank}_{\wtt}(X)) \myeq  \frac{5\kappa/6}{\pi_{\wtt}(X)}\int_1^X \frac{\sum_{n\geq 1}n\cdot \theta_n(H)}{H^{1/6}}\, dH + O\left(X^{-1/3}\right).$$
Moreover, $\lim_{X\to\infty} \mathbb{E}(\operatorname{AvgSelRank}_{\wtt}(X)) = \sum_{n\geq 1} n s_n =  1.26449978\ldots$. 
\end{prop}
\begin{proof}
In order to compute the average Selmer rank, we note that
$$\operatorname{AvgSelRank}_{\wtt}(X) = \frac{1}{\pi_{\wtt}(X)}\sum_{E\in \wtt(X)} \selrank(E)=\frac{1}{\pi_{\wtt}(X)}\sum_{n\geq 1} \sum_{E\in\wts_n(X)} n = \frac{1}{\pi_{\wtt}(X)}\sum_{n\geq 1} n\cdot \pi_{\wts_n}(X).$$
Thus, by Prop. \ref{prop-numberselcurves} we have that the expected value of $\sum_{n\geq 1} n\cdot \pi_{\wts_n}(X)$ is given by
\begin{align*} 
\mathbb{E}\left(\sum_{n\geq 1} n\cdot \pi_{\wts_n}(X)\right) &\myeq  \sum_{n\geq 1}  n\cdot \left( \frac{5\kappa}{6}\int_1^X \frac{\theta_n(H)}{H^{1/6}}\, dH +O\left(X^{1/2}\right) \right)\\
 &=\frac{5\kappa}{6}\int_1^X \frac{\sum_{n\geq 1}n\cdot \theta_n(H)}{H^{1/6}}\, dH + \sum_{n\geq 1}n\cdot  O\left(X^{1/2}\right).
\end{align*} 
We need further information on the error term $\sum_{n\geq 1} n\cdot O(X^{1/2})$. Recall that in Corollary \ref{cor-countcurvesaverage1} we showed that, for each $n\geq 1$,  the error term is bounded by $\frac{6\kappa \cdot C_1 \cdot  \theta^{\text{sup}} \cdot e}{X^{1/2}}$. Here, for $n\geq 1$, we have $\theta^{\text{sup}}=1$ for $n=1$ and $s_n$ for $n\geq 2$, and $e=e_n$, and $C_1$ and $\kappa$ are constants that do not depend on $n$. Thus, the total error term is bounded by
$$\frac{6\kappa C_1 e_1}{X^{1/2}} + \frac{6\kappa C_1}{X^{1/2}}\cdot \sum_{n\geq 2} ns_n e_n.$$
By Lemma \ref{lem-seriesconverge}, the series $\sum_{n\geq 1} ns_n e_n$ converges. Thus, there is a constant $C_2>0$ such that the total error term is bounded by $C/X^{1/2}$. It follows that 
\begin{align*} 
\mathbb{E}(\operatorname{AvgSelRank}_{\wtt}(X)) &\myeq \frac{(5\kappa)/6}{\pi_{\wtt}(X)}\int_1^X \frac{\sum_{n\geq 1}n\cdot \theta_n(H)}{H^{1/6}}\, dH + O\left(X^{-1/3}\right),
\end{align*} 
by Lemma \ref{lem-inaveweights}, where we have used that $\pi_{\wtt}(X) = O(X^{5/6})$. This proves the main statement of the result.

For the second statement, recall that $\sum_{n\geq 1} ns_n$ converges. Since we are assuming $0\leq \theta_n(X)\leq s_n$ for $n\geq 2$, it follows that $\sum_{n\geq 1} n\theta_n(X)$ converges for any $X$, and $\lim_{X\to \infty} \sum_{n\geq 1} n\theta_n(X) = \sum_{n\geq 1} ns_n = 1.26449978\ldots.$ Next, we calculate the limit of $\mathbb{E}(\operatorname{AvgSelRank}_{\wtt}(X))$ as $X\to \infty$. Let $\alpha = \sum_{n\geq 1} ns_n$. Then,\small 
\begin{eqnarray*} \mathbb{E}(\operatorname{AvgSelRank}_{\wtt}(X)) &\myeq& \frac{5/6}{\pi_{\wtt}(X)}\int_1^X \frac{\sum_{n\geq 1}n\cdot \theta_n(H)}{H^{1/6}}\, dH + O\left(X^{-1/3}\right)\\
&=& \frac{5/6}{\pi_{\wtt}(X)}\left(\int_1^X \frac{\left(\sum_{n\geq 1}n\cdot \theta_n(H) - \alpha\right)}{H^{1/6}}\, dH +\int_1^X \frac{ \alpha}{H^{1/6}}\, dH \right)+ O\left(X^{-1/3}\right).
\end{eqnarray*}
\normalsize
Now, since $f(X)=\sum_{n\geq 1}n\cdot \theta_n(X) - \alpha$ goes to $0$ as $X\to\infty$, it follows that $X^{-5/6}\int_1^X f(H)H^{-1/6}\, dH$ also vanishes in the limit. Hence,
\begin{eqnarray*}\lim_{X\to\infty}\mathbb{E}(\operatorname{AvgSelRank}_{\wtt}(X)) &\myeq& \lim_{X\to \infty}\left( \frac{5/6}{\pi_{\wtt}(X)}\int_1^X \frac{\sum_{n\geq 1}n\cdot \theta_n(H)}{H^{1/6}}\, dH + O\left(X^{-1/3}\right)\right)\\
&=&\lim_{X\to\infty} \frac{5/6}{\pi_{\wtt}(X)}\int_1^X \frac{ \alpha}{H^{1/6}}\, dH  = \alpha =\sum_{n\geq 1} ns_n,\end{eqnarray*} 
as we wanted to prove.
\end{proof}

\begin{figure}[h!]
\includegraphics[width=6.6in]{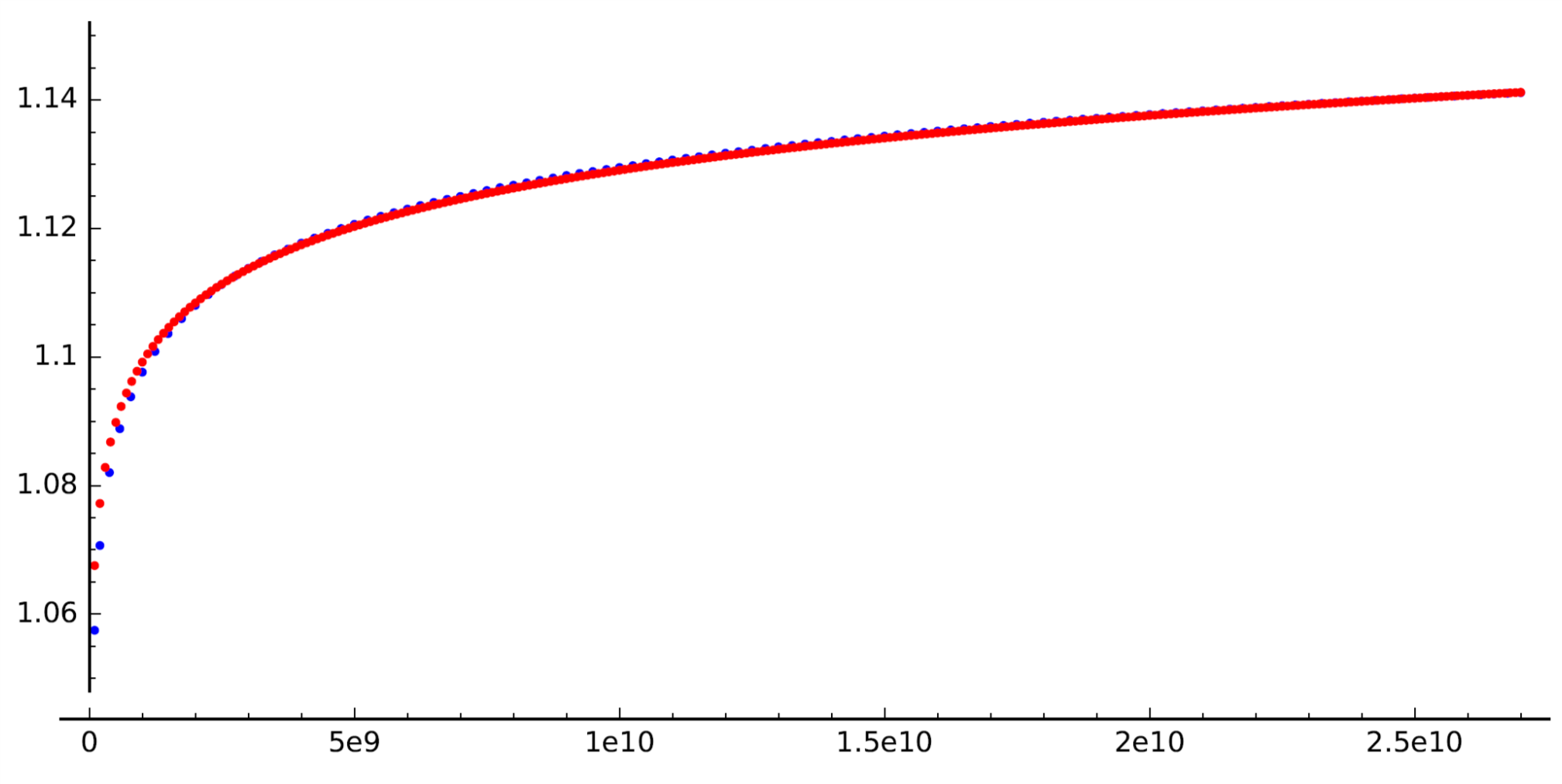}
\caption{Values of $\operatorname{AvgSelRank}(X)$ using the BHKSSW database (blue dots), and the corresponding predictions from Proposition \ref{prop-aveselrank} (red dots).}
\label{fig-aveselrank}
\end{figure}

\section{Random model, following Cram\'er, part 2} \label{sec-cramermodel2}

In this section, we begin with a discussion of how to study the success/failure of the Hasse principle in Selmer elements, we define a probability function $\operatorname{pHasse}_{n}(X)$, and then later we use the discussion of this probability to construct our formal probability space on test elliptic curves in Hypothesis \ref{hypb}.

Let $E/\Q$ be an elliptic curve, and let $\Sel_2(E/\Q)$ and $\Sh(E/\Q)[2]$ be, respectively, the $2$-Selmer group of $E/\Q$ and the $2$-torsion of Sha. We would like to understand how often an element of $\Sel_2(E/\Q)$ is in the image of $E(\Q)/2E(\Q)/(E(\Q)_\text{tors}/2E(\Q)_\text{tors})$ under the natural injection $\delta_E$ of the short sequence in Eq. (\ref{eq-shortseq}) already mentioned in Section \ref{sec-selmer-ratio}. Equivalently, we would like to know when an element of $\Sel_2(E/\Q)$ reduces to a non-trivial element in the quotient  $\Sel_2(E/\Q)/(E(\Q)/2E(\Q))\cong \Sh(E/\Q)[2]$. Inspired by the Cohen-Lenstra heuristics for number fields, Delaunay (\cite{del1}, \cite{del2}) has conjectured certain distributions of Tate-Shafarevich groups (see also Section 5 of \cite{ppvm} for a rich account of Delaunay's conjectures and other related works). As in the case of the results on the density of Selmer ranks discussed in Section \ref{sec-selmer-ratio}, Delaunay's heuristics provide the (conjectural) limit value of the density (i.e., the total probability) of curves with a certain structure of $\Sh(E/\Q)$. However, for our purposes, we are interested in the average size of $\Sh(E/\Q)[2]$ at height $X$, for a curve of fixed Selmer rank $n$. In other words, we are interested in the following conditional probability that measures the failure of the Hasse principle at a given $2$-Selmer element of height $X$:
$${\operatorname{pHasse}}_{n}(X)=\text{Prob}(s\in \Sel_2(E/\Q) \text{ is trivial in } \Sh(E/\Q)[2]\ |\ E\in \s_n \text{ and } \h(E/\Q)=X).$$
An element $s\in \Sel_2(E/\Q)$, in turn, can be visualized as a homogeneous space $H\in \operatorname{WC}(E/\Q)$ in the Weil-Ch\^atelet group of $E/\Q$, such that $H$ is locally solvable everywhere, and the quantity ${\operatorname{pHasse}}_{n}(X)$ would be realized as the probability of $H(\Q)$ having a rational point (see \cite{cremonamazur}).

At this point, we could measure $\text{pHasse}_n(E/\Q)$, the average failure of the Hasse principle for the  $2$-Selmer elements coming from a fixed elliptic curve $E$ of height $X$, as usual, by 
$$\frac{1}{\# \Sh(E/\Q)[2]}= \frac{\#(E(\Q)/2E(\Q))}{\#\Sel_2(E/\Q)} = \frac{1}{2^{\selrank(E(\Q))-\rank(E(\Q))}}.$$ 
However, this ratio does not capture correctly the {\it probability} that a $2$-Selmer element is trivial in $\Sh$. Indeed, it is important to note that if $s,s'\in \Sel_2(E/\Q)$ are two distinct  elements, then the events $s\equiv 0 \in \Sh(E/\Q)$ and $s'\equiv 0 \in \Sh(E/\Q)$ are in general {\it not} independent from a probabilistic point of view. Indeed, if the $2$-Selmer rank of $E/\Q$ is $n$, then $\Sel_2(E/\Q)$ (modulo $2$-torsion contributions) has order $2^n$, but the size of $\Sh(E/\Q)[2]$ is dictated by the classes of $n$ generators $s_1,\ldots,s_n$ of $\Sel_2(E/\Q)/(\text{$2$-torsion})$. Thus, a better measure for ${\operatorname{pHasse}}_{n}(X)$ may be
$$\frac{\rank_{\Z/2\Z}(E(\Q)/2E(\Q))-\rank_{\Z/2\Z}E(\Q)[2]}{\rank_{\Z/2\Z}(\Sel_2(E/\Q))-\rank_{\Z/2\Z}E(\Q)[2]}=\frac{\rank(E(\Q))}{\selrank(E(\Q))}.$$
As it turns out, this ratio is not the correct measure either for odd Selmer rank. If we assume that $\Sh(E/\Q)[2^\infty]$ is finite, then the existence of the Cassels-Tate pairing (\cite{cassels}) $$\Gamma\colon\Sh(E/\Q)[2^\infty]\times \Sh(E/\Q)[2^\infty]\to \Q/\Z,$$
which is a non-degenerate, alternating, and bilinear, implies that the $\F_2$-dimension of $\Sh(E/\Q)[2]$ is always even. It follows that $\rank(E(\Q))\equiv \selrank(E(\Q)) \bmod 2$. In particular, if $\selrank(E(\Q))=n=2k$ or $1+2k$, then the $\F_2$-dimension of $\Sh(E/\Q)[2]$ is in fact dictated by $2k$ classes of $\Sel_2(E/\Q)$ (if $n=1$, then $k=0$, so we will assume that $n\geq 2$ from now on in this section). Therefore, the correct way to define the failure of the Hasse principle for a given elliptic curve is as follows.

\begin{defn}
Let $E/\Q$ be an elliptic curve of Selmer rank $n\geq 2$. We define the average ratio of failure of the Hasse principle of the $2$-Selmer elements (modulo Mordell--Weil $2$-torsion) of $E/\Q$ by
$$\operatorname{pHasse}_n(E/\Q) = \begin{cases}
\cfrac{\rank(E(\Q))}{\selrank(E(\Q))} & \text{\normalsize if $n$ is even, and}\\
\\
\cfrac{\rank(E(\Q))-1}{\selrank(E(\Q))-1} & \text{\normalsize if $n$ is odd.}
\end{cases}$$
\normalsize 
\end{defn}
In other words, $\operatorname{pHasse}_n(E/\Q)$ is the probability that a generator of $\Sel_2(E/\Q)/(\text{2-torsion})$, from any given set of generators, has a non-trivial image in $\Sh(E/\Q)[2]$. We note here that, in all cases, we have $$\operatorname{pHasse}_n(E/\Q)=\cfrac{\rank(E(\Q))-(n \bmod 2)}{n - (n\bmod 2)}.$$

\begin{remark}\label{rem-half}
The fact that the $\F_2$-dimension of $\Sh(E/\Q)[2]$ is even implies that $n=\selrank(E(\Q))$ and $\rank(E(\Q))$ have the same parity. Thus, the rank of $E(\Q)$ is determined by $\lfloor n/2\rfloor$ pairs of generators $\{(s_1,\hat{s}_1),(s_2,\hat{s}_2),\ldots,(s_{\lfloor n/2\rfloor},\hat{s}_{\lfloor n/2\rfloor})\}$ of $\Sel_2(E/\Q)$ such that $s_i\equiv 0 \in \Sh$ if and only if $\hat{s}_i\equiv 0 \in \Sh$. Indeed, let us assume that $n\geq 2$, and first assume that $n$ is even, $n=2k$. Let $E/\Q$ be an elliptic curve of Selmer rank $n$, and let $s_1$ be an arbitrary element of $\Sel_2(E/\Q)/(\text{$2$-torsion})$. We distinguish two cases:
\begin{itemize}
\item If $s_1\equiv 0 \in \Sh$, then $\dim_{\F_2}(\Sh(E/\Q)[2])$ is now at most $2k-2$, and this means that there exists a Selmer element $\hat{s}_1$, linearly independent from $s_1$, such that $\hat{s}_1\equiv 0\in \Sh$ as well. 

\item Otherwise, if $s_1$ represents a non-trivial element in $\Sh$, and if $\Gamma: \Sh(E/\Q)[2]\times \Sh(E/\Q)[2]\to \F_2$ is the Cassels-Tate (non-degenerate, alternating, and bilinear) pairing, than we can choose a non-trivial element $[\hat{s}_1]\in \Sh(E/\Q)[2]$ such that $\Gamma([s_1],[\hat{s}_1])=1$. In particular, $[\hat{s}_1]$ is linearly independent of the class of $s_1$ in $\Sh$, and therefore if $\hat{s}_1$ is now any Selmer element representing the same class $[\hat{s}_1]$ of $\Sh$, then $\hat{s}_1$ and $s_1$ are also linearly independent in $\Sel_2$.
\end{itemize}
In either case, we have found a pair $(s_1,\hat{s}_1)$ of Selmer elements such that $s_i\equiv 0 \in \Sh$ if and only if $\hat{s}_i\equiv 0 \in \Sh$. We can continue this process to find pairs $(s_1,\hat{s}_1),\ldots,(s_k,\hat{s}_k)$. Now let $n=1+2k$ be odd. The proof is analogous, except that if $\dim_{\F_2} \Sel_2(E/\Q)$ is odd, then $\dim_{\F_2}(\Sh(E/\Q)[2])$ must be even, and so there is automatically a Selmer element $s_0$ that is trivial in $\Sh$. Now we can proceed as above to find pairs $(s_1,\hat{s}_1),\ldots,(s_k,\hat{s}_k)$ such that $s_i\equiv 0 \in \Sh$ if and only if $\hat{s}_i\equiv 0 \in \Sh$.

Thus, it may be best to define  $$\operatorname{pHasse}_n(E/\Q)= \frac{ \frac{1}{2}(\rank(E(\Q)) - (n\bmod 2))}{ \frac{1}{2}(\selrank(E(\Q))-(n\bmod 2))}$$
but, of course, the factors of $\frac{1}{2}$ cancel out and this definition is equivalent to the one given above. This simple remark will be crucial when computing the probability of a given Mordell--Weil rank $r$ among curves of Selmer rank $n$ in Theorem \ref{thm-ranksbinomial}.
\end{remark}
 
\begin{remark}
	Let $E=(X,n,\Sel_2)$ be a test elliptic curve, as in Definition \ref{defn-testellipticcurve}. The same considerations stated in this section about the parity of $n$ explain our reasons to define $\Sel_2$ as a vector of $\lfloor\frac{n}{2}\rfloor =(n-(n\bmod 2))/2$ symbols $s_{E,1},\ldots,s_{E,\lfloor\frac{n}{2}\rfloor}$ in $\{\Sh,\operatorname{MW}\}$.
\end{remark} 
 
Now, we turn our attention back to test elliptic curves and our Cram\'er-like model. First, we define the (MW) rank of a test elliptic curve.

\begin{defn}\label{defn-rank} We define the rank (or MW rank) of a test elliptic curve $E=(X,n,\Sel_2)$, as follows:
	$$\rank(E) = (n\bmod 2) + 2\cdot \# \{ \operatorname{MW} \text{ elements in } \Sel_2(E) \}.$$
\end{defn} 

\begin{remark}\label{rem-rank1}
	If $n=1$ and $E=(X,1,\Sel_2)$ is a test curve in $\wss_1^X$, then $\Sel_2(E)$ is empty, and $\rank(E)=(n\bmod 2)=1$, so we will concentrate on the case of $n\geq 2$ from now on.
\end{remark}

\begin{example}
	Let $E$ and $E'$ be the test elliptic curves given by $$E=(107245762628,4,(\text{MW},\Sh )), \text{ and } E'=(18932679356,3,(\Sh ))$$ that appeared in Example \ref{ex-testellipticcurves}. Then,
	$$\rank(E)=(4\bmod 2) + 2\cdot 1 = 2, \text{ and } \rank(E')=(3\bmod 2)+2\cdot 0=1.$$
\end{example}

We are ready to translate our remarks above into a hypothesis for a probabilistic model of test Selmer elements, and define probability spaces $\left(\Sel_{2,n}^X,\mathcal{F}_n^X,P_n^X\right)$ and $\left(\Sel_{2,n}^{\overline{X}},\mathcal{F}_n^{\overline{X}},P_n^{\overline{X}}\right)$ as follows.

\begin{hypb}[Hypothesis B, or $H_B$] \label{hypb}
Let $n\geq 2$ be fixed, let $X\geq 1$, and define 
$$\Sel_{2,n}^{X}=\bigcup_{E\in \wss_n^X} \{s_{E,i} : 1\leq i \leq \left\lfloor\frac{n}{2}\right\rfloor\}=\bigcup_{E\in \wss_n^X} \{s_{E,1},\ldots,s_{E,\lfloor n/2\rfloor}\}$$
where the union is over test elliptic curves $E=(X,n,\Sel_2(E))$ of fixed height $X$ and fixed Selmer rank $n$, and $\Sel_2=(s_{E,1},\ldots,s_{E,\lfloor n/2\rfloor})$. 
\begin{enumerate}
\item[(a)]
 Let $\rho_n(X)$ be a function $[1,\infty)\to (0,1)$ such that $\lim_{X\to\infty} \rho_n(X)=0$, with $\rho_n(X) = O(X^{-f_n})$, for some $f_n>0$.  We define a probability space $\left(\Sel_{2,n}^X,\mathcal{F}_n^X,P_n^X\right)$ by defining a probability measure $P_n^X$ as follows:
\begin{itemize}
	\item $\MW_n^X$ is the subset of $\Sel_{2,n}^X$ of $\MW$ symbols, and $\Sh_n^X=\Sel_{2,n}^X\setminus \MW_n^X$ is the subset of $\Sh$ symbols.
	\item $\mathcal{F}_n^X=\left\{\emptyset,\MW_n^X,\Sh_n^X, \Sel_{2,n}^X \right \}$.
	\item $P_n^X\left(\MW_n^X\right) = \rho_n(X)$, and $P_n^X\left(\Sh_n^X\right)=1-\rho_n(X)$.
\end{itemize}
In other words, $\mathcal{F}_n^X$ and $P_n^X$ are chosen so that the random variable $Y_{\operatorname{Hasse},n,X}\colon\Sel_{2,n}^X\to \{0,1\}$ that takes values
	$$Y_{\operatorname{Hasse},n,X}(s)=\begin{cases}
	1 & \text{ if } s\in\MW_n^X,\\
	0 & \text{ if } s\in\Sh_n^X.
	\end{cases}$$
	is $\mathcal{F}_n^X$-measurable and $Y_{\operatorname{Hasse},n,X}$ follows a Bernoulli distribution with probability $\rho_n(X)$.

\item[(b)] 
	If $m\geq 1$ is fixed, and $\overline{X}=(X_1,\ldots,X_m)$ is a vector of arbitrary heights $X_i\geq 1$, we define $\Sel_{2,n}({\overline{X}})=\prod_{i=1}^m \{\Sel_2(E) : E \in \wss_{n}^{X_i} \}$ as the set of $m\times \lfloor n/2\rfloor$ matrices with coefficients in $\Sel_{2,n} = \bigcup_{X\geq 1} \Sel_{2,n}^X$, such that the $i$-th row is a vector $\Sel_2(E)$ in for some $E \in \wss_{n}^{X_i} $. In addition, for $1\leq i\leq m$ and $1\leq j \leq \lfloor n/2 \rfloor$, we define random variables $Y_{\overline{X},i,j}\colon \Sel_{2,n}({\overline{X}}) \to \{0,1\}$ such that if $M=(\Sel_{2}(E_i))_{i=1}^m \in \Sel_{2,n}({\overline{X}})$, then
	$$Y_{\overline{X},i,j}(M)=\begin{cases}
	1 & \text{ if } s_{E_i,j}\in\MW_n^X,\\
	0 & \text{ if } s_{E_i,j}\in\Sh_n^X.
	\end{cases}$$ 
	Then, we define a probability space $(\Sel_{2,n}(\overline{X}),\mathcal{F}_n^{\overline{X}},P_n^{\overline{X}})$ so that, for each $i,j$, the random variable $Y_{\overline{X},i,j}$ is $\mathcal{F}_n^{\overline{X}}$-measurable, and follows a Bernoulli distribution with probability $\rho_n(X_i)$. Moreover:
	\begin{enumerate}
		\item[(b.1)] If $i_1\neq i_2$, and $1\leq j_1,j_2\leq \lfloor n/2\rfloor$, then the variables  $Y_{\overline{X},i_1,j_1}$ and $Y_{\overline{X},i_2,j_2}$ are independent and uncorrelated. 
		\item[(b.2)] Let $i$ be fixed. Let $1\leq k\leq \lfloor n/2\rfloor$, and $1\leq j_1<j_2<\ldots<j_k\leq \lfloor n/2 \rfloor$. The variables $\{Y_j=Y_{\overline{X},i,j}: 1\leq j \leq \lfloor n/2\rfloor \}$ are not necessarily mutually independent but the expected value $\mathbb{E}\left(Y_{j_1}Y_{j_2}\cdots Y_{j_k} \right)$
		only depends on $n$, $k$, and $X_i$, and it is independent of the choice of indices $j_1,\ldots,j_k$.  
		\end{enumerate} 
\end{enumerate}
\end{hypb}

\begin{remark}
	A few remarks are in order about Hypothesis \ref{hypb}. 
	\begin{enumerate}
		\item In order to minimize baroque symbols, we have used the same notation for the $\sigma$-algebra and probability measure both in Hypotheses \ref{hypa} and \ref{hypb}. The context will hopefully make clear what probability space we are working with.
		
		\item Hypothesis \ref{hypb} assumes that $\lim_{X\to \infty} \rho_n(X)=0$ and $\rho_n(X)=O(X^{-f_n})$. We assume so because of the empirical data we discuss below in Section \ref{sec-refineb}. The author is working on a follow up paper that would perhaps provide a more theoretical reason to assume such growth (decay) for $\rho_n(X)$.
	\end{enumerate}
\end{remark}

\begin{defn} We say that the random variables $Y_1,\ldots, Y_m$ are {\it equicorrelated} if $$\mathbb{E}(Y_{i_1}Y_{i_1}\cdots Y_{i_k})$$
	only depends on $n$, $k$, and $X$, and it is independent of the choice of indices $1\leq i_1<\ldots<i_k\leq m$. 
\end{defn}

\begin{remark} Let $n\geq 2$, let $m\geq 1$, let $1\leq i \leq m$ be fixed, let $X$ be a height, and let $\overline{X}=(X)$. Let $Y_j=Y_{\overline{X},i,j}$ as in Hypothesis \ref{hypb}, part (b.2). Then, the equicorrelation condition of $H_B$, part (2), does not add any conditions at all when $n=2,3$. When $n=4,5$, equicorrelation simply says that $\mathbb{E}(Y_1 )=\mathbb{E}(Y_2)$. This is already implied by the assumption that $Y_1$ and $Y_2$ follow the same Bernoulli distribution (so in fact $\mathbb{E}(Y_1)=\mathbb{E}(Y_2)=\rho_n(X)$). However, the equicorrelation does add new information about the random variables $\{Y_i\}$ for $n\geq 6$. For instance, when $n=6$, it says that 
$$\mathbb{E}(Y_1Y_2)=\mathbb{E}(Y_1Y_3)=\mathbb{E}(Y_2Y_3).$$ 
 When $n=8$, it says that
$$\mathbb{E}(Y_1Y_2)=\mathbb{E}(Y_1Y_3)=\mathbb{E}(Y_1Y_4)=\mathbb{E}(Y_2Y_3)=\mathbb{E}(Y_2Y_4)=\mathbb{E}(Y_3Y_4),$$
and also
$$\mathbb{E}(Y_1Y_2Y_3)=\mathbb{E}(Y_1Y_2Y_4)=\mathbb{E}(Y_1Y_3Y_4)=\mathbb{E}(Y_2Y_3Y_4).$$
\end{remark}

\section{The probability that a $2$-Selmer element is globally solvable}\label{sec-hasse-ratio}

The following two results describe the effects of equicorrelation on the covariance of the random variables. We remind the reader that the covariance of two random variables $V,W$ is given by 
$\Cov(V,W)=\mathbb{E}(VW)-\mathbb{E}(V)\cdot \mathbb{E}(W).$

\begin{lemma}\label{lem-equicorr}
Let $Z,Z',W,W'$ be random variables such that $\mathbb{E}(Z)=\mathbb{E}(Z')$, $\mathbb{E}(W)=\mathbb{E}(W')$.  Then, $\Cov(Z,W)=\Cov(Z',W')$ if and only if $\mathbb{E}(ZW)=\mathbb{E}(Z'W')$, if and only if $\mathbb{E}((1-Z)W)=\mathbb{E}((1-Z')W')$.
\end{lemma}
\begin{proof} By definition $\mathbb{E}(ZW) = \mathbb{E}(Z)\mathbb{E}(W) + \Cov(Z,W)$. Thus,
\begin{eqnarray*} 
  \mathbb{E}(ZW) - \mathbb{E}(Z'W') 
&=& \mathbb{E}(Z)\mathbb{E}(W) + \Cov(Z,W)-(\mathbb{E}(Z')\mathbb{E}(W') + \Cov(Z',W'))\\
&=& \mathbb{E}(Z)\mathbb{E}(W) -\mathbb{E}(Z')\mathbb{E}(W') + \Cov(Z,W) -  \Cov(Z',W')\\
&=&\Cov(Z,W) -  \Cov(Z',W').
\end{eqnarray*}
Thus, $\Cov(Z,W) =  \Cov(Z',W')$ if and only if $\mathbb{E}(ZW)=\mathbb{E}(Z'W')$. Similarly,
\begin{eqnarray*} 
  &  & \mathbb{E}((1-Z)W) - \mathbb{E}((1-Z')W') \\
&=& \mathbb{E}(1-Z)\mathbb{E}(W) + \Cov(1-Z,W)-(\mathbb{E}(1-Z')\mathbb{E}(W') + \Cov(1-Z',W'))\\
&=& (1-\mathbb{E}(Z))\mathbb{E}(W) -(1-\mathbb{E}(Z'))\mathbb{E}(W') - \Cov(Z,W) +  \Cov(Z',W')\\
&=& - \Cov(Z,W) +  \Cov(Z',W'),
\end{eqnarray*}
as claimed, where we have used the fact that $\Cov(a+bX,Y)=b\Cov(X,Y)$, for any constants $a,b$ and random variables $X,Y$.
\end{proof}

\begin{lemma}\label{lem-covlinear}
	Let $m\geq 1$, and let $Y_1,\ldots,Y_{2m}$ be random variables such that $\Cov(Y_i,Y_j)=C$ is constant, for any $1\leq i,j\leq 2m$ with $i\neq j$. Then,
$$\Cov(Y_1+\cdots+Y_m,Y_{m+1}+\cdots+Y_{2m}) = m^2\cdot C,$$
and $\Var(Y_1+\ldots+Y_m)=\sum_{i=1}^m\Var(Y_i) + (m^2-m)\cdot C.$
\end{lemma}
\begin{proof}
	By the linearity property of the covariance,
	$$\Cov(Y_1+\cdots+Y_m,Y_{m+1}+\cdots+Y_{2m}) = \sum_{1\leq i,j\leq m} \Cov(Y_i,Y_{j+m}) = m^2\cdot C.$$
	Similarly,
	\begin{align*} 
	\Var(Y_1+\cdots+Y_m) &= \mathbb{E}((Y_1+\cdots +Y_m)^2)-\mathbb{E}((Y_1+\cdots +Y_m))^2 = \sum_{1\leq i,j\leq m}\mathbb{E}(Y_iY_j) - \mathbb{E}(Y_i)\mathbb{E}(Y_j)\\
	&= \sum_{i=1}^m \Var(Y_i) + \sum_{i\neq j}\Cov(Y_i,Y_j) = \sum_{i=1}^m \Var(Y_i) + (m^2-m)\cdot C.
	\end{align*} 
\end{proof}

\begin{prop}\label{prop-equicorr} 
Assume $H_B$, Let $n\geq 2$, let $m\geq 1$, let $1\leq i \leq m$ be fixed, let $X$ be a height, and let $\overline{X}=(X)$. Let $Y_j=Y_{\overline{X},i,j}$ as in Hypothesis \ref{hypb}, part (b.2).  Let $1\leq s, t\leq \lfloor n/2 \rfloor$ with $s+t\leq \lfloor n/2 \rfloor$. Then, there is a function  $C_{s,t}^n(X)$ such that 
$$C_{s,t}^n(X) = \Cov(Y_{i_1}Y_{i_2}\cdots Y_{i_s},Y_{k_1}Y_{k_2}\cdots Y_{k_t})$$
for any sets of indices $1\leq i_1<i_2<\ldots<i_s\leq \lfloor n/2 \rfloor$ and $1\leq k_1<k_2<\ldots<k_t\leq \lfloor n/2 \rfloor$ with $\{i_u\}\cap \{k_v\}=\emptyset$.
\end{prop}
\begin{proof} Let $1\leq i_1<i_2<\ldots<i_s\leq m$ and $1\leq k_1<k_2<\ldots<k_t\leq m$ with $\{i_u\}\cap \{k_v\}=\emptyset$, and let $1\leq i_1'<i_2'<\ldots<i_s'\leq m$ and $1\leq k_1'<k_2'<\ldots<k_t'\leq m$ with $\{i_u'\}\cap \{k_v'\}=\emptyset$ be another set of such indices. By $H_B$, part (b.2), the random variables $\{Y_j\}$ are equicorrelated, i.e.,  $\mathbb{E}^n_{s}(X)=\mathbb{E}(Y_{i_1}Y_{i_2}\cdots Y_{i_s})=\mathbb{E}(Y_{i_1'}Y_{i_2'}\cdots Y_{i_s'})$ and similarly $\mathbb{E}^n_{t}(X)=\mathbb{E}(Y_{k_1}Y_{k_2}\cdots Y_{k_t})=\mathbb{E}(Y_{k_1'}Y_{k_2'}\cdots Y_{k_t'})$ and also $$\mathbb{E}^n_{s+t}(X)=\mathbb{E}(Y_{i_1}Y_{i_2}\cdots Y_{i_s}Y_{k_1}Y_{k_2}\cdots Y_{k_t}) = \mathbb{E}(Y_{i_1'}Y_{i_2'}\cdots Y_{i_s'}Y_{k_1'}Y_{k_2'}\cdots Y_{k_t'}).$$
Then, we can apply Lemma \ref{lem-equicorr} with $Z=Y_{i_1}Y_{i_2}\cdots Y_{i_s}$, $W=Y_{k_1}Y_{k_2}\cdots Y_{k_t}$, $Z'=Y_{i_1'}Y_{i_2'}\cdots Y_{i_s'}$, and $W'=Y_{k_1'}Y_{k_2'}\cdots Y_{k_t'}$, to obtain the equality of the covariance terms. Thus, the covariance is independent of the chosen sets of $s$ and $t$ distinct random variables in $\{Y_i\}$, and in fact it only depends on $n$, $s$, $t$, and $X$.
\end{proof}

Hypothesis B asserts that $Y_{\operatorname{Hasse},n,X} \sim B(1,\rho_n(X))$, i.e., $Y_{\operatorname{Hasse},n,X}$ follows a binomial distribution with one trial. Now we want to reconstruct the distribution of the rank of a test curve $E\in \wss_n^X$ from that of $Y_{\operatorname{Hasse},n,X}$. We remind the reader that if $X$ is a height, $m=1$, then  $\overline{X}=(X)$ and   $\Sel_{2,n}(\overline{X})=\{\Sel_2(E) : E \in \wss_{n}^{X} \}$.
\begin{theorem}\label{thm-ranksbinomial}
Let $n\geq 1$ be fixed, assume $H_B$, let $R_n = \{ 0,1,\ldots, \lfloor n/2 \rfloor \}$,  and let $X$ be a height. Let $\rank_{n,X}\colon \Sel_{2,n}(\overline{X})\to R_n$ be the function given by  the random variable $Y_1+\ldots+Y_{\lfloor n/2 \rfloor}$ if $n\geq 2$ (and equal $0$ if $n=1$), where $Y_j=Y_{\overline{X},1,j}$ for any $1\leq j \leq \lfloor n/2 \rfloor$.  Then:
\begin{enumerate}
\item Let $E=(X,n,\vec{s}_E )$ be a test elliptic curve with $\vec{s}_E = \Sel_2(E)$. Then
$$\rank_{n,X}(\vec{s}_E) = \frac{\rank(E)-(n\bmod 2)}{2}$$
where $\rank(E)$  appeared in Definition \ref{defn-rank}.
\item If $n\geq 2$,  then the expected value and variance of $\rank_{n,X}$ are given by
\begin{align*} 
\mathbb{E}(\rank_{n,X}) &= \lfloor n/2 \rfloor\cdot \rho_n(X), \text{ and}\\
\Var(\rank_{n,X} ) & =\lfloor n/2 \rfloor\cdot(\rho_n(X)(1-\rho_n(X)) + (\lfloor n/2 \rfloor-1)\cdot C_{1,1}^n(X)),
\end{align*}
where $C_{1,1}^n(X)=\operatorname{Cov}(Y_{j_1},Y_{j_2})$ is the covariance function of  any two random variables given by Proposition \ref{prop-equicorr}.
\item If the variables  $\{Y_j\}$ were mutually uncorrelated (resp. approximately uncorrelated, i.e., if  $C_{1,1}^n(X)\approx 0$), then $\rank_{n,X}$ follows (resp. approximately) a binomial distribution of the form $B(\lfloor n/2 \rfloor,\rho_n(X))$, with expected value $\lfloor n/2 \rfloor\cdot\rho_n(X)$ and variance $\lfloor n/2 \rfloor\cdot\rho_n(X)(1-\rho_n(X))$. 
\item Let $X$ and $X'$ be heights, let $\overline{X}=(X,X')$, and let $\Sel_{2,n}(\overline{X})$ be as in $H_B$. Then, the variables $\rank_{n,X}=\sum_{j=1}^{\lfloor n/2\rfloor} Y_{\overline{X},1,j}$ and $\rank_{n,X'}=\sum_{k=1}^{\lfloor n/2\rfloor} Y_{\overline{X},2,k}$ are uncorrelated.
\end{enumerate}
\end{theorem}
\begin{proof}
For part (1), note that if $n=1$, then $\rank(E)=1$ (see Remark \ref{rem-rank1}) and, therefore $\rank_{n,X}(E)=(\rank(E)-(n\bmod 2))/2=0$.  For the rest of the proof, let us assume that $n\geq 2$. If $X$ is a height, and  $E=(X,n,\Sel_2)$ is a test elliptic curve with $\vec{s}_E = \Sel_2(E)=(s_{E,1},\ldots,s_{E,\lfloor n/2 \rfloor})$, then
\begin{align*} \frac{1}{2}(\rank(E)-(n\bmod 2)) & = \#\{ 1\leq j\leq \lfloor n/2 \rfloor  : s_{E,j} = \MW\}\\ &=\sum_{j=1}^{\lfloor n/2 \rfloor} Y_{\overline{X},1,j}(\vec{s}_E) = \sum_{j=1}^{\lfloor n/2 \rfloor} Y_j(\vec{s}_E)=\rank_{n,X}(\vec{s}_E).\end{align*}

For (2), we first compute the expected value of $\rank_{n,X}$: 
$$\mathbb{E}(\rank_{n,X}) = \mathbb{E}\left(\sum_{j=1}^{\lfloor n/2 \rfloor} Y_j \right) =  \sum_{j=1}^{\lfloor n/2 \rfloor} \mathbb{E}\left( Y_j\right) = \lfloor n/2 \rfloor \cdot \rho_n(X),$$
since each $Y_j\sim B(1,\rho_n(X))$ by Hypothesis B. Let us now calculate the variance of $\rank_{n,X}=\sum Y_j$.
\begin{eqnarray*}
\Var(\rank_{n,X}) &=& \Var\left(\sum_{j=1}^{\lfloor n/2 \rfloor} Y_j \right) =  \sum_{j=1}^{\lfloor n/2 \rfloor} \Var\left( Y_j  \right) +  \lfloor n/2 \rfloor(\lfloor n/2 \rfloor -1 ) \cdot C_{1,1}^n(X)  \\
&=& \lfloor n/2 \rfloor \cdot \rho_n(X)(1-\rho_n(X)) + \lfloor n/2 \rfloor(\lfloor n/2 \rfloor -1)\cdot C_{1,1}^n(X)\\
&=& \lfloor n/2 \rfloor \cdot (\rho_n(X)(1-\rho_n(X)) + (\lfloor n/2 \rfloor -1)\cdot C_{1,1}^n(X)),
\end{eqnarray*}
where we have used Lemma \ref{lem-covlinear}, the fact that for any $j_1\neq j_2$, we have  $\operatorname{Cov}(Y_{j_1},Y_{j_2})=C_{1,1}^n(X)$ by Proposition \ref{prop-equicorr}, and $Y_j \sim B(1,\rho_n(X))$. This proves (2).

In particular, if the random variables $Y_j$ were uncorrelated samples of a Bernoulli distribution (or similarly if $C_{1,1}^n(X)\approx 0$), then $\rank_{n,X}=\sum Y_i$ would follow a binomial distribution $B(\lfloor n/2 \rfloor,\rho_n(X))$. This proves (3).

For (4), let $X$ and $X'$ be heights, let $\overline{X}=(X,X')$, and let $\Sel_{2,n}(\overline{X})$ be as in $H_B$, and we will show that $\rank_{n,X}=\sum_{j=1}^{\lfloor n/2\rfloor} Y_{\overline{X},1,j}=\sum_j Y_j$ and $\rank_{n,X'}=\sum_{k=1}^{\lfloor n/2\rfloor} Y_{\overline{X},2,k}=\sum_{k} Y'_{k}$ are uncorrelated. Indeed,
\begin{align*} \mathbb{E}((\rank_{n,X})(\rank_{n,X'})) &= \mathbb{E}\left(\left(\sum_j Y_j\right)\left(\sum_{k} Y_{k}'\right)\right) = \mathbb{E}\left(\sum_{j,k} Y_jY_k'\right)=\sum_{j,k} \mathbb{E}(Y_jY_k')\\
&=\sum_{j,k} \mathbb{E}(Y_j)\mathbb{E}(Y_k') = \left(\sum_j \mathbb{E}(Y_j)\right)\left(\sum_k \mathbb{E}(Y_k')\right)\\
&=\left(\mathbb{E}\left(\sum_j Y_j\right)\right)\left(\mathbb{E}\left(\sum_j Y_k'\right)\right)=\mathbb{E}(\rank_{n,X})\mathbb{E}(\rank_{n,X'}),
\end{align*}
where we have used the fact that $\mathbb{E}(Y_jY_k')= \mathbb{E}(Y_j)\mathbb{E}(Y_k')$ because $ Y_{\overline{X},1,j}$ and $Y_{\overline{X},2,k}$ are uncorrelated by $H_B$, part (b.1). 
This completes the proof of (4) and of the theorem.
\end{proof}

Using Theorem \ref{thm-ranksbinomial}, we shall describe the average rank and  distribution of curves by Mordell--Weil rank in a sample set of test curves of Selmer rank $n$.

\begin{cor}\label{cor-hasseaverank} Let $E_1,\ldots, E_m$ be a sample of $m\geq 1$ test elliptic curves of Selmer rank $n$, chosen independently, and of heights $X_1,\ldots,X_m$. Then, the expected value of the average rank is $$\mathbb{E}\left(\frac{1}{m}\sum_{i=1}^m \rank(E_i)\right) = (n\bmod 2) +  \frac{2\lfloor n/2 \rfloor}{m}\sum_{i=1}^m \rho_n(X_i)$$ with standard error given by  $\frac{1}{m}\sqrt{\lfloor n/2 \rfloor \sum_{i=1}^m (\rho_n(X_i)(1-\rho_n(X_i)) + (\lfloor n/2 \rfloor -1)C_{1,1}^n(X_i))}$. 
\end{cor}
 \begin{proof} 
Let $E_1,\ldots,E_m$ be as in the statement. Let $\overline{X} = (X_1,\ldots,X_m)$ and let $\Sel_{2,n}(\overline{X})$ be as in $H_B$. Then, Theorem \ref{thm-ranksbinomial} gives us the expected value and variance of $\rank_{n,X_i}(\Sel_2(E_i)))=(\rank(E_i)-(n\bmod 2))/2$ and therefore we can compute the expected value.
 $$\mathbb{E}\left(\frac{1}{m}\sum_{i=1}^m \frac{\rank(E_i)-(n\bmod 2)}{2}\right)=\frac{1}{m}\sum_{i=1}^m \mathbb{E}\left( \rank_{n,X_i}(\Sel_2(E_i))\right)=\frac{1}{m}\sum_{i=1}^m \lfloor n/2 \rfloor\rho_n(X_i),$$
 since $\mathbb{E}(\rank_{n,X_i})=\lfloor n/2 \rfloor\rho_n(X_i)$.
 Thus, $$\mathbb{E}\left(\frac{1}{m}\sum_{i=1}^m \rank(E_i)\right) =(n\bmod 2)+\frac{2\lfloor n/2 \rfloor}{m}\sum_{i=1}^m \rho_n(X_i),$$ as claimed. Next, Theorem \ref{thm-ranksbinomial}, part (4), shows that the values of the random variables $Z_i=\rank_{n,X_i}$ are mutually uncorrelated. In particular, the covariance $\Cov(Z_i,Z_j)=0$ for all $i\neq j$, and it follows that $\Var(Z_i+Z_j)=\Var(Z_i)+\Var(Z_j)+2\Cov(Z_i,Z_j)=\Var(Z_i)+\Var(Z_j)$. Hence, we can compute the variance as follows:
 \begin{eqnarray*} 
 \operatorname{Var}\left(\frac{1}{m}\sum_{i=1}^m \rank_{n,X_i}\right) &=& \frac{1}{m^2}\sum_{i=1}^m \operatorname{Var}(\rank_{n,X_i})\\
 &=& \frac{1}{m^2}\sum_{i=1}^m \lfloor n/2 \rfloor \cdot (\rho_n(X_i)(1-\rho_n(X_i)) + (\lfloor n/2 \rfloor -1)\cdot C_{1,1}^n(X_i)),
 \end{eqnarray*}
 and therefore the standard error is given by 
 $$\sqrt{\frac{\lfloor n/2 \rfloor}{m^2} \sum_{i=1}^m (\rho_n(X_i)(1-\rho_n(X_i)) + (\lfloor n/2 \rfloor -1)C_{1,1}^n(X_i))}$$ 
 as desired.
 \end{proof}

Before we go on to describe the probability of a given Mordell--Weil rank, we need a result on equicorrelated random variables.

\begin{lemma}\label{lem-equicorr2} 
Suppose that the random variables $\{Y_i\}_{i=1}^n$ are equicorrelated. Then:
\begin{enumerate} 
\item For any $1\leq m\leq n$, and any indices $1\leq i_1<\cdots<i_m\leq n$ and $1\leq i_1'<\cdots<i_m'\leq n$,
$$\mathbb{E}((1-Y_{i_1})\cdots (1-Y_{i_m})) = \mathbb{E}((1-Y_{i_1'})\cdots(1-Y_{i_m'})).$$
\item If $X$ and $\{Y_i\}$ are all distinct equicorrelated random  variables, and $1\leq m\leq n$, then:
$$\Cov(X,(1-Y_1)(1-Y_2)\cdots (1-Y_m))= \sum_{i=1}^m (-1)^i\binom{m}{i}\Cov\left(X,\prod_{k=1}^i Y_k\right).$$
\end{enumerate} 
\end{lemma}
\begin{proof}
Part (1) can be easily shown via induction on $m$, where the induction step was essentially proved in Lemma \ref{lem-equicorr}. For part (2), we note that $\Cov(X,(1-Y_1)(1-Y_2)\cdots (1-Y_m))$ equals
\begin{eqnarray*}
 &=& \Cov\left(X,1-\left(\sum_i Y_i\right)+\left(\sum_{i\neq j}Y_iY_j\right)+\cdots+(-1)^m\prod_i Y_i\right)\\
 &=& -\sum_i\Cov(X,Y_i)+\sum_{i\neq j}\Cov(X,Y_iY_j)+\cdots+(-1)^m\Cov\left(X,\prod_i Y_i\right)\\
&=& \sum_{i=1}^m (-1)^i\binom{m}{i}\Cov\left(X,\prod_{k=1}^i Y_k\right),
\end{eqnarray*}
where we have used $\Cov(X,\prod_{s=1}^t Y_{i_s}) =\Cov(X,Y_1\cdots Y_t)$ for any indices $1\leq i_1<\cdots< i_t\leq m$  by Proposition \ref{prop-equicorr}. 
\end{proof}

\begin{remark}\label{rem-notation}
Let us introduce some more notation to simplify our formulas. By Lemmas \ref{lem-equicorr} and \ref{lem-equicorr2}, if $Y_1,\ldots,Y_m$ are distinct equicorrelated random variables, and $1\leq s,t$ with $s+t\leq m$, then the value of
\begin{eqnarray}\label{eq-est}
\mathbb{E}_{s,t}=\mathbb{E}(Y_{i_1}\cdots Y_{i_s}(1-Y_{i_{s+1}})\cdots(1-Y_{i_{s+t}})),
\end{eqnarray}
is independent for any set of $s+t$ distinct indices $\{i_k\}_{k=1}^{s+t} \subseteq \{1,\ldots,m\}$. When the random variables $Y_1,\ldots,Y_{\lfloor n/2 \rfloor}$ are the variables $Y_{\overline{X},i,j}$, with $1\leq j\leq \lfloor n/2 \rfloor$ and a fixed $i$, given by Hypothesis $H_B$, we will write $\mathbb{E}_{s,t}^n(X)=\mathbb{E}_{s,t}$, or simply $\mathbb{E}_{s,t}^n$, to indicate the expected value of a product of random variables as in Equation (\ref{eq-est}) above (which extends the notation $\mathbb{E}_{k}^n(X)=\mathbb{E}(Y_1\cdots Y_k)$ of $H_B$). We also write $\mathbb{E}_{0,0}^{1}(X)=1$. The following lemma gives recursive formulas to compute any expected value $\mathbb{E}_{s,t}^n$.
\end{remark}

\begin{lemma}\label{lem-equicorr3}
Let $X$ be a height, let $\overline{X}=(X)$, and let $Y_j=Y_{\overline{X},1,j}$ for $1\leq j \leq \lfloor n/2 \rfloor$ be the equicorrelated random variables given by Hypothesis $H_B$. Let $0\leq s,t$ with $s+t\leq \lfloor n/2\rfloor$, and let $C_{s,t}^n(X)$ be the covariance coefficient of Proposition \ref{prop-equicorr}. Then, with notation as in Remark \ref{rem-notation}, we have identities
\begin{enumerate}
\item $\mathbb{E}_{1,0}^n = \rho_{n}(X)$ and $\mathbb{E}_{0,1}^n = 1 - \mathbb{E}_{1,0}^n=1-\rho_n(X)$.
\item If $s\geq 2$, then $\mathbb{E}_{s,0}^n = \mathbb{E}_{s-1,0}^n\cdot \mathbb{E}_{1,0}^n + C_{s-1,1}^n(X)$. 
\item If $t\geq 1$, then $\mathbb{E}_{0,t}^n = \mathbb{E}_{0,t-1}^n\cdot \mathbb{E}_{0,1}^n - \sum_{i=1}^{t-1} (-1)^i\binom{t-1}{i}C_{1,i}^n(X)$. 
\item  $\mathbb{E}_{s,t}^n = \mathbb{E}_{s,0}^n\cdot \mathbb{E}_{0,t}^n + \sum_{i=1}^t (-1)^i\binom{t}{i}C_{s,i}^n(X)$.
\end{enumerate}
\end{lemma}
\begin{proof} With notation as in the statement, we compute:
\begin{enumerate}
\item $\mathbb{E}_{1,0}^n = \mathbb{E}(Y_1)=\rho_{n}(X)$ and $\mathbb{E}_{0,1}^n =\mathbb{E}(1-Y_1)= 1 - \mathbb{E}_{1,0}^n$.
\item If $s\geq 2$, then $\mathbb{E}_{s,0}^n = \mathbb{E}(Y_1\cdots Y_{s-1})\mathbb{E}(Y_s)+\Cov(Y_1\cdots Y_{s-1},Y_s)=\mathbb{E}_{s-1,0}^n\cdot \mathbb{E}_{1,0}^n + C_{s-1,1}^n(X)$. 
\item If $t\geq 1$, then $$\mathbb{E}_{0,t}^n =\mathbb{E}((1-Y_1)\cdots (1-Y_{t-1}))\mathbb{E}(1-Y_t)+\Cov((1-Y_1)\cdots (1-Y_{t-1}),1-Y_t)$$
and the covariance term equals $-\Cov((1-Y_1)\cdots (1-Y_{t-1}),Y_t)$ which in turn is $$- \sum_{i=1}^{t-1} (-1)^i\binom{t-1}{i}\Cov\left(Y_t,\prod_{k=1}^i Y_k\right)=-\sum_{i=1}^{t-1} (-1)^i\binom{t-1}{i}C_{1,i}^n(X)$$ by Lemma \ref{lem-equicorr2}.
\item  $\mathbb{E}_{s,t}^n = \mathbb{E}_{s,0}^n\cdot \mathbb{E}_{0,t}^n + \Cov(Y_1\cdots Y_s,(1-Y_1)\cdots (1-Y_{t}))$ and, by Lemma \ref{lem-equicorr2}, the covariance term equals $\sum_{i=1}^t (-1)^i\binom{t}{i}C_{s,i}^n(X)$ as claimed.
\end{enumerate}
\end{proof}

\begin{cor}\label{cor-hasse1}
Let us assume $H_B$. Let $n\geq 2$ be fixed, let $X$ be a height, let $\overline{X}=(X)$. Then:
\begin{enumerate}  
\item The
random variable $Y_{\rk = n-2j}\colon \Sel_{2,n}(\overline{X})\to \{0,1\}$ given by
$$Y_{\operatorname{rk} = n-2j}(\vec{s}_E) = \begin{cases}
1 & \text{ if } \rank_{n,X}(\vec{s}_E)=n-2j,\\
0 & \text{ otherwise,}
\end{cases}$$
is given by
$$Y_{\operatorname{rk} = n-2j} = \sum_{1\leq k_1<\cdots < k_{m(j)}\leq \lfloor n/2 \rfloor} Y_{k_1}\cdot Y_{k_2}\cdots Y_{k_{m(j)}} \cdot \frac{\prod_{i=1}^{\lfloor n/2 \rfloor} (1-Y_i)}{(1-Y_{k_1})(1-Y_{k_2})\cdots (1-Y_{k_{m(j)}})}$$
where $m(j) = \lfloor n/2 \rfloor - j$, and the random variables $Y_j = Y_{\overline{X},1,j}$ are as given by Hypothesis B, such that $\rank_{n,X} = \sum Y_j$. 
\item If the variables $Y_j$ are mutually uncorrelated, then the expected value of $Y_{\operatorname{rk} = n-2j}$ is given by 
$$\mathbb{E}(Y_{\operatorname{rk} = n-2j}) = \binom{\lfloor n/2 \rfloor}{j}\rho_{n}(X)^{\lfloor n/2 \rfloor-j}(1-\rho_{n}(X))^{j}.$$
\item In general, the expected value of $Y_{\operatorname{rk} = n-2j}$, using the notation of Remark \ref{rem-notation} and Lemma \ref{lem-equicorr3}, is given by
$$\mathbb{E}(Y_{\operatorname{rk} = n-2j})=\binom{\lfloor n/2 \rfloor}{j}\cdot \mathbb{E}_{m(j),j}^n (X),$$
where $\mathbb{E}_{m(j),j}^n (X)$ can be calculated recursively using the formulae of Lemma \ref{lem-equicorr3}.
\end{enumerate} 
\end{cor}
 \begin{proof}
 Let $\{Y_j\}$ be the random variables given by Hypothesis B, such that $\rank_{n,X}=\sum Y_j$. It follows that $\rank_{n,X} = n-2j$ if and only if there are exactly $m(j)$ coordinates of $\vec{s}_E=(s_{E,1},\ldots,s_{E,{\lfloor n/2 \rfloor}})$ that are a $\MW$ symbol, if and only if there are exactly $m(j)$ indices $ 1\leq k_1 <\cdots < k_{m(j)} \leq \lfloor n/2 \rfloor$ such that $Y_{k_1}=\cdots=Y_{k_{m(j)}}=1$ and $Y_j=0$ for all other indices. If we fix one such $m(j)$-tuple of indices, then this occurs exactly when the random variable
 $$Y_{k_1}\cdot Y_{k_2}\cdots Y_{k_{m(j)}} \cdot \frac{\prod_{i=1}^{\lfloor n/2 \rfloor} (1-Y_i)}{(1-Y_{k_1})(1-Y_{k_2})\cdots (1-Y_{k_{m(j)}})}$$
 takes the value $1$. Finally, adding over all the possible $m(j)$-tuples  $(k_1,\ldots,k_{m(j)})$, we obtain the random variable equal to $Y_{\rk = n-2j}$, as in the statement.
 
 For the second part of the statement, if the random variables $Y_j$ are mutually uncorrelated, then  
 \begin{eqnarray*}
 	\mathbb{E}(Y_{\rk = n-2j}) & = &  \sum_{1\leq k_1<\cdots < k_{m(j)}\leq \lfloor n/2 \rfloor} \mathbb{E}(Y_{k_1})\cdots \mathbb{E}(Y_{k_{m(j)}}) \cdot \frac{\prod_{i=1}^{\lfloor n/2 \rfloor} (1-\mathbb{E}(Y_i))}{(1-\mathbb{E}(Y_{k_1}))\cdots (1-\mathbb{E}(Y_{k_{m(j)}}))}\\
 	&=&  \binom{\lfloor n/2 \rfloor}{m(j)} \rho_n(X)^{m(j)}(1-\rho_n(X))^j = \binom{\lfloor n/2 \rfloor}{j} \rho_n(X)^{m(j)}(1-\rho_n(X))^j,
 \end{eqnarray*}
 as claimed, where we have used the facts that (a) if  $\operatorname{Cov}(Y,Y')=0$ (or $\approx 0$), then $\mathbb{E}(YY')\cong \mathbb{E}(Y)\cdot \mathbb{E}(Y')$ and $\mathbb{E}(1-Y)=1-\mathbb{E}(Y)$, and (b) that  $\mathbb{E}(Y_j)=\rho_n(X)$.

 For (3), we can calculate the expected value $\mathbb{E}(Y_{\rk = n-2j})$ as follows:
 \begin{eqnarray*}
 \mathbb{E}(Y_{\rk = n-2j}) &=& \mathbb{E}\left(\sum_{1\leq k_1<\cdots < k_{m(j)}\leq \lfloor n/2 \rfloor} Y_{k_1}\cdot Y_{k_2}\cdots Y_{k_{m(j)}} \cdot \frac{\prod_{i=1}^{\lfloor n/2 \rfloor} (1-Y_i)}{(1-Y_{k_1})(1-Y_{k_2})\cdots (1-Y_{k_{m(j)}})}\right) \\
 &=& \sum_{1\leq k_1<\cdots < k_{m(j)}\leq \lfloor n/2 \rfloor} \mathbb{E}\left( Y_{k_1}\cdot Y_{k_2}\cdots Y_{k_{m(j)}} \cdot \frac{\prod_{i=1}^{\lfloor n/2 \rfloor} (1-Y_i)}{(1-Y_{k_1})(1-Y_{k_2})\cdots (1-Y_{k_{m(j)}})}\right)\\
 &=& \sum_{1\leq k_1<\cdots < k_{m(j)}\leq \lfloor n/2 \rfloor} \mathbb{E}_{m(j),j}^n (X) = \binom{\lfloor n/2 \rfloor}{j}\cdot \mathbb{E}_{m(j),j}^n (X),   
 \end{eqnarray*}
 where we have used equicorrelation of random variables $Y_j$ for the equality of the expected value of the product of any $m(j)$ random  variables, and parts (1) and (2) of Lemma \ref{lem-equicorr2}. 
 \end{proof}

Let us simplify the formulas of Corollary \ref{cor-hasse1} for $n=1,\ldots,5$.

\begin{cor}\label{cor-rankprobabilities}
For every $r\geq 0$, we define
$$\wrr_r^X = \{E=(X,n,\Sel_2)\in \wee^X : \rank(E)=r\}.$$
If we assume $H_B$, then the probabilities
$$p_n(r) = \operatorname{Prob}(E\in \wrr_r^X\ |\ E\in \wss_{n}^X) = \mathbb{E}(Y_{\rk = r})$$
for $n=1,\ldots,5$ and $0\leq r \leq n$ are given by the formulas in Table \ref{tab-probabilities}.

\begin{table}[h!]
\centering
\def\arraystretch{1.5}
\begin{tabular}{c|c c c c c}
$p_n(r)$ & $2$ & $3$ & $4$ & $5$ \\
\hline 
$r=0$  & $1-\rho_2(X)$ & $0$ & $(1-\rho_4(X))^2+C_{1,1}^4(X)$ & $0$\\
$1$  &  $0$ & $1-\rho_3(X)$ & $0$ & $(1-\rho_5(X))^2+C_{1,1}^5(X)$ \\
$2$  & $\rho_2(X)$ & $0$ & $2\rho_4(X)(1-\rho_4(X))-2C_{1,1}^4(X)$ & $0$ \\
$3$  &  & $\rho_3(X)$  &  $0$ & $2\rho_5(X)(1-\rho_5(X))-2C_{1,1}^5(X)$  \\
$4$  &  &  & $\rho_4(X)^2+C_{1,1}^4(X)$ & $0$\\
$5$  & &  & & $\rho_5(X)^2+C_{1,1}^5(X)$
\end{tabular}
\caption{Values of $p_n(r) = \operatorname{Prob}(E\in \wrr_r^X\ |\ E\in \wss_{n}^X)$
for $n=2,\ldots,5$ and $0\leq r \leq n$. Note that $p_1(0)=0$ and $p_1(1)=1$.}
\label{tab-probabilities}
\end{table}
\end{cor}
\begin{proof}
If $E=(X,n,\Sel_2)$ is a test elliptic curve with Selmer rank $n=1$, then $\Sel_2=()$ is empty, and $\rank(E)=0$ by definition. Thus, $p_1(0)=0$ and $p_1(1)=1$. When $n=2$ or $3$, then $\Sel_2(E)=(s_{E,1})$, so there is a unique random variable $Y_1(\Sel_2(E))=Y_{\text{Hasse},n,X}(s_{E,1})$, with mean $\rho_n(X)$, such that $\rank_{1,X} = Y_1$. It follows that $p_n(n)=\rho_n(X)$ and $p_n(n-2)=1-\rho_n(X)$, and $p_n(r)=0$ for $r\neq n-2,n$.

 Finally, if $n=4,5$, then $\rank_{n,X}=Y_1+Y_2$, and Corollary \ref{cor-hasse1} says that
 $$Y_{\rk=n} = Y_1Y_2,\ Y_{\rk=n-2}=Y_1(1-Y_2)+(1-Y_1)Y_2,\ Y_{\rk=n-4}=(1-Y_1)(1-Y_2),$$
 with expected value, respectively, given by
 \begin{eqnarray*}
 \mathbb{E}(Y_{\rk=n})&=&\mathbb{E}(Y_1)\mathbb{E}(Y_2)+\operatorname{Cov}(Y_1,Y_2)=\rho_n(X)^2+C_{1,1}^n(X),\\
 \mathbb{E}(Y_{\rk=n-2})&=&\mathbb{E}(Y_1)(1-\mathbb{E}(Y_2))-\operatorname{Cov}(Y_1,Y_2)+(1-\mathbb{E}(Y_1))\mathbb{E}(Y_2)-\operatorname{Cov}(Y_1,Y_2)\\
 &=&2\rho_n(X)(1-\rho_n(X))-2C_{1,1}^n(X),\\ 
  \mathbb{E}(Y_{\rk=n-4})&=&(1-\mathbb{E}(Y_1))(1-\mathbb{E}(Y_2))+\operatorname{Cov}(Y_1,Y_2)=(1-\rho_n(X))^2+C_{1,1}^n(X),
\end{eqnarray*}
where we have used the equality $\mathbb{E}(Y_1Y_2)=\mathbb{E}(Y_1)\mathbb{E}(Y_2)+\operatorname{Cov}(Y_1,Y_2)$ and the fact that the covariance satisfies  $\operatorname{Cov}(a+bY_1,c+dY_2)=bd\operatorname{Cov}(Y_1,Y_2)$ for constants $a,b,c,d$.
\end{proof}

\begin{remark}
The formulas for the expected value of $Y_{\rk=n-2j}$ for $n\geq 6$, unfortunately, cannot be written just in terms of $\mathbb{E}(Y_i)$ and $C_{1,1}^n(X)=\operatorname{Cov}(Y_i,Y_j)$ for $i\neq j$. One needs to know other higher moments of the random variables $Y_i$. For instance, let $n=6$. Then,
\begin{eqnarray*}\mathbb{E}(Y_{\rk=6}) &=&\mathbb{E}(Y_1Y_2Y_3)=\mathbb{E}(Y_1Y_2)\mathbb{E}(Y_3)+\operatorname{Cov}(Y_1Y_2,Y_3)\\
&=&\mathbb{E}(Y_1)\mathbb{E}(Y_2)\mathbb{E}(Y_3)+\operatorname{Cov}(Y_1,Y_2)\mathbb{E}(Y_3)+\operatorname{Cov}(Y_1Y_2,Y_3)\\
&=& \rho_6(X)^3+C_{1,1}^6(X)\rho_6(X)+C_{2,1}^6(X).
\end{eqnarray*}
The formulae for $\mathbb{E}(Y_{\rk=6-2j})$ can be written in terms of the functions $\rho_6(X)$, $C_{1,1}^6$, and $C_{2,1}^6(X)$. For example,
\begin{eqnarray*}\mathbb{E}(Y_{\rk=4}) &=&\mathbb{E}(Y_1Y_2(1-Y_3))+\mathbb{E}(Y_1(1-Y_2)Y_3)+\mathbb{E}((1-Y_1)Y_2Y_3)=3\cdot \mathbb{E}(Y_1Y_2(1-Y_3))\\
&=& 3(\mathbb{E}(Y_1Y_2)\mathbb{E}(1-Y_3)+\operatorname{Cov}(Y_1Y_2,1-Y_3))\\
&=& 3(\mathbb{E}(Y_1Y_2)(1-\mathbb{E}(Y_3))-\operatorname{Cov}(Y_1Y_2,Y_3))\\
&=& 3((\mathbb{E}(Y_1)\mathbb{E}(Y_2)+\Cov(Y_1,Y_2))(1-\mathbb{E}(Y_3))-\operatorname{Cov}(Y_1Y_2,Y_3))\\
&=& 3(\rho_6(X)^2(1-\rho_6(X))+C_{1,1}^6(X)(1-\rho_6(X))-C_{2,1}^6(X)).
\end{eqnarray*}
\end{remark}

We conclude this section with an estimate of the growth (decay) of the expected values $\mathbb{E}(Y_{\operatorname{rk} = n-2j})$.

\begin{lemma} For each $0\leq j \leq \lfloor n/2 \rfloor$  we have that 
$$\mathbb{E}(Y_{\operatorname{rk} = n-2j}) = \binom{\lfloor n/2 \rfloor}{j}\rho_{n}(X)^{\lfloor n/2 \rfloor-j}(1-\rho_{n}(X))^{j}= \begin{cases}
O\left(X^{-f_n\cdot (\lfloor n/2\rfloor -j)}\right) & \text{ if } j\neq \lfloor n/2 \rfloor,\\
1+O\left(X^{-f_n}\right) & \text{ if } j=\lfloor n/2 \rfloor.
\end{cases}$$
\end{lemma}
\begin{proof} Since $\rho_n(X) = O(X^{-f_n})$, we have for $0\leq j \leq \lfloor n/2 \rfloor$ that 
	\begin{align*} \mathbb{E}(Y_{\operatorname{rk} = n-2j}) &= \binom{\lfloor n/2 \rfloor}{j}\rho_{n}(X)^{\lfloor n/2 \rfloor-j}(1-\rho_{n}(X))^{j}\\
	& = \binom{\lfloor n/2 \rfloor}{j}O(X^{-f_n})^{\lfloor n/2 \rfloor-j}(1-O(X^{-f_n}))^{j} \\
	&= \begin{cases}
	O\left(X^{-f_n\cdot (\lfloor n/2\rfloor -j)}\right) & \text{ if } j\neq \lfloor n/2 \rfloor,\\
	1-O\left(X^{-f_n}\right) & \text{ if } j=\lfloor n/2 \rfloor.
	\end{cases} 
	\end{align*} 
\end{proof}

\subsection{Refining Hypothesis B}\label{sec-refineb}

 As we did for Hypothesis A, in order to refine Hypothesis B, we shall use the sequence $(\E^X)_{X\geq 1}$ of ordinary elliptic curves as a representative of $\TT$ (see Remarks \ref{rem-testellipticcurve} and \ref{rem-actualellipticcurves}) to come up with candidates for our functions $\rho_n(X)$.  In order to estimate the values of $\rho_n(X)$, we define the following moving ratio measuring the failure of the Hasse principle for $2$-Selmer elements coming from elliptic curves of Selmer rank $n$ and up to height $X$.

\begin{defn}\label{defn-rhonx}
Let $\wtt \in {\bf T}$ be an arbitrary sequence. For each $n\geq 2$, and $N\geq 0$, we define the average failure of the Hasse principle for test Selmer elements in the height interval $(X,X+N]$ by
\begin{eqnarray*}
\rho_n(X,N) &=& 
\frac{\sum_{E\in\wts_n((X,X+N])} (\rank(E)-(n\bmod 2))}{\sum_{E\in \wts_n((X,X+N])} (\selrank(E)-(n\bmod 2))} \\
&=& \frac{\sum_{E\in\wts_n((X,X+N])} (\rank(E))-(n\bmod 2))}{(n-(n\bmod 2))\cdot \pi_{\wts_n}((X,X+N])}.
\end{eqnarray*}
\end{defn}

\begin{cor}\label{cor-travelratio2}
Let $\wtt \in {\bf T}$ be an arbitrary sequence, and assume Hypothesis B. Then, the expected value of $\rho_n(X,N)$ is given by $\rho_n(X)+O(X^{-f_n})$. Moreover, the standard error is given by
\begin{eqnarray*}  \sqrt{ \frac{\rho_n(X)(1-\rho_n(X))+ (\lfloor n/2 \rfloor -1)C_{1,1}^n(X) + O(X^{-f_n})}{\pi_{\wts}((X,X+N])}}\end{eqnarray*}
where $C_{1,1}^n(X)=0$ for $n=2,3$, and we assume here that $C_{1,1}^n(X)=c_n+O(X^{-f_n})$ for some constant $c_n$. 
\end{cor}
\begin{proof}
The expected value of $\rho_n(X,N)$ is given by
\begin{align*}
\mathbb{E}(\rho_n(X,N)) &= \mathbb{E}\left(\frac{\sum_{E\in\wts_n((X,X+N])} (\rank(E))-(n\bmod 2))}{(n-(n\bmod 2))\cdot \pi_{\wts_n}((X,X+N])}\right)\\
& = \frac{1}{\lfloor n/2 \rfloor} \cdot \mathbb{E}\left(\sum_{H=X+1}^{X+N} \sum_{E\in\wts_n((H,H])} \frac{1}{\pi_{\wts_n}((X,X+N])} \cdot \frac{\rank(E)-(n\bmod 2)}{2}  \right)\\
& = \sum_{H=X+1}^{X+N} \frac{\pi_{\wts_n}((H,H])}{\pi_{\wts_n}((X,X+N])} \cdot \rho_n(H).
\end{align*}
by Theorem \ref{thm-ranksbinomial}.  Putting $\alpha_H = \pi_{\wtt}([H,H])/\pi_{\wtt}((X,X+N])$, we note that $\sum_{H=X+1}^{X+N} \alpha_H = 1$. Since $\rho_n(X) =  O(X^{-f_n})$ by Hypothesis \ref{hypb}, now Lemma \ref{lem-sumofthetas} shows that $\mathbb{E}(\rho_n(X,N))=\rho_n(X) + O(X^{-f_n})$, as claimed. 

$$\operatorname{Var}\left(\frac{1}{m}\sum_{i=1}^m Y_{\Sel,n,X_i}(E_i) \right) = \frac{1}{m^2}\sum_{i=1}^m \operatorname{Var}(Y_{\Sel,n,X_i}(E_i))=\frac{1}{m^2}\sum_{i=1}^m \theta_n(X_i)(1-\theta_n(X_i)),$$

Similarly, the variance $\operatorname{Var}(\rho_n((X,X+N]))$ is given by
\begin{eqnarray*}
	 &=& \frac{1}{(\pi_{\wts}((X,X+N]))^2}\sum_{E\in\wts((X,X+N])} \operatorname{Var}\left(\frac{1}{\lfloor n/2 \rfloor} \cdot \frac{\rank(E)-(n\bmod 2)}{2}\right)\\
	&=& 
	\frac{1}{(\pi_{\wts}((X,X+N]))^2}\sum_{E\in\wts((X,X+N])} \rho_n(\h(E))(1-\rho_n(\h(E)))+(\lfloor n/2 \rfloor -1)C_{1,1}^n(\h(E))\\
	&=& \frac{1}{\pi_{\wts}((X,X+N])}\sum_{H=X+1}^{X+N} \frac{\pi_{\wts}([H,H])}{\pi_{\wts}((X,X+N])}\cdot  (\rho_n(H)(1-\rho_n(H))+(\lfloor n/2 \rfloor -1)C_{1,1}^n(H))\\
	&=&  \frac{1}{\pi_{\wts}((X,X+N])}\cdot (\rho_n(X)\cdot (1-\rho_n(X))+(\lfloor n/2 \rfloor -1)C_{1,1}^n(X)) + O(X^{-f_n})),
\end{eqnarray*}
by Lemma \ref{lem-sumofthetas} and Corollary \ref{cor-hasseaverank}, and we have assumed that $C_{1,1}^n(X)=c_n+O(X^{-f_n})$ for some constant $c_n$. Thus, the standard error is as claimed in the statement. 
\end{proof}

We have used the BHKSSW data to estimate probability function $\rho_n(X)$ using the  moving ratios $\displaystyle \rho_n(X,N)$ of Corollary \ref{cor-travelratio2}. We have plotted values of  $\rho_n(X,0.25\cdot 10^9)$ for $n=2,\ldots,5$ using the BHKSSW database, and the graphs can be found in Figure \ref{fig-rhonx}. Recall that we have assumed in $H_B$ that $\lim_{X\to \infty} \rho_n(X)=0$. We will see below that a ``best-fit'' model of the data and graphs depicted here support this assumption, though the convergence to $0$ is very slow.

\begin{figure}[h!]
\includegraphics[width=6.6in]{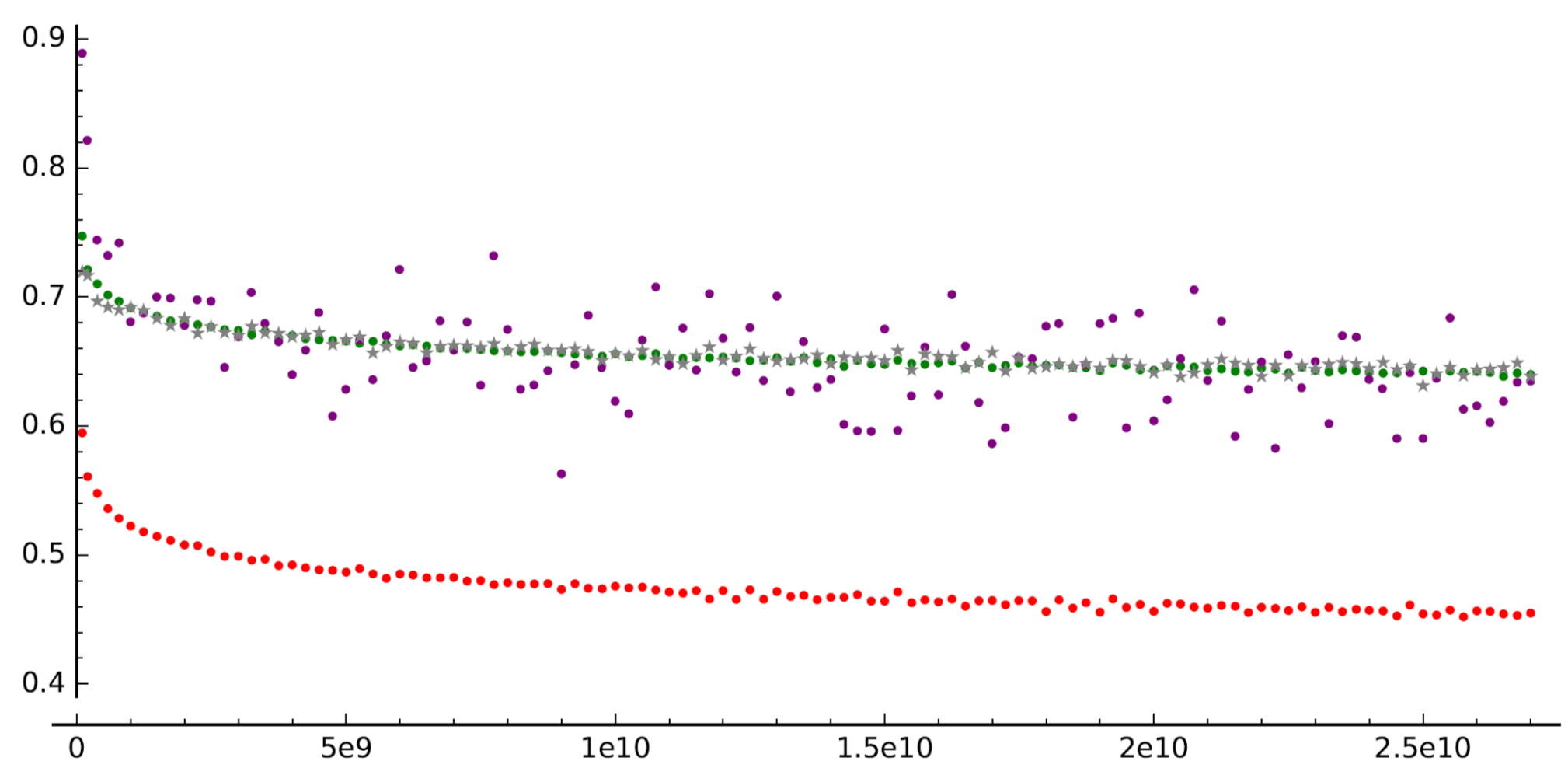}
\caption{Graphs of the moving ratios $\rho_n(X,0.25\cdot 10^9)$ for $n=2$ (green), $3$ (red), $4$ (gray stars), $5$ (purple).}
\label{fig-rhonx}
\end{figure}
In Table \ref{tab-hasseave} we record the last values of $\rho_n(X,0.25\cdot 10^9)$ that appear in the graphs (which correspond to $X\approx 2.675\cdot 10^{10}$). We also record the values of $\pi_{\s_n}$ in $[2.675\cdot 10^{10},2.7\cdot 10^{10}]$. The total number of elliptic curves in the same interval is $\num[group-separator={,}]{1828235}$.

\begin{table}[h!]
\centering
\def\arraystretch{1.5}
\begin{tabular}{c||c|c|c|c}
\hline
$n$ & $2$ & $3$ & $4$ & $5$\\
\hline 
$\pi_{\s_n}([2.675\cdot 10^{10},2.7\cdot 10^{10}])$ & $\num[group-separator={,}]{476579}$ & $\num[group-separator={,}]{104922}$ & $\num[group-separator={,}]{7945}$ & $152$\\
\hline 
$\rho_n(2.675\cdot 10^{10},0.25\cdot 10^9)$ & 
 $0.63989181$ & $0.45496654$ & $0.63857772$ & $0.63486842$\\
 \hline 
\end{tabular}
\caption{The number of curves of Selmer rank $2\leq n\leq 5$, and the values of $\rho_n(X,N)$ in the interval $[2.675\cdot 10^{10},2.7\cdot 10^{10}]$.}
\label{tab-hasseave}
\end{table} 

Finally, we have found (using SageMath) best-fit models for the data of $\rho_n(X,N)$ of the form 
$$\rho_n(X,N) \approx \frac{D_n}{X^{f_n}}.$$
and we provide the values of $D_n$ and $f_n$ in Table \ref{tab-rhomodel}. We have compared the models with the data in Figure \ref{fig-rhonx2}. 

\begin{table}[h!]
\centering
\def\arraystretch{1.5}
\begin{tabular}{c||c|c|c|c}
\hline
$n$ &   $2$ & $3$ & $4$ & $5$\\
\hline 
$D_n$ &  $1.12465347$ & $1.30937016$ & $1.07928016$ & $1.79161787$\\
 \hline
 $f_n$ & $0.02344245$ & $0.04412662$ & $0.02158211$ & $0.04383626$ \\
 \hline 
\end{tabular}
\caption{The coefficients of the best-fit models $\rho_n(X,N)\approx D_n/X^{f_n}$.}
\label{tab-rhomodel}
\end{table} 

\begin{figure}[h!]
\includegraphics[width=6.6in]{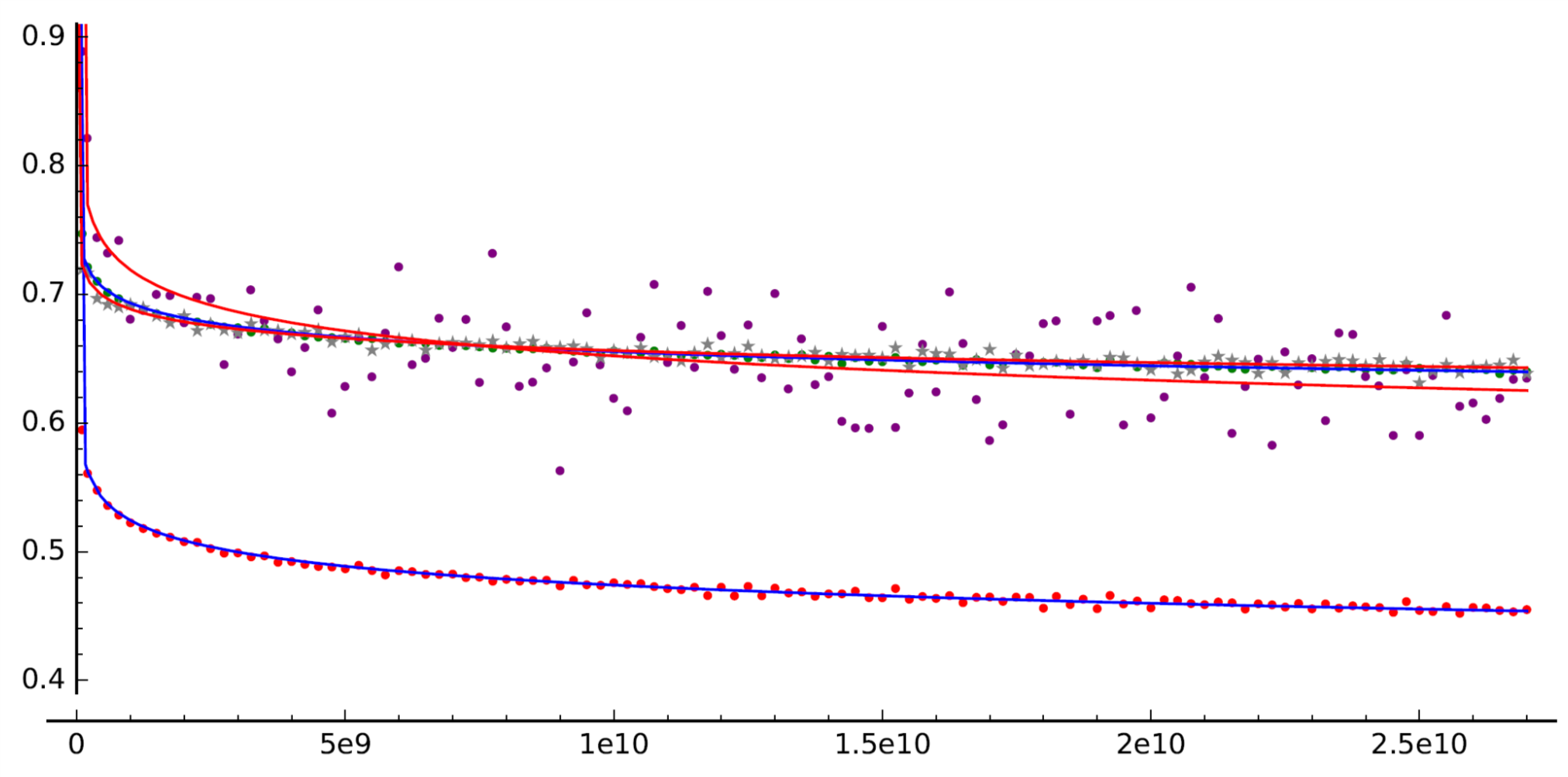}
\caption{Graphs of the moving ratios $\rho_n(X,0.025\cdot 10^9)$ for $n=2$ (green), $3$ (red), $4$ (gray stars), $5$ (purple), and the corresponding models of the form $D_n/X^{f_n}$ (in blue for $n=2,3$ and red for $n=4,5$).}
\label{fig-rhonx2}
\end{figure}

\begin{hypa}[Hypothesis $H_B'$]\label{conj-hasseratio}
Hypothesis $H_B$ holds and, for every $n\geq 2$, there are constants $D_n$ and $f_n$ such that $\displaystyle \rho_n(X)= \frac{D_n}{X^{f_n}}.$ Moreover, for $n=2,\ldots,5$ the values of  $D_n$ and $f_n$ are approximately as given in Table \ref{tab-rhomodel}.  In addition, for $i,j\geq 0$ (not both zero), we shall assume that $\mathbb{E}_{i,j}^n(X)$  is a continuous function with $$\mathbb{E}_{i,j}^n(X) = \begin{cases}
O(X^{-f_n\cdot i}) & \text{ if } i>0,\\
1-O(X^{-f_n}) & \text{ if } i=0 \text{ and } j>0.
\end{cases} $$
Further, we shall assume that there is a constant $C_{\mathbb{E}}>0$ that does not depend on $n$, such that $|\frac{d}{dX}(\mathbb{E}_{i,j}^n(X))| \leq C_{\mathbb{E}}/X^{f_n+1}$.
\end{hypa}

\begin{remark} \label{rem-growthcovariance}
In particular, by Lemma \ref{lem-equicorr3}, we have  $$\mathbb{E}_{1,1}^n (X) = \mathbb{E}_{1,0}^n (X)\cdot \mathbb{E}_{0,1}^n(X) -C_{1,1}^n(X) = \rho_n(X)(1-\rho_n(X)) - C_{1,1}^n(X),$$ and so, if we assume Hypothesis \ref{conj-hasseratio}, then
$$C_{1,1}^n(X) = \mathbb{E}_{1,1}^n (X) - \rho_n(X)(1-\rho_n(X)) = O(X^{-f_n}).$$
\end{remark} 

\begin{remark} \label{rem-covariance}
Before we can discuss the standard error in the approximation $\rho_n(X,N)\approx\rho_n(X)$ we need to estimate the covariance functions $C_{s,t}^n(X)$. This can be done via the formulas for the expected value of $Y_{\operatorname{rk} = n-2j}$ given by Corollary \ref{cor-hasse1} and, for $n=1,2,3,4,5$, the simplified formulas given by Corollary \ref{cor-rankprobabilities}. The first thing to note is that for $n=1,2,3$, we have $C_{s,t}^n(X)=0$ for all possible values of $s,t$ since there is either none ($n=1$) or only one random variable $Y_1$ that intervenes ($n=2,3$). For $n=4$ and $5$ there are two random variables $Y_1$ and $Y_2$ and
$$C_{1,1}^n(X) =  \mathbb{E}(Y_{\operatorname{rk} = n}) - \mathbb{E}(Y_1)\mathbb{E}(Y_2)=\mathbb{E}(Y_{\operatorname{rk} = n}) - \rho_n(X)^2. $$
In Figures \ref{fig-cov4} and \ref{fig-cov5} we have plotted approximate covariance values of $C_{1,1}^4(X)$ and $C_{1,1}^5(X)$, respectively, using sample height intervals $[X,X+0.25\cdot 10^9]$, together with the best linear fits for the data which are given by
$$-0.02677186+(4.06113344\cdot 10^{-14})x\ \text{ and }\ -0.01328180+(1.28980002\cdot 10^{-12})x,$$
respectively. In particular, we observe that $|C_{1,1}^4(X) -(- 0.025)|\lessapprox 0.015$ and $|C_{1,1}^5(X)-0|\lessapprox 0.1$. Thus, below, we will approximate $C_{1,1}^4(X)\approx -0.025$ and $C_{1,1}^5(X)\approx 0$.
\begin{figure}[h!]
\includegraphics[width=6.6in]{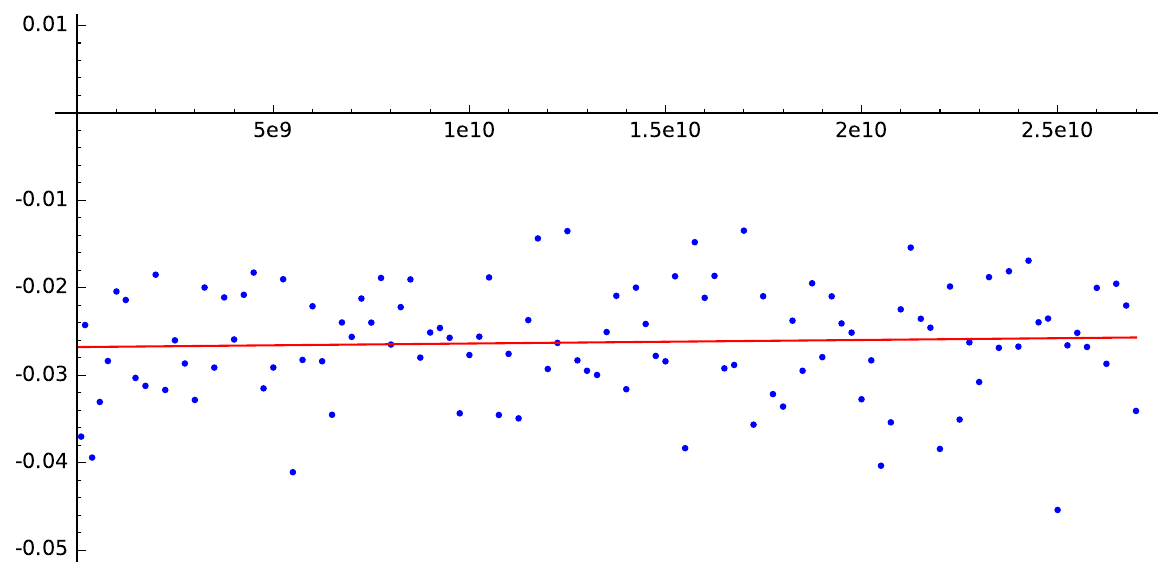}
\caption{Approximate values of $C_{1,1}^4(X)$ using sample height intervals $[X,X+0.25\cdot 10^9]$ to estimate $\mathbb{E}(Y_{\operatorname{rk} = 4}) - \rho_4(X)^2$ using the family of elliptic curves. The best-fit line is given by $-0.02677186+(4.06113344\cdot 10^{-14})x$.}
\label{fig-cov4}
\end{figure}

\begin{figure}[h!]
\includegraphics[width=6.6in]{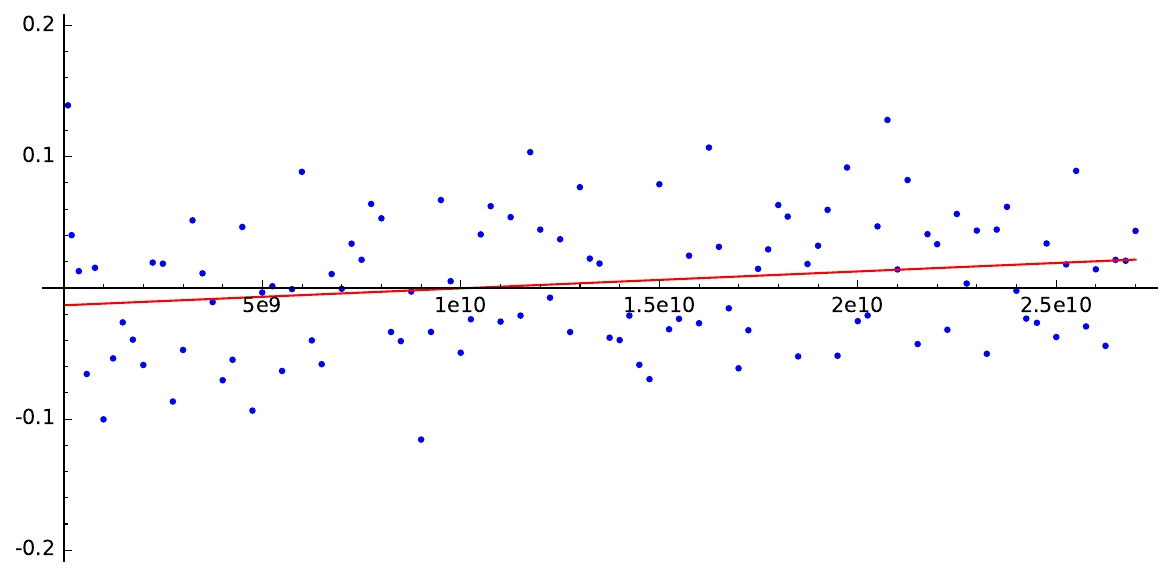}
\caption{Approximate values of $C_{1,1}^5(X)$ using sample height intervals $[X,X+0.25\cdot 10^9]$ to estimate $\mathbb{E}(Y_{\operatorname{rk} = 5}) - \rho_5(X)^2$ using the family of elliptic curves. The best-fit line is given by $-0.01328180+(1.28980002\cdot 10^{-12})x$.}
\label{fig-cov5}
\end{figure}
\end{remark} 

\begin{remark}
Let us assume Hypothesis \ref{conj-hasseratio}, and let us use Corollary \ref{cor-travelratio2} to estimate the standard error in the approximation $\rho_n(X)\approx \rho_n(X,N)$. The error should be given by the expression
$$\text{err}_{1,n}(X,N)=\sqrt{\frac{ \rho_n(X)(1-\rho_n(X))+ (\lfloor n/2 \rfloor -1)C_{1,1}^n(X)}{\pi_{\s_n}((X,X+N])}}$$
or by the expression
$$\text{err}_{2,n}(X,N)=   \sqrt{\cfrac{6X^{1/6}(1+C_nX^{-e_n})(\rho_n(X)(1-\rho_n(X))+ (\lfloor n/2 \rfloor -1)C_{1,1}^n(X))}{5\kappa s_n N} },$$
if we assume $H_A$ and Hypothesis \ref{conj-selmerratio} also. Using our calculations of Remark \ref{rem-covariance}, we will take $C_{1,1}^n(X)=0$ for $n=2,3$, and $C_{1,1}^4(X)=-0.025$, and $C_{1,1}^5(X)=0$.  In Table \ref{tab-errors2} we include the values of: $\rho_n(X,N)$, our model of $\rho_n(X)$, the error of the model $|\rho_n(X,N)-\rho_n(X)|$, and the predicted standard errors $\text{err}_{i,n}(X,N)$, for $i=1,2$, and $X=2.675\cdot 10^{10}$, with $N=0.25\cdot 10^9$.
\begin{table}[h!]
\centering
\def\arraystretch{1.5}
\begin{tabular}{c||c|c|c|c}
\hline 
$n$ & $2$ & $3$ & $4$ & $5$\\
\hline 
$\pi_{\s_n}([2.675\cdot 10^{10},2.7\cdot 10^{10}])$ & $\num[group-separator={,}]{476579}$ & $\num[group-separator={,}]{104922}$ & $\num[group-separator={,}]{7945}$ & $152$\\
\hline 
$\rho_n(2.675\cdot 10^{10},0.25\cdot 10^9)$ & 
 $0.63989181$ & $0.45496654$ & $0.63857772$ & $0.63486842$\\
 \hline
 $\rho_n(2.675\cdot 10^{10})$  & $0.63996477$ & $0.45404630$ & $0.64309203$ & $0.62550968$ \\
 \hline 
 $|\text{Error}|=|\rho_n(X,N) - \rho_n(X)|$  & $0.00007296$ & $0.00092023$ & $0.00451431$ & $0.00935873$\\
 \hline 
 $\text{err}_{1,n}(2.675\cdot 10^{10},0.25\cdot 10^9)$  & $0.00069531$ & $0.00153707$ & $0.00717531$ & $0.05551757$\\
 \hline 
 $\text{err}_{2,n}(2.675\cdot 10^{10},0.25\cdot 10^9)$  & $0.00069208$ & $0.00152440$ & $0.00700827$ & $0.05462609$\\
  \hline
\end{tabular}
\caption{Values of: $\rho_n(X,N)$, our model of $\rho_n(X)$, the error $|\rho_n(X,N)-\rho_n(X)|$, and the two predicted standard errors $\text{err}_{i,n}(X,N)$, for $i=1,2$, and $X=2.675\cdot 10^{10}$, $N=0.25\cdot 10^9$.}
\label{tab-errors2}
\end{table}
\end{remark} 

\begin{remark} 
As we can see from the errors in Table \ref{tab-errors2}, we seem to have insufficient data for $n=5$, so our models of $\rho_5(X)$ are not as accurate as we would wish. Indeed, the estimated standard error is $\approx 0.05$ and the value of $\rho_n(2.675\cdot 10^{10})\approx 0.62$, so the error here is about $10\%$.
\end{remark}

\begin{remark}
As we have mentioned earlier in Remark \ref{rem-sellargeheight} the BHKSSW database (\cite{BHKSSW}) also includes small databases of random samples of elliptic curves at larger heights. In order to test $H_B$ and Hypothesis \ref{conj-hasseratio}, we have calculated the average Hasse ratio for the curves in $\E_k$ (with notation as in Remark \ref{rem-sellargeheight}), and have plotted the ratios together with our models for $\rho_n(X)$, in Figure \ref{fig-rhonx3} (note: the $x$-axis is in logarithmic scale). We have also computed the predicted errors (a calculation similar to that carried out in Table \ref{tab-errors2}) and the predictions seem to match the data in large heights, as well.
\begin{figure}[h!]
\includegraphics[width=6.6in]{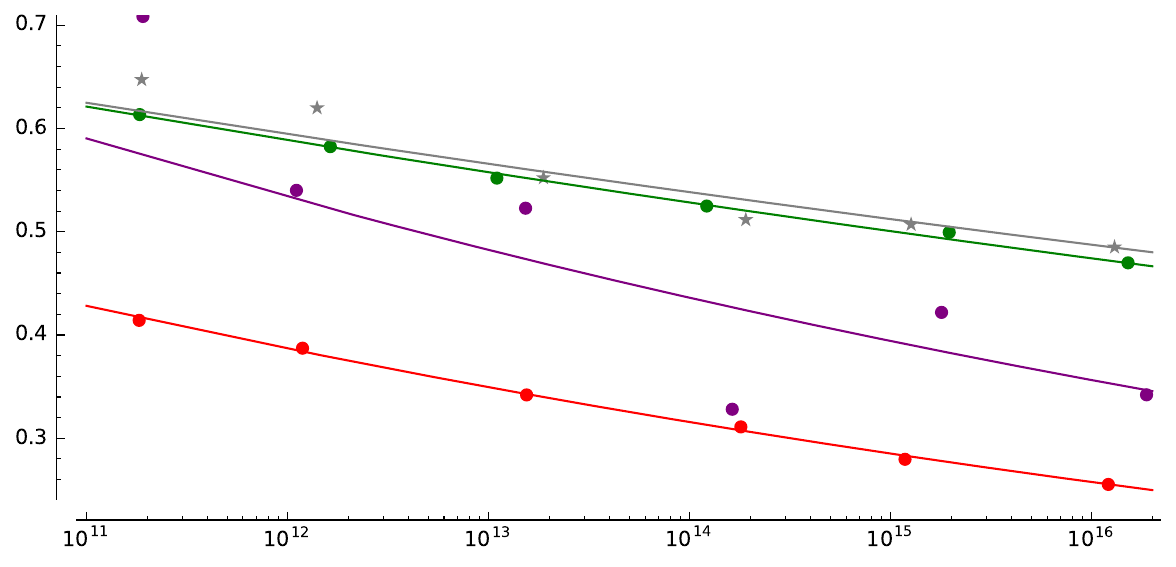}
\caption{Graphs of the moving ratios $\rho_n(X,N)$ for the curves of large height in the database BHKSSW, for $n=2$ (green), $3$ (red), $4$ (gray stars), $5$ (purple), and the corresponding models of the form $D_n/X^{f_n}$. The $x$-axis is in logarithmic scale.}
\label{fig-rhonx3}
\end{figure}
\end{remark}

\begin{remark} \label{rem-j1728}
It would be interesting to compute the ratio $\rho_n(X)$ in families of quadratic twists. However, such families are very ``thin'' in the family of all elliptic curves, and the convergence of the Hasse ratios to $\rho_n(X)$ would be unreliable. In order to provide some data in this direction, we have calculated the Selmer rank and Mordell--Weil rank in a family of twists (quadratic and quartic) of $y^2=x^3+x$. More precisely, we consider the curves $E_A:y^2=x^3+Ax$, with fourth-power-free $1\leq A\leq 10^6$ (curves up to height $4\cdot 10^{18}$). Then, we have calculated the moving ratios $\rho_n$ in slices of $\num[group-separator={,}]{10000}$ curves, and graphed them against the models of Hypothesis \ref{conj-hasseratio}. See Figure \ref{fig-rhonx1728}. Note, however, that we do not expect the exact same behavior in this family, since $j(E_A/\Q)=1728$ is fixed, and therefore it is a family of twists (quadratic and quartic). It is likely that if $H_B$ holds, then a similar condition is true for $j=1728$ up to a constant. That is, we may have $\rho_{n,1728}(X)\approx C_{1728}\cdot \rho_n(X)$, where $C_{1728}$ is a fixed constant. At any rate, the family of curves with $j=1728$ is very sparse within the family of all curves, and the data only indicates some consistency with our expectations.

\begin{figure}[h!]
\includegraphics[width=6.6in]{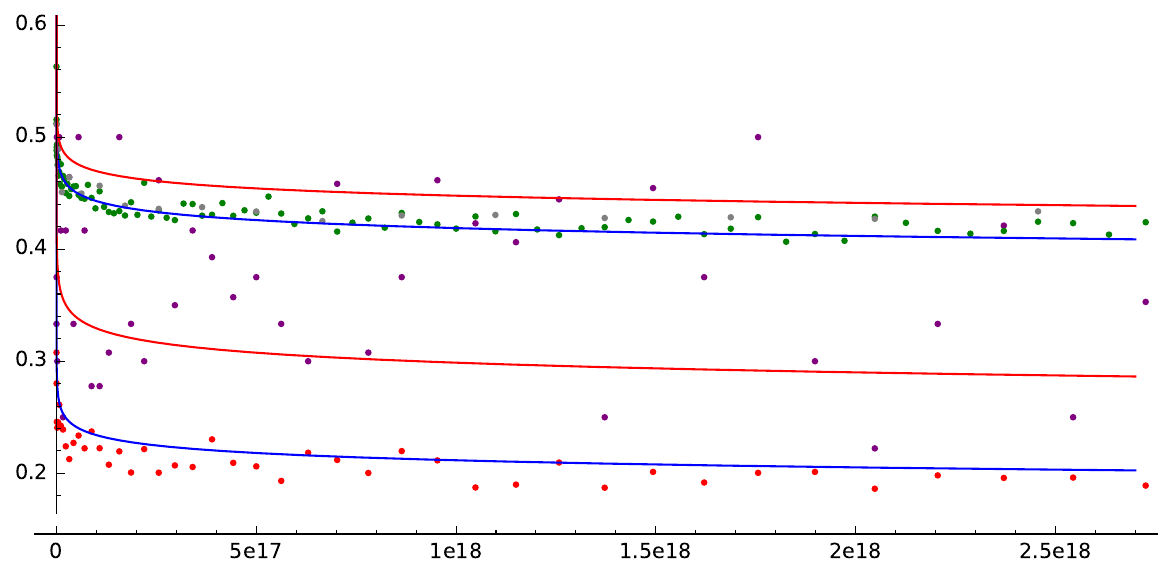}
\caption{Graphs of the moving ratios $\rho_n$ in the family $y^2=x^3+Ax$ for $n=2$ (green), $3$ (red), $4$ (gray), $5$ (purple), compared to the models of the form $D_n/X^{f_n}$ (in blue for $n=2,3$ and red for $n=4,5$).}
\label{fig-rhonx1728}
\end{figure}

\end{remark}

\begin{example}\label{ex-rankbinomial}
Theorem \ref{thm-ranksbinomial}, assuming $H_B$, provides the expected value and variance for the rank of an elliptic curve $E/\Q$ of Selmer rank $n$ and height $X$. More precisely, in Corollary \ref{cor-hasse1} and \ref{cor-rankprobabilities}, we give formulas for the probabilities for each rank. Now that we have models for $\rho_n(X)$ and $C_{1,1}^n(X)$ (as in Remark \ref{rem-covariance}), we can look at the distribution of ranks in intervals. Let us consider, for instance, the curves $\E(I)$ in the height interval $I=[20\cdot 10^9,20.25\cdot 10^9]$ in the BHKSSW database. For each $n=2,3,4,5$ we have created histograms using the number of curves of Selmer rank $n$ and Mordell--Weil rank $0\leq n$ (in blue bars), and also created histograms with the number of M--W ranks that we would expect from Corollary \ref{cor-rankprobabilities} (in green bars). The resulting histograms can be found in Figure \ref{fig-hist2} (together with the graph of the normal distribution that would approximate the binomial $B(\lfloor n/2 \rfloor,\rho_n(X))$. We have also included the raw data of ranks observed and ranks predicted in Table \ref{tab-hist2data}.

\begin{center}
\begin{figure}[h!]
\includegraphics[width=6.6in]{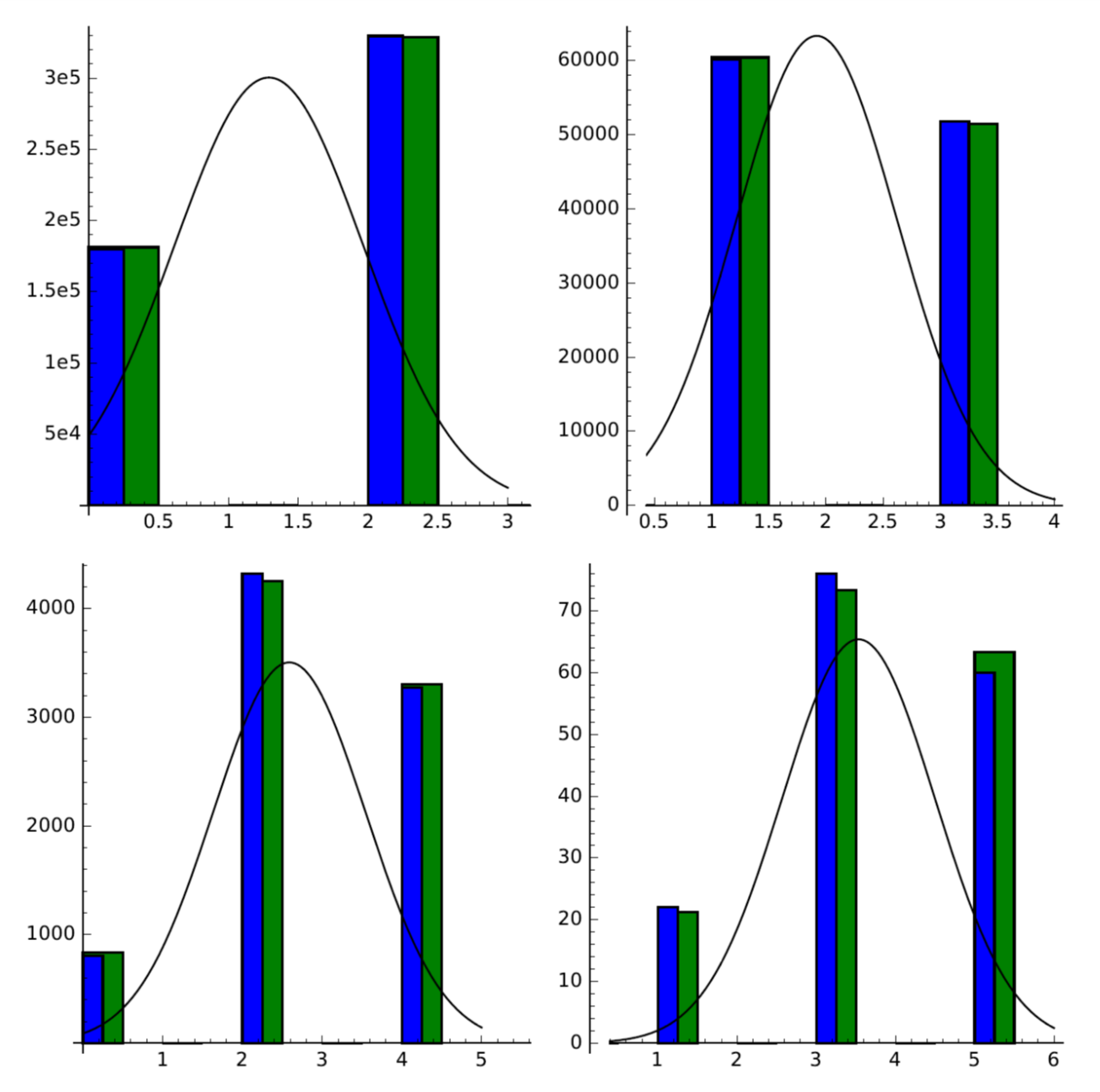}
\caption{Histogram (in blue) of the distribution of Mordell--Weil ranks among elliptic curves in $\E([2\cdot 10^{10},2.025\cdot 10^{10}])$ by Selmer rank $n=2,3,4,5$, and compared to the histogram (in green) of the  M--W ranks that we would expect from Theorem \ref{thm-ranksbinomial}. The graph is that of the normal distribution that best approximates the binomial.}
\label{fig-hist2}
\end{figure}
\end{center}

\begin{table}[h!]
\centering
\def\arraystretch{1.5}
\begin{tabular}{c|ccc}
$n$ & $\pi_{\s_n}([2\cdot 10^{10},2.025\cdot 10^{10}])$ & M--W ranks observed in $\s_n$ & M--W ranks predicted\\
\hline
$2$ & $\num[group-separator={,}]{509845}$ & $[180128,0,329717,0,0,0]$ & $[181246.58, 0, 328598.41,0,0,0]$ \\
$3$ & $\num[group-separator={,}]{111926}$ & $[0,60149,0,51777,0,0]$ & $[0, 60455.09, 0, 51470.90,0,0]$ \\
$4$ & $\num[group-separator={,}]{8399}$ & $[803,0,4321,0,3275,0]$ & $[836.68, 0, 4256.52, 0, 3305.78,0]$ \\
$5$ & $\num[group-separator={,}]{158}$ & $[0,22,0,76,0,60]$ & $[0, 21.24, 0, 73.38, 0, 63.36]$ \\
\end{tabular}
\caption{Mordell--Weil ranks observed in the interval height interval $[2\cdot 10^{10},2.025\cdot 10^{10}]$ and the ranks predicted by the distribution of Theorem \ref{thm-ranksbinomial}.}
\label{tab-hist2data}
\end{table}
\end{example}

\section{Predicting the number of curves with a given rank up to height $X$}\label{sec-predictrank}

Let $X,r\geq 0$ be fixed. We denote the set of elliptic curves of height $\leq X$ and Mordell--Weil rank $r$ by
\begin{eqnarray*}
\mathcal{R}_r(X) = \{ E \in \E(X) : \rank(E(\Q))= r \},
\end{eqnarray*}
and we write $\pi_{\mathcal{R}_r}(X)=\#\mathcal{R}_r(X)$. We refer the reader to Sections 3.3 and 3.4 of \cite{ppvm} for a summary of conjectures about $\pi_{\mathcal{R}_r}(X)$, but we point out two in particular:
\begin{itemize}
\item Watkins (\cite{watkins}; see also \cite{BMSW} for an expository paper) has conjectured that there is a constant $c$ such that
$$\sum_{k=1}^\infty \pi_{\mathcal{R}_{2k}}(X) = (c+o(1))X^{19/24}(\log X)^{3/8}.$$
\item Park, Poonen, Voight, and Wood (\cite{ppvm}) have developed a heuristic that predicts:
\begin{enumerate}
\item All but finitely many elliptic curves satisfy $\rank(E(\Q))\leq 21$.
\item For $1\leq r\leq 20$, we have $\sum_{k=r}^\infty \pi_{\mathcal{R}_k}(X) = X^{(21-r)/24+o(1)}$.
\item $\sum_{k=21}^\infty \pi_{\mathcal{R}_k}(X) \leq  X^{o(1)}$.
\end{enumerate} 
Note that prediction (3) follows from (1) and (2).
\end{itemize}

In this section, we denote the set of test elliptic curves of height $\leq X$ and rank $r$ by
\begin{eqnarray*}
	\wrr_r(X) = \{ E \in \wee(X) : \rank(E)= r \},
\end{eqnarray*}
and if $\wtt\in\mathbf{T}$, we write $\wtr_r(X)$ for the subset of test elliptic curves in $\wtt$ of rank $r$ and height $\leq X$. Then, we write $\pi_{\wtr_r}(X)=\#\wtr_r(X)$. We shall assume hypotheses $H_A$ and $H_B$, and derive the expected value of $\pi_{\wtr_r}(X)$ that follows from the probability distributions we have studied in previous sections. We shall study $\pi_{\wtr_r}(X)$ as the sum of the contributions of rank $r$ coming from each Selmer rank $n=r+2j$. That is, we shall approximate $\pi_{\wtr_r}(X)$ by approximating each term in the infinite sum
$$\pi_{\wtr_r}(X) = \sum_{j=0}^\infty \pi_{\wtr_r\cap \wss_{r+2j}}(X).$$
Thus, for fixed $r\geq 0$, we first give the expected value of $\pi_{\wtr_r\cap \wss_{r+2j}}(X)$ for each $j\geq 0$.

\begin{theorem}\label{thm-predictrank}
Let $X,r\geq 0,j\geq 0$ be fixed, such that $n(j)=r+2j\geq 2$. Let $\wtt\in\mathbf{T}$ be arbitrary. If we assume $H_A$ and $H_B$, then the expected value of  $\pi_{\wtr_r\cap \wss_{n(j)}}(X)$ is given by $$\mathbb{E}\left(\pi_{\wtr_r\cap\wss_{n(j)}} (X)\right) \myeq  \frac{5\kappa}{6}\binom{\lfloor \frac{r}{2} \rfloor+j}{j}\int_0^{X} \cfrac{\theta_{n(j)}(H)}{H^{1/6}}\cdot \mathbb{E}_{\lfloor \frac{r}{2} \rfloor,j}^{n(j)}(H)\, dH +  O\left(X^{1/2}\right),$$
where $\mathbb{E}_{\lfloor \frac{r}{2} \rfloor,j}^{n(j)}(H)$ is the expected value defined in Remark \ref{rem-notation}. 
Further, if we assume Hypothesis \ref{conj-selmerratio}, then 
$$\mathbb{E}\left(\pi_{\wtr_r\cap \wss_{n(j)}}(X)\right) \myeq \frac{5\kappa}{6}  \binom{\lfloor \frac{r}{2} \rfloor+j}{j} \int_{0}^X \cfrac{ s_{n(j)}\cdot \mathbb{E}_{\lfloor \frac{r}{2} \rfloor,j}^{n(j)}(H)}{(1+C_{n(j)}H^{-e_{n(j)}})\cdot H^{1/6}}\, dH +  O\left(X^{1/2}\right).$$
Moreover, in both cases the implied error in the approximation is bounded by $C\theta_{n(j)}^\text{sup}/X^{1/2}$, for some constant $C$ that does not depend on $n$ or $j$, and where $\theta_{n(j)}^\text{sup}$ is the supremum of $\theta_{n(j)}(X)$ in $(0,\infty)$.
\end{theorem} 
\begin{proof} 
Let us write $n(j)=r+2j$. Thus, $\lfloor \frac{n(j)}{2} \rfloor = \lfloor \frac{r}{2} \rfloor +j$. We compute the expected value of $\pi_{\wtr_r}(X)$ as follows:
\begin{eqnarray*} \mathbb{E}\left(\pi_{\wtr_r\cap\wss_{n(j)}} (X)\right) &=&  \mathbb{E}\left(\#\{E\in \wts_{n(j)}(X) : \rank(E)=r \}\right)\\
&=&  \sum_{T=1}^X \mathbb{E}\left(\#\{E\in \wts_{n(j)}([T,T]) : \rank(E)=r \}\right)\\
&=&  \sum_{T=1}^X \mathbb{E}\left(\pi_{\wts_{n(j)}}([T,T])\right)\cdot \operatorname{Prob}\left(\rank(E)=r\ |\ E\in \wss_{r+2j}^T\right)\\
&=& \sum_{T=1}^X \pi_{\wtt_{n(j)}}([T,T])\cdot \theta_{n(j)}(H)\cdot \operatorname{Prob}\left(\rank(E)=r\ |\ E\in \wss_{r+2j}^T\right)\\
&=& \sum_{T=1}^X \pi_{\wtt_{n(j)}}([T,T])\cdot \theta_{n(j)}(H)\cdot \binom{\lfloor \frac{n(j)}{2} \rfloor}{j}\cdot  \mathbb{E}_{\lfloor \frac{r}{2} \rfloor,j}^{n(j)}(H),
\end{eqnarray*}
by the basic properties of the expected value, and Corollary \ref{cor-hasse1} for the probability of rank $r$ in $\wss_{n(j)}([T,T])$. Let $\theta(X) =\theta_{n(j)}(X)\cdot   \mathbb{E}_{\lfloor \frac{r}{2} \rfloor,j}^{n(j)}(X)$. Thus,
\begin{align*} \theta(X) &=\theta_{n(j)}(X)\cdot   \mathbb{E}_{\lfloor \frac{r}{2} \rfloor,j}^{n(j)}(X) =(s_{n(j)} + O(X^{-e_n}))\cdot  \begin{cases}
O(X^{-f_{n(j)}\cdot i}) & \text{ if } i>0,\\
1-O(X^{-f_{n(j)}}) & \text{ if } i=0 \text{ and } j>0.
\end{cases} \\
& = \begin{cases}
O(X^{-f_{n(j)}\cdot \lfloor \frac{r}{2} \rfloor}) & \text{ if } \lfloor \frac{r}{2} \rfloor>0,\\
s_{n(j)} +O(X^{-\min\{e_{n(j)}, f_{n(j)}  \}}) & \text{ if } \lfloor \frac{r}{2} \rfloor=0 \text{ and } j>0,
\end{cases} 
\end{align*}
and 
\begin{align*} |\theta(X)'| &= |(\theta_{n(j)}(X)\cdot  \mathbb{E}_{\lfloor \frac{r}{2} \rfloor,j}^{n(j)}(X))'|=|\theta_{n(j)}(X)'\cdot  \mathbb{E}_{\lfloor \frac{r}{2} \rfloor,j}^{n(j)}(X)+\theta_{n(j)}(X)\cdot  \mathbb{E}_{\lfloor \frac{r}{2} \rfloor,j}^{n(j)}(X)'| \\
&\leq \frac{\theta_{n(j)}^{\text{sup}}\cdot C_e}{X^{e_{n(j)}+1}} + \frac{\theta_{n(j)}^{\text{sup}}\cdot C_{\mathbb{E}}}{X^{f_{n(j)}+1}}\leq \frac{\theta_{n(j)}^{\text{sup}}\cdot \max\{C_e,C_{\mathbb{E}} \}}{X^{\min\{e_{n(j)},f_{n(j)} \}+1}},
\end{align*}
 where we have used Remark \ref{rem-derivativeoftheta}, and the growth assumptions on $\mathbb{E}_{i,j}^n(X)$ and its derivative imposed by Hypothesis \ref{conj-hasseratio}. Thus, Corollary \ref{cor-travelratio} for the value of $\pi_{\wts_{n(j)}}([H,H])$ (on average) and Corollary \ref{cor-countcurvesaverage1} show that 
\begin{eqnarray*}\mathbb{E}\left(\pi_{\wtr_r\cap\wss_{n(j)}} (X)\right) & \myeq & \sum_{T=0}^{X-1} \left( \frac{5\kappa}{6}\int_T^{T+1} \cfrac{\theta_{n(j)}(H)}{H^{1/6}}\cdot \binom{\lfloor \frac{n(j)}{2} \rfloor}{j}\cdot  \mathbb{E}_{\lfloor \frac{r}{2} \rfloor,j}^{n(j)}(H)\, dH +  O\left(\frac{1}{T^{1/2}}\right) \right)\\
&\myeq &  \frac{5\kappa}{6}\binom{\lfloor \frac{r}{2} \rfloor+j}{j}\int_0^{X} \cfrac{\theta_{n(j)}(H)}{H^{1/6}}\cdot \mathbb{E}_{\lfloor \frac{r}{2} \rfloor,j}^{n(j)}(H)\, dH + O\left(X^{1/2}\right).
\end{eqnarray*}
where the implied error in the approximation is bounded by $C\cdot \theta_{n(j)}^\text{sup}/X^{1/2}$, where the constant $C$ does not depend on $n$ or $j$. 
If we further assume Hypothesis \ref{conj-selmerratio}, then 
\begin{eqnarray*}\pi_{\wtr_r\cap\wss_{n(j)}} (X)&\myeq & \frac{5\kappa}{6}  \binom{\lfloor \frac{r}{2} \rfloor+j}{j} \int_{0}^X \cfrac{ s_{n(j)}\cdot \mathbb{E}_{\lfloor \frac{r}{2} \rfloor,j}^{n(j)}(H)}{(1+C_{n(j)}H^{-e_{n(j)}})\cdot H^{1/6}}\, dH + O\left(X^{1/2}\right),
\end{eqnarray*}
as claimed.
\end{proof} 

If we now use the formula $\pi_{\wtr_r}(X) = \sum_{j=0}^\infty \pi_{\wtr_r\cap \wss_{r+2j}}(X)$ and the fact that $\sum_{n=0}^\infty \theta_n^{\text{sup}}\leq 1+1+\sum_{n=2}^\infty s_n$ is bounded (by Lemma \ref{lem-seriesconverge}) we obtain the following result. 

\begin{cor}\label{cor-predictrank}
Let $X,r\geq 0$ be fixed, and let $\wtt \in \mathbf{T}$ be arbitrary. If we assume $H_A$ and $H_B$, then the expected value of $\pi_{\wtr_r}(X)$ is given by the formula
$$\mathbb{E}\left(\pi_{\wtr_r} (X)\right) \myeq \frac{5\kappa}{6}\sum_{j=0}^\infty \binom{\lfloor \frac{r}{2} \rfloor+j}{j}\int_0^{X} \cfrac{\theta_{n(j)}(H)}{H^{1/6}}\cdot \mathbb{E}_{\lfloor \frac{r}{2} \rfloor,j}^{n(j)}(H)\, dH +  O\left(X^{1/2}\right).$$
where  $\mathbb{E}_{\lfloor \frac{r}{2} \rfloor,j}^{n(j)}(H)$ is the expected value defined in Remark \ref{rem-notation}.
\end{cor}

\begin{remark}
If we assume $H_A$, $H_B$, and Hypotheses \ref{conj-selmerratio} and \ref{conj-hasseratio}, and in addition (for the sake of simplicity) we assume that the random  variables $Y_1,\ldots,Y_{\lfloor n(j)/2\rfloor}$ are independent in $S=\bigcup \Sel_2(E)$, then we would have
\begin{eqnarray*}\mathbb{E}(\pi_{\wtr_r}(X)) &\myeq & \frac{5\kappa}{6} \sum_{j=0}^\infty \binom{\lfloor \frac{r}{2} \rfloor+j}{j} \int_{0}^X \cfrac{\theta_{n(j)}(H)}{H^{1/6}}\cdot \rho_{n(j)}(H)^{\lfloor r/2 \rfloor}(1-\rho_{n(j)}(H))^j\, dH +  O\left(X^{1/2}\right)\\
&\myeq & \frac{5\kappa}{6} \sum_{j=0}^\infty \binom{\lfloor \frac{r}{2} \rfloor+j}{j} \int_{0}^X \cfrac{ s_{n(j)}\cdot (D_{n(j)})^{\lfloor \frac{r}{2} \rfloor}\cdot (H^{f_{n(j)}}-D_{n(j)})^j}{(1+C_{n(j)}H^{-e_{n(j)}})\cdot H^{1/6+(\lfloor \frac{r}{2} \rfloor+j)\cdot f_{n(j)}}}\, dH +  O\left(X^{1/2}\right).
\end{eqnarray*} 
If we simplify this expression further by just retaining the highest order term (and for now assume $r\geq 2$). We obtain the following approximations:
\begin{eqnarray*} \pi_{\wtr_r}(X) &\approx & \frac{5\kappa}{6} \sum_{j=0}^\infty \binom{\lfloor \frac{r}{2} \rfloor+j}{j} \cdot   s_{n(j)}\cdot (D_{n(j)})^{\lfloor \frac{r}{2} \rfloor} \int_{0}^X \cfrac{ 1 }{H^{1/6+\lfloor \frac{r}{2} \rfloor\cdot f_{n(j)}}}\, dH\\
&\approx & \frac{5\kappa}{6} \sum_{j=0}^\infty \binom{\lfloor \frac{r}{2} \rfloor+j}{j} \cdot s_{n(j)}\cdot (D_{n(j)})^{\lfloor \frac{r}{2} \rfloor} \cdot \frac{X^{5/6-\lfloor \frac{r}{2} \rfloor\cdot f_{n(j)}}}{5/6-\lfloor \frac{r}{2} \rfloor\cdot f_{n(j)}}.
\end{eqnarray*}
In particular, if there is $j\geq 0$ such that $\lfloor \frac{r}{2} \rfloor\cdot f_{n(j)}<5/6$, then there are infinitely many (test) elliptic curves with rank $r$ (and Selmer rank $n(j)$). With the data we have at our disposal, for $n\leq 5$, according to Table \ref{tab-rhomodel}, we see that
$$\left\lfloor \frac{r}{2} \right\rfloor\cdot f_{n(j)}\leq \left\lfloor \frac{n}{2} \right\rfloor\cdot f_{n}\leq 2\cdot 0.045=0.09\leq 0.8\overline{3}=5/6.$$
Some further speculation (to be taken with a grain of salt due to the accumulation of assumptions and simplifications):
\begin{itemize}
	\item If the values of $\{f_{2n}\}$, then $f_{2n}\leq 0.024$ and $\lfloor \frac{r}{2} \rfloor\cdot f_{2n}\leq 5/6$ at least for all $r\leq 69$.
	\item If the values of $\{f_{2n+1}\}$, then $f_{2n+1}\leq 0.045$ and $\lfloor \frac{r}{2} \rfloor\cdot f_{2n}\leq 5/6$ at least for all $r\leq 37$.
	\item  If $f_n\leq 10/(6n)$, then we would always have $\lfloor \frac{r}{2} \rfloor\cdot f_{n(j)}<5/6$. Note that according to Table \ref{tab-rhomodel} we have 
	$$f_2 \leq 0.024\leq 0.8\overline{3}=\frac{10}{12},\ f_3\leq 0.045\leq 0.\overline{5}= \frac{10}{18},\ f_4\leq 0.022\leq 0.41\overline{6}= \frac{10}{24},\ f_5\leq 0.044\leq 0.\overline{3}= \frac{10}{30}.$$
\end{itemize}

\end{remark}

In our next result, we use Theorem \ref{thm-predictrank} to write formulas for the contribution in rank $r=1,\ldots,5$ coming from Selmer ranks $n=1,\ldots,5$.

\begin{cor}\label{cor-predictrank2}
If we assume $H_A$, $H_B$, and Hypotheses \ref{conj-selmerratio} and \ref{conj-hasseratio}, then the formulas in Corollary  \ref{cor-rankprobabilities} imply approximations of $\pi_{\wtr_r\cap \wss_n}(X)$ as given in Table \ref{tab-predictrank}, for $1\leq r\leq n\leq 5$ and $r\equiv n\bmod 2$.
\end{cor} 

\begin{remark}\label{rem-predictrank}
Using the formulas given by Corollary \ref{cor-predictrank2} and Table \ref{tab-predictrank2}, we can give approximations of $\pi_{\mathcal{R}_r}(X)$. For instance,
$$\pi_{\mathcal{R}_1}(X)\approx \pi_{\mathcal{R}_1\cap \s_1}(X)+\pi_{\mathcal{R}_1\cap \s_3}(X)+\pi_{\mathcal{R}_1\cap \s_5}(X),$$
$$\pi_{\mathcal{R}_2}(X)\approx \pi_{\mathcal{R}_2\cap \s_2}(X)+\pi_{\mathcal{R}_2\cap \s_4}(X),\ \quad \pi_{\mathcal{R}_3}(X)\approx \pi_{\mathcal{R}_3\cap \s_3}(X)+\pi_{\mathcal{R}_3\cap \s_5}(X),$$
$$\pi_{\mathcal{R}_4}(X)\approx \pi_{\mathcal{R}_4\cap \s_4}(X),\ \quad \pi_{\mathcal{R}_5}(X)\approx \pi_{\mathcal{R}_5\cap \s_5}(X).$$
We have used SageMath to numerically integrate and compute said approximations, and we have graphed the results in Figures \ref{fig-predictrank123} (for $r=1,2,3$) and \ref{fig-predictrank45} (for $r=4,5$). In Table \ref{tab-predictrank2} we have included the values of $\pi_{\mathcal{R}_r}(2.7\cdot 10^{10})$ according to the data, the values of our approximation, the error, and the relative error (as a percentage of the actual value), and also $s_r \cdot (2.7\cdot 10^{10})^{1/2}$, which is, approximately, the size of the error as expected from Corollary \ref{cor-predictrank}.
\end{remark}

\begin{center}
\begin{figure}[h!]
\includegraphics[width=6.6in]{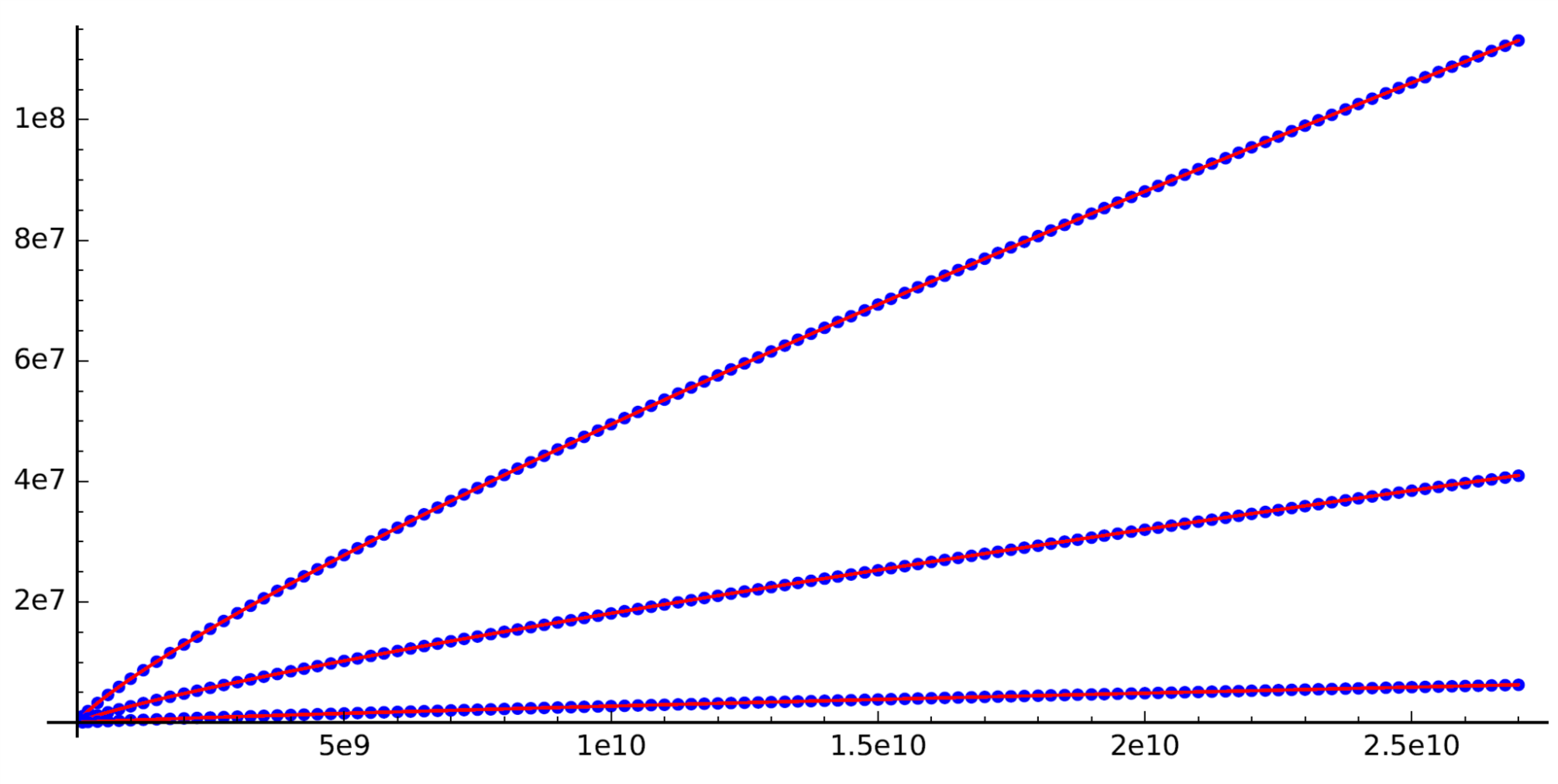}
\caption{Values of $\pi_{\mathcal{R}_r}(X)$ from the BHKSSW database (blue dots) for $r=1,2,3$, and the approximations given in Remark \ref{rem-predictrank} (in red).} \label{fig-predictrank123}
\end{figure}
\end{center}

\begin{center}
\begin{figure}[h!]
\includegraphics[width=6.6in]{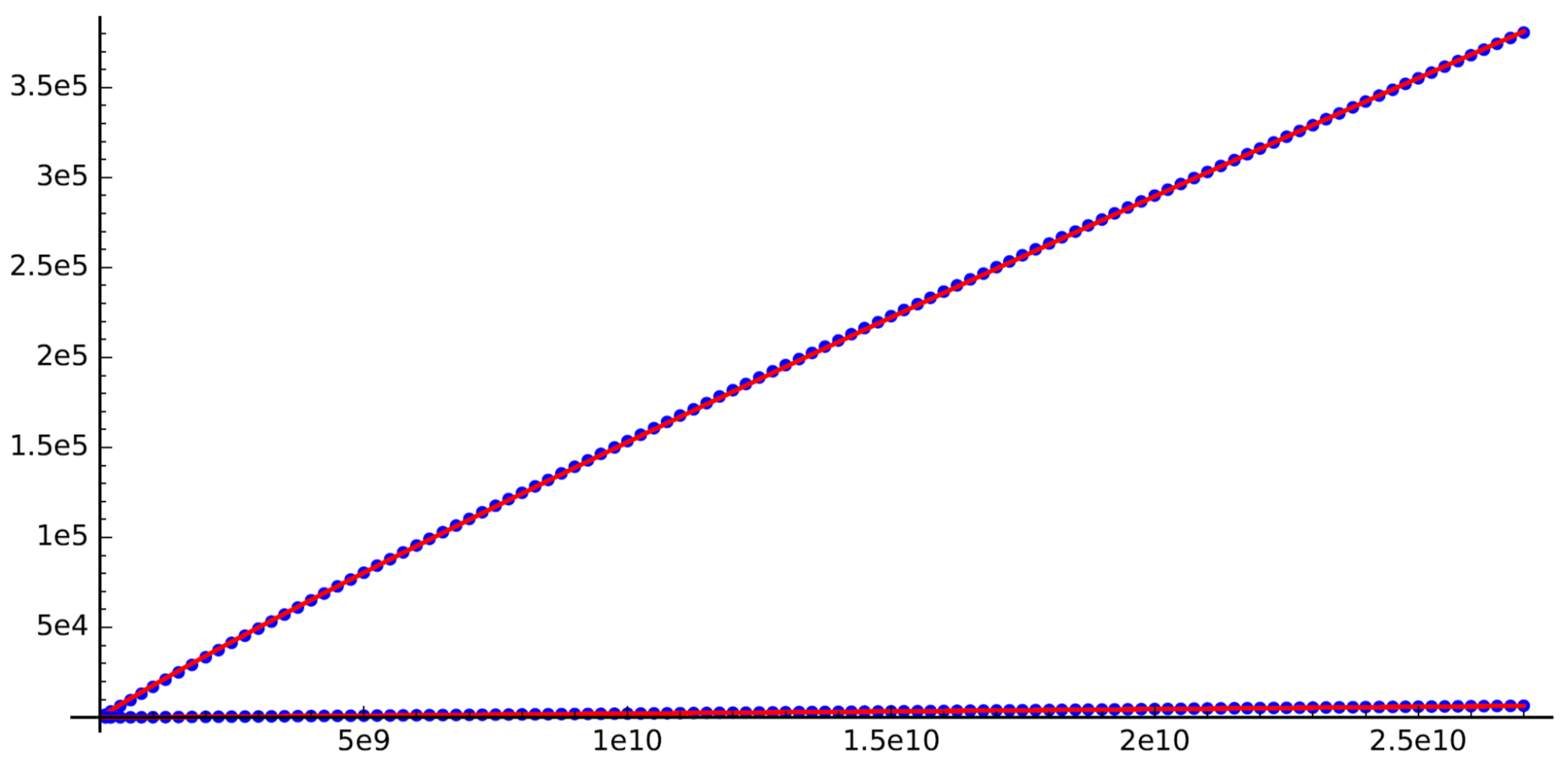}
\caption{Values of $\pi_{\mathcal{R}_r}(X)$ from the BHKSSW database (blue dots) for $r=4,5$, and the approximations given in Remark \ref{rem-predictrank} (in red).} \label{fig-predictrank45}
\end{figure}
\end{center}

\begin{table}[h!]
\centering
\def\arraystretch{3.5}
\begin{tabular}{c|c}
$\pi_{\wtr_1\cap \wss_1}(X)$ & $\displaystyle \frac{5\kappa}{6}  \int_{0}^X \cfrac{\theta_{1}(H)}{H^{1/6}}\, dH \approx  \frac{5\kappa}{6}  \int_{0}^X \cfrac{s_1}{(1+C_1H^{-e_1})H^{1/6}}\, dH$\\
$\pi_{\wtr_1\cap \wss_3}(X)$ & $\displaystyle \frac{5\kappa}{6}  \int_{0}^X \cfrac{\theta_{3}(H)}{H^{1/6}}\cdot (1-\rho_{3}(H))\, dH \approx  \frac{5\kappa}{6}  \int_{0}^X \cfrac{s_3\cdot (H^{f_3}-D_3)}{(1+C_3H^{-e_3})H^{1/6+f_3}}\, dH.$\\
$\pi_{\wtr_1\cap \wss_5}(X)$ & $\displaystyle \frac{5\kappa}{6}  \int_{0}^X \cfrac{\theta_{5}(H)}{H^{1/6}}\cdot (1-\rho_{5}(H))^2\, dH \approx  \frac{5\kappa}{6}  \int_{0}^X \cfrac{s_5\cdot (H^{f_5}-D_5)^2}{(1+C_5H^{-e_5})H^{1/6+2f_5}}\, dH.$\\
$\pi_{\wtr_2\cap \wss_2}(X)$ & $\displaystyle \frac{5\kappa}{6}  \int_{0}^X \cfrac{\theta_{2}(H)}{H^{1/6}}\cdot \rho_{2}(H)\, dH \approx  \frac{5\kappa}{6}  \int_{0}^X \cfrac{s_2\cdot D_2}{(1+C_2H^{-e_2})H^{1/6+f_2}}\, dH.$\\
$\pi_{\wtr_2\cap \wss_4}(X)$ & $\displaystyle \frac{10\kappa}{6}  \int_{0}^X \cfrac{\theta_{4}(H)}{H^{1/6}}\cdot (\rho_{4}(H)(1-\rho_4(H))-C_{1,1}^4(X))\, dH$\\
 & $\displaystyle \approx \frac{10\kappa}{6}  \int_{0}^X \cfrac{s_4\cdot (-(D_4)^2 + D_4\cdot H^{f_4}+0.025\cdot (H^{f_4})^2)}{(1+C_4H^{-e_4})H^{1/6+2f_4}}\, dH.$\\
 $\pi_{\wtr_3\cap \wss_3}(X)$ & $\displaystyle \frac{5\kappa}{6}  \int_{0}^X \cfrac{\theta_{3}(H)}{H^{1/6}}\cdot \rho_{3}(H)\, dH \approx  \frac{5\kappa}{6}  \int_{0}^X \cfrac{s_3\cdot D_3}{(1+C_3H^{-e_3})H^{1/6+f_3}}\, dH.$\\
 $\pi_{\wtr_3\cap \wss_5}(X)$ & $\displaystyle \frac{10\kappa}{6}  \int_{0}^X \cfrac{\theta_{5}(H)}{H^{1/6}}\cdot (\rho_{5}(H)(1-\rho_5(H))-C_{1,1}^5(X))\, dH$\\
  & $\displaystyle \approx \frac{10\kappa}{6}  \int_{0}^X \cfrac{s_5\cdot D_5 \cdot (H^{f_5}-D_5)}{(1+C_5H^{-e_5})H^{1/6+2f_5}}\, dH.$\\
  $\pi_{\wtr_4\cap \wss_4}(X)$ & $\displaystyle \frac{5\kappa}{6}  \int_{0}^X \cfrac{\theta_{4}(H)}{H^{1/6}}\cdot (\rho_{4}(H)^2+C_{1,1}^4(X))\, dH \approx  \frac{5\kappa}{6}  \int_{0}^X \cfrac{s_4\cdot ((D_4)^2-0.025\cdot (H^{f_4})^2)}{(1+C_4H^{-e_4})H^{1/6+2f_4}}\, dH.$\\
    $\pi_{\wtr_5\cap \wss_5}(X)$ & $\displaystyle \frac{5\kappa}{6}  \int_{0}^X \cfrac{\theta_{5}(H)}{H^{1/6}}\cdot (\rho_{5}(H)^2+C_{1,1}^5(X))\, dH \approx  \frac{5\kappa}{6}  \int_{0}^X \cfrac{s_5\cdot (D_5)^2}{(1+C_5H^{-e_5})H^{1/6+2f_5}}\, dH.$\\
\end{tabular}
\caption{Approximate values of $\pi_{\wtr_r\cap\wss_n}(X)$ for $1\leq r\leq n\leq 5$ and $r\equiv n\bmod 2$.}\label{tab-predictrank}
\end{table}

\begin{table}[h!]
\centering
\def\arraystretch{2}
\begin{tabular}{c|ccccc}
& $r=1$ & $2$ & $3$ & $4$ & $5$\\
\hline
$\pi_{\mathcal{R}_r}(2.7\cdot 10^{10})$ & $113128929$ & $40949289$ & $6259157$ & $380519$ & $6481$ \\
Approximate value & $113133971$ &
 $41005107$ & 
 $6273138$ &
 $381272$ &
 $6438$ \\
 $|\text{Error}|$ & $5042$ & $55818$ & $13981$ & $753$ & $43$\\
 Error $\%$ &  $0.004456$ & $0.136310$ & $0.223368$ & $0.197887$ &
  $0.663477$\\
  Predicted error $\approx s_r\cdot X^{1/2}$ & $68848.72$ & $45942.96$ & $13112.47$ & $1749.97$ & $111.73$\\ 
\end{tabular}
\caption{Values of $\pi_{\mathcal{R}_r}(2.7\cdot 10^{10})$ from the BHKSWW database, the approximate values (rounded to the closest integer) given by numerical integration of the formulas in Table \ref{tab-predictrank} and Remark \ref{rem-predictrank}, the absolute error, the error as a percentage of the actual value of $\pi_{\mathcal{R}_r}$, and the size of the predicted error $s_r\cdot (2.7\cdot 10^{10})^{1/2}$ from Corollary \ref{cor-predictrank}.}
\label{tab-predictrank2}
\end{table}

Next, using the formulas from Corollary \ref{cor-predictrank2}, we can estimate the rate of growth of our rank counting functions $\pi_{\mathcal{R}_r\cap\s_n}(X)$. For instance, the next corollary does this for $r=n=1$, $2$, and $3$.

\begin{cor}
Let $\wtt\in\mathbf{T}$ be arbitrary. If we assume $H_A$, $H_B$, and Hypotheses \ref{conj-selmerratio} and \ref{conj-hasseratio}, then there are explicit computable positive constants $\lambda_r$ and $h_r$, for $n=1,2,3$ such that 
$$\mathbb{E}(\pi_{\wtr_1\cap \wss_1}(X)) \myeq \lambda_1+ \kappa s_1 X^{5/6}\cdot\sum_{m=0}^\infty \frac{(-C_1)^m}{1-(6/5)\cdot m e_1}X^{-me_1} + O(X^{1/2}),$$
$$\mathbb{E}(\pi_{\wtr_2\cap \wss_2}(X)) \myeq \lambda_2+ \kappa s_2 D_2 X^{5/6-f_2}\cdot\sum_{m=0}^\infty \frac{(-C_2)^m}{1-(6/5)\cdot (f_2+m e_2)}X^{-me_2}+ O(X^{1/2}),$$
$$\mathbb{E}(\pi_{\wtr_3\cap \wss_3}(X)) \myeq \lambda_3+\kappa s_3 D_3 X^{5/6-f_3}\cdot\sum_{m=0}^\infty \frac{(-C_3)^m}{1-(6/5)\cdot (f_3+m e_3)}X^{-me_3}+ O(X^{1/2}),$$
for any $X\geq h_r$.
\end{cor}
\begin{proof}
For each $r=1,2,3$, let $h_r> 0$ be the smallest natural number such that $|C_rh_r^{-e_r}|<1$, where $C_r$ are the constants in Hypothesis \ref{conj-selmerratio}.  Then,
$$\pi_{\wtr_r\cap \wss_r}(X) = \pi_{\wtr_r\cap \wss_r}(h_r) + \pi_{\wtr_r\cap \wss_r}([h_r,X])$$
and, by Corollary \ref{cor-predictrank} we have
$$\mathbb{E}(\pi_{\wtr_r\cap \wss_{r}}([h_0,X])) \myeq \frac{5\kappa}{6}  \binom{\lfloor \frac{r}{2} \rfloor}{0} \int_{h_0}^X \cfrac{ s_{r}\cdot \mathbb{E}_{\lfloor \frac{r}{2} \rfloor,0}^{r}(H)}{(1+C_{r}H^{-e_{r}})\cdot H^{1/6}}\, dH + O(X^{1/2}).$$
Further, since $|C_rh_r^{-e_r}|<1$, we can write
$$\frac{1}{1+C_rH^{-e_r}} = \sum_{m=0}^\infty (-C_r)^m H^{-me_r}$$
for any $H\geq h_r$. Now, $\mathbb{E}_{\lfloor \frac{r}{2} \rfloor,0}^{r}(H)=1$ for $r=1$, and by Corollary \ref{cor-rankprobabilities}, we have $\mathbb{E}_{\lfloor \frac{r}{2} \rfloor,0}^{r}(H) = \rho_r(X)$ for $r=2,3$. Further, assuming $H_B$ we have $\rho_n(X)=D_n/X^{f_n}$. Putting everything together we obtain, for instance, the following approximation formula for $\mathbb{E}(\pi_{\wtr_1\cap \wss_1}(X))$
\begin{eqnarray*} & \myeq  & \pi_{\wtr_1\cap \wss_1}(h_1)+\pi_{\wtr_1\cap \wss_1}([h_1,X]) =   \pi_{\wtr_1\cap \wss_1}(h_1) + \frac{5\kappa}{6}  \int_{h_1}^X \cfrac{\theta_{1}(H)}{H^{1/6}}\, dH + O(X^{1/2})
\end{eqnarray*}
\begin{eqnarray*}
&= & \pi_{\wtr_1\cap \wss_1}(h_1)+ \frac{5\kappa s_1}{6}  \int_{h_1}^X \sum_{m=0}^\infty (-C_r)^m H^{-\frac{1}{6}-me_1}\, dH + O(X^{1/2})\\
&= & \pi_{\wtr_1\cap \wss_1}(h_1)-\left(\kappa s_1 h_1^{5/6}\cdot\sum_{m=0}^\infty \frac{(-C_1)^m}{1-(6/5)\cdot m e_1}h_1^{-me_1}\right)+\kappa s_1 X^{5/6}\cdot\sum_{m=0}^\infty \frac{(-C_1)^m}{1-(6/5)\cdot m e_1}X^{-me_1}\\
&= & \lambda_1+ \kappa s_1 X^{5/6}\cdot\sum_{m=0}^\infty \frac{(-C_1)^m}{1-(6/5)\cdot m e_1}X^{-me_1} + O(X^{1/2}),\end{eqnarray*}
with $\lambda_1 = \pi_{\wtr_1\cap \wss_1}(h_0)-\left(\kappa s_1 h_1^{5/6}\cdot\sum_{m=0}^\infty \frac{(-C_1)^m}{1-(6/5)\cdot m e_1}h_1^{-me_1}\right)$, and we derive formulas for $r=2$ and $r=3$ in a similar manner.
\end{proof}

\section{Predicting the average rank}\label{sec-predictave}

In this section we shall  estimate the average rank of all elliptic curves of height $\leq X$:
$$\operatorname{AvgRank}_\mathcal{E}(X)=\cfrac{\sum_{E\in \mathcal{E}(X)} \rank(E(\Q))}{\pi_\mathcal{E}(X)}$$
We quote here the average rank conjecture as in \cite{poonen} (see \cite{goldfeld} for Goldfeld's version for quadratic twists).
\begin{conj}\label{conj-5050} Fix a global field $k$. Asymtotically, $50\%$ of elliptic curves over $k$ have rank $0$, and $50\%$ have rank $1$. Moreover, the average rank is $1/2$.
\end{conj} 

We consider the average rank contributions from the subsets of test elliptic curves of each Selmer rank $n\geq 1$:
$$\operatorname{AvgRank}_{\wts_n}(X)=\cfrac{\sum_{E\in \wts_n(X)} \rank(E)}{\pi_{\wtt}(X)}$$
and later we will put them together to estimate the total average rank.

\begin{thm}\label{thm-averank}
Let $\wtt\in\mathbf{T}$ be arbitrary. Assume $H_A$ and $H_B$, and let $n\geq 1$ be fixed. Then, the expected value of $\operatorname{AvgRank}_{\wts_n}(X)$ is given on average by   
$$\frac{5\kappa}{6\pi_{\wtt}(X)} \cdot\int_1^X  \frac{\theta_n(H)}{H^{1/6}} \left((n\bmod 2) + 2\left\lfloor{\frac{n}{2}}\right\rfloor \rho_n(H)\right)\, dH +  O(X^{-1/3}),$$
where the implied error in the approximation is bounded by $C\cdot \theta_{n}^\text{sup}/X^{1/3}$, for some constant $C$ that does not depend on $n$. Moreover, the error in approximating $\operatorname{AvgRank}_{\wts_n}(X)$ by its expected value is given by, on average, by
$$\sqrt{\frac{5\kappa \lfloor n/2 \rfloor}{6\pi_{\wtt}(X)^2} \int_0^X \frac{\theta_n(H)}{H^{1/6}} (\rho_n(H)(1-\rho_n(H)) + (\lfloor n/2 \rfloor -1) C_{1,1}^n(H))\, dH + O(X^{-7/6})}.$$
\end{thm}
\begin{proof}
We compute the expected value of the average rank in the sequence $\wtt$ as follows:
\begin{eqnarray*}
\mathbb{E}(\operatorname{AvgRank}_{\wts_n}(X)) &=& \mathbb{E}\left( \cfrac{\sum_{E\in \wts_n(X)} \rank(E)}{\pi_{\wtt}(X)}\right) = \frac{1}{\pi_{\wtt}(X)} \mathbb{E}\left( \sum_{E\in \wts_n(X)} \rank(E)\right)\\
&=& \frac{1}{\pi_{\wtt}(X)} \cdot \left(\sum_{E\in \wts_n(X)} (n\bmod 2) + 2\left\lfloor{\frac{n}{2}}\right\rfloor  \rho_n(\h(E))\right)\\
&=& \frac{1}{\pi_{\wtt}(X)} \cdot \sum_{H=1}^X \sum_{E\in \wts_n([H,H])} (n\bmod 2) + 2\left\lfloor{\frac{n}{2}}\right\rfloor  \rho_n(H)\\
&=& \frac{1}{\pi_{\wtt}(X)} \cdot \sum_{H=1}^X \pi_{\wts_n}([H,H])\cdot \left( (n\bmod 2) + 2\left\lfloor{\frac{n}{2}}\right\rfloor   \rho_n(H)\right)
\end{eqnarray*}
by Corollary \ref{cor-hasseaverank}. In particular, Definition \ref{defn-TT}, Corollary \ref{cor-countcurvesaverage}, and $H_A$ imply
\begin{eqnarray*}
  &\myeq & \frac{1}{\pi_{\wtt}(X)} \cdot \left(\frac{5\kappa}{6}\int_1^X  \frac{\theta_n(H)}{H^{1/6}} \left((n\bmod 2) + 2\left\lfloor{\frac{n}{2}}\right\rfloor \rho_n(H)\right)\, dH + O(X^{1/2})\right)\\
  &\myeq & \frac{5\kappa}{6 \pi_{\wtt}(X)} \cdot \int_1^X  \frac{\theta_n(H)}{H^{1/6}} \left((n\bmod 2) + 2\left\lfloor{\frac{n}{2}}\right\rfloor \rho_n(H)\right)\, dH + O(X^{-1/3}),\\
\end{eqnarray*}
where we have used the fact that $\wtt\in\mathbf{T}$ for the estimate $\pi_{\wtt}(X) = O(X^{5/6})$, and the implied error in the approximation is bounded by $C\cdot \theta_{n}^\text{sup}/X^{1/3}$, for some constant $C$ that does not depend on $n$. Moreover, by Corollary \ref{cor-hasseaverank}, the standard error in the approximation of the average by the expected value is given by
\begin{eqnarray*}
& & \frac{1}{\pi_{\wtt}(X)}\sqrt{\lfloor n/2 \rfloor \sum_{E\in \wts_n(X)} \rho_n(\h(E))(1-\rho_n(\h(E))) + (\lfloor n/2 \rfloor -1) C_{1,1}^n(\h(E))}\\
&\myeq &  \frac{1}{\pi_{\wtt}(X)} \sqrt{\frac{5\kappa \lfloor n/2 \rfloor}{6} \int_1^X \frac{\theta_n(H)}{H^{1/6}} (\rho_n(H)(1-\rho_n(H)) + (\lfloor n/2 \rfloor -1) C_{1,1}^n(H))\, dH + O(X^{1/2})},\\
&\myeq &  \sqrt{\frac{5\kappa \lfloor n/2 \rfloor}{6\pi_{\wtt}(X)^2} \int_1^X \frac{\theta_n(H)}{H^{1/6}} (\rho_n(H)(1-\rho_n(H)) + (\lfloor n/2 \rfloor -1) C_{1,1}^n(H))\, dH + O(X^{-7/6})}.
\end{eqnarray*}
\end{proof}

\begin{remark}\label{rem-approxave}
Let $h_n$ be the smallest positive integer such that $|C_nh_n^{-e_n}|<1$. If we assume Hypotheses \ref{conj-selmerratio} and \ref{conj-hasseratio}, then $\operatorname{AvgRank}_{\wts_n}(X)$ is given, on average, by
\begin{eqnarray*}
 &\myeq & \frac{5\kappa}{6\pi_{\wtt}(X)} \cdot\int_1^X  \frac{\theta_n(H)}{H^{1/6}} \left((n\bmod 2) + 2\left\lfloor{\frac{n}{2}}\right\rfloor \rho_n(H)\right)\, dH + O(X^{-1/3})\\
&\myeq & \frac{(5/6)\kappa s_n}{\pi_{\wtt}(X)} \cdot\int_1^X  \frac{1}{H^{1/6}(1+C_nH^{-e_n})} \left((n\bmod 2) + 2\left\lfloor{\frac{n}{2}}\right\rfloor\frac{D_n}{H^{f_n}}\right)\, dH+ O(X^{-1/3})\\
&\myeq & \frac{(5/6)\kappa s_n}{\pi_{\wtt}(X)} \cdot\left( \mu_n + \int_{h_n}^X \sum_{m=0}^\infty (-C_n)^m H^{-1/6-me_n}\left((n\bmod 2) + 2\left\lfloor{\frac{n}{2}}\right\rfloor\frac{D_n}{H^{f_n}}\right)\, dH\right)+ O(X^{-1/3})
\end{eqnarray*}
where $\mu_n = \int_{1}^{h_n} \sum_{m=0}^\infty (-C_n)^m H^{-1/6-me_n}\left((n\bmod 2) + 2\left\lfloor{\frac{n}{2}}\right\rfloor\frac{D_n}{H^{f_n}}\right)\, dH$. Thus, we get on average

\begin{eqnarray*}
&  & \frac{(5/6)\kappa s_n}{\pi_{\wtt}(X)}\cdot \left(\mu_n - \overline{n} \sum_{m=0}^\infty  \frac{(-C_n)^m}{5/6-me_n}(h_n)^{5/6-me_n} - 2\left\lfloor{\frac{n}{2}}\right\rfloor \sum_{m=0}^\infty  \frac{D_n(-C_n)^m}{5/6-f_n-me_n}(h_n)^{5/6-f_n-me_n}\right)\\
& & +\frac{(5/6)\kappa s_n}{\pi_{\wtt}(X)}\cdot \left( \overline{n} \sum_{m=0}^\infty  \frac{(-C_n)^m}{5/6-me_n}X^{5/6-me_n} + 2\left\lfloor{\frac{n}{2}}\right\rfloor \sum_{m=0}^\infty  \frac{D_n(-C_n)^m}{5/6-f_n-me_n}X^{5/6-f_n-me_n}\right)+ O(X^{-1/3}),
\end{eqnarray*}
where we have abbreviated $\displaystyle \overline{n} = (n\bmod 2)$, and below we shall write $\tau_n$ for the contents inside the first parenthesis, i.e., $\tau_n = \mu_n - \overline{n} \sum_{m=0}^\infty \ldots -  2\left\lfloor{\frac{n}{2}}\right\rfloor \sum_{m=0}^\infty \ldots$.
\begin{eqnarray*}
& \myeq & \frac{\kappa s_nX^{5/6}}{\pi_{\wtt}(X)}\cdot \left( \frac{\tau_n}{X^{5/6}} +    \sum_{m=0}^\infty  \left(\frac{(n\bmod 2)(-C_n)^m}{1-(6/5)me_n} +  X^{-f_n} \frac{2\left\lfloor{\frac{n}{2}}\right\rfloor D_n(-C_n)^m}{1-(6/5)(f_n+me_n)}\right)X^{-me_n}\right)+ O(X^{-1/3})\\
& \myeq & s_n \cdot  \left( \frac{\tau_n}{X^{5/6}} +   \sum_{m=0}^\infty  \left(\frac{(n\bmod 2)(-C_n)^m}{1-(6/5)me_n} +  X^{-f_n} \frac{2\left\lfloor{\frac{n}{2}}\right\rfloor D_n(-C_n)^m}{1-(6/5)(f_n+me_n)}\right)X^{-me_n}\right)+ O(X^{-1/3}).
\end{eqnarray*}
Hence, we obtain the following result about the average rank of (test) elliptic curves.
\end{remark}

\begin{cor}\label{cor-averank}
If we assume $H_A$ and $H_B$, and Hypotheses \ref{conj-selmerratio} and \ref{conj-hasseratio}, then there are constants $\tau_n$ such that the expected value of $\operatorname{AvgRank}_{\wtt}(X)$ is given by 
\begin{eqnarray*}
 &=& \sum_{n=1}^\infty 
\operatorname{AvgRank}_{\wts_n}(X)\\
&\myeq & \sum_{n=1}^\infty s_n \cdot  \left(\frac{\tau_n}{X^{5/6}} +  \sum_{m=0}^\infty  \left(\frac{(n\bmod 2)(-C_n)^m}{1-(6/5)me_n} +  X^{-f_n} \frac{2\left\lfloor{\frac{n}{2}}\right\rfloor D_n (-C_n)^m}{1-(6/5)(f_n+me_n)}\right)X^{-me_n}\right) + O(X^{-1/3}).
\end{eqnarray*}
with standard error $\leq \cfrac{\sum_{n=2}^\infty \sqrt{\lfloor n/2 \rfloor \cdot (\lfloor n/2 \rfloor -3/4)\cdot s_n}}{\sqrt{\kappa}X^{5/12}}$.
In particular,
$$\lim_{X\to \infty} \operatorname{AvgRank}_{\wtt}(X) \myeq \sum_{k=0}^\infty s_{2k+1}=\frac{1}{2},$$
in the sense that the expected value goes on average to $1/2$ and the standard error goes to $0$ on average as $X\to \infty$.
\end{cor}
\begin{proof}
The approximation of the average rank is an immediate consequence of our approximation of the contribution to the average rank coming from each Selmer rank $n$ given in Remark \ref{rem-approxave}. From the approximation, it follows that
$$\lim_{X\to \infty} \operatorname{AvgRank}_{\wtt}(X) \myeq \sum_{n=1}^\infty s_{n} \cdot (n \bmod 2) =  \sum_{k=0}^\infty s_{2k+1}.$$
Finally, we point out that, by Proposition 2.6 of \cite{poonen}, the values $s_n$ have a generating function
$$\sum_{n\geq 0} s_n z^n = \prod_{i=0}^\infty \frac{1+2^{-i}z}{1+2^{-i}}.$$
In particular, for $z=1$ we obtain that $\sum_n s_n =1$, for $z=-1$ we obtain that $\sum_n (-1)^n s_n = 0$, and therefore $\sum_{k} s_{2k+1}=\sum_{n\equiv 1 \bmod 2} s_n = \frac{1}{2} \left(\sum_n s_n -\sum_n (-1)^n s_n\right)=\frac{1}{2}.$ Let us now estimate the error in the approximation of $\operatorname{AvgRank}_{\wts_n}(X)$ using Corollary \ref{cor-hasseaverank}:

\begin{eqnarray*}
& & \lfloor n/2 \rfloor \sum_{E\in \wts(X)} \rho_n(\h(E))(1-\rho_n(\h(E))) + (\lfloor n/2 \rfloor -1) C_{1,1}^n(\h(E))\\
&\myeq &  \frac{5\kappa \lfloor n/2 \rfloor}{6} \int_1^X \frac{\theta_n(H)}{H^{1/6}} (\rho_n(H)(1-\rho_n(H)) + (\lfloor n/2 \rfloor -1) C_{1,1}^n(H))\, dH.\\
&\myeq &  \frac{5\kappa \lfloor n/2 \rfloor s_n}{6} \int_1^X \frac{1}{H^{1/6}(1+C_nH^{-e_n})} \left(\frac{D_n}{H^{f_n}}\left(1-\frac{D_n}{H^{f_n}}\right) + (\lfloor n/2 \rfloor -1) C_{1,1}^n(H)\right)\, dH.
\end{eqnarray*}
Recall that the covariance coefficient $C_{1,1}^n(X)$ is given by $\mathbb{E}(Y_1Y_2)-\mathbb{E}(Y_1)\mathbb{E}(Y_2)$, and since the random  variables $Y_i$ take only the values $0,1$, we have $0\leq \mathbb{E}(Y_1)=\rho_n(X)\leq 1$. In particular, $|C_{1,1}^n(X)|\leq 1$. Also, notice that $y(1-y)$ in the interval $[0,1]$ obtains the maximum value of $1/4$ at $y=1/2$.   Thus,
\begin{eqnarray*}
&\myeq &  \frac{5\kappa \lfloor n/2 \rfloor s_n}{6} \int_1^X \frac{1}{H^{1/6}(1+C_nH^{-e_n})} \left(\frac{D_n}{H^{f_n}}\left(1-\frac{D_n}{H^{f_n}}\right) + (\lfloor n/2 \rfloor -1) C_{1,1}^n(H)\right)\, dH\\
&\leq& \frac{5\kappa \lfloor n/2 \rfloor s_n}{6} \int_1^X \frac{1/4+ (\lfloor n/2 \rfloor -1)}{H^{1/6}}\, dH\\
&\leq & \kappa \lfloor n/2 \rfloor s_n (\lfloor n/2 \rfloor -3/4)X^{5/6} 
\end{eqnarray*}
Thus, the standard error in the approximation of $\operatorname{AvgRank}_{\wtt}(X)$ is bounded by
$$\frac{\sum_{n=2}^\infty \sqrt{\kappa\lfloor n/2 \rfloor (\lfloor n/2 \rfloor -3/4)s_nX^{5/6}}}{\pi_{\wtt}(X)}$$
Since $\pi_{\wtt}(X)=\kappa X^{5/6}+O(X^{1/2})$, it suffices to show that $\sum_{n=2}^\infty \sqrt{\lfloor n/2 \rfloor (\lfloor n/2 \rfloor -3/4)s_n}$ is convergent. Let us define $t_1=s_1$ and 
$$t_n = \frac{t_1}{2^{\frac{n(n-1)}{2}-1}}$$
for $n\geq 2$. Then, the definition of $s_n$ implies that $s_n\leq t_n$, and therefore,
$$\sum_{n=2}^N \sqrt{\lfloor n/2 \rfloor (\lfloor n/2 \rfloor -3/4)s_n} \leq \sum_{n=2}^N \frac{n}{2}\sqrt{s_n}\leq \sum_{n=2}^N \frac{n}{2}\sqrt{t_n}\leq \sum_{n=2}^N \frac{n}{2} \frac{\sqrt{t_1}}{2^{\frac{n(n-1)-2}{4}}}\leq \sum_{n=2}^N \frac{\sqrt{s_1}\cdot n}{2^{\frac{n(n-1)+2}{4}}}$$
for any $N$, and therefore $\sum_{n=2}^\infty \sqrt{\lfloor n/2 \rfloor (\lfloor n/2 \rfloor -3/4)s_n}$ is convergent. Thus, the standard error goes to $0$ on average as $X\to\infty$, as desired.
\end{proof}

\begin{remark}\label{rem-averank}
Using SageMath, in Figure \ref{fig-averank} we have plotted values of $\operatorname{AvgRank}_{\mathcal{E}}(X)$ from the BHKSSW database, and (via numerical integration) the sum of the approximations given in Theorem \ref{thm-averank} of $\operatorname{AvgRank}_{\s_n}(X)$ for $n=1,\ldots,5$. According to the database, we have
$$\operatorname{AvgRank}_{\mathcal{E}}(2.7\cdot 10^{10}) = 0.90197580$$
while our approximation gives $0.90244770$. Thus, the absolute error is $0.00047189$, which represents a $0.0523\%$ of the true value.
\end{remark}

\begin{center}
\begin{figure}[h!]
\includegraphics[width=6.6in]{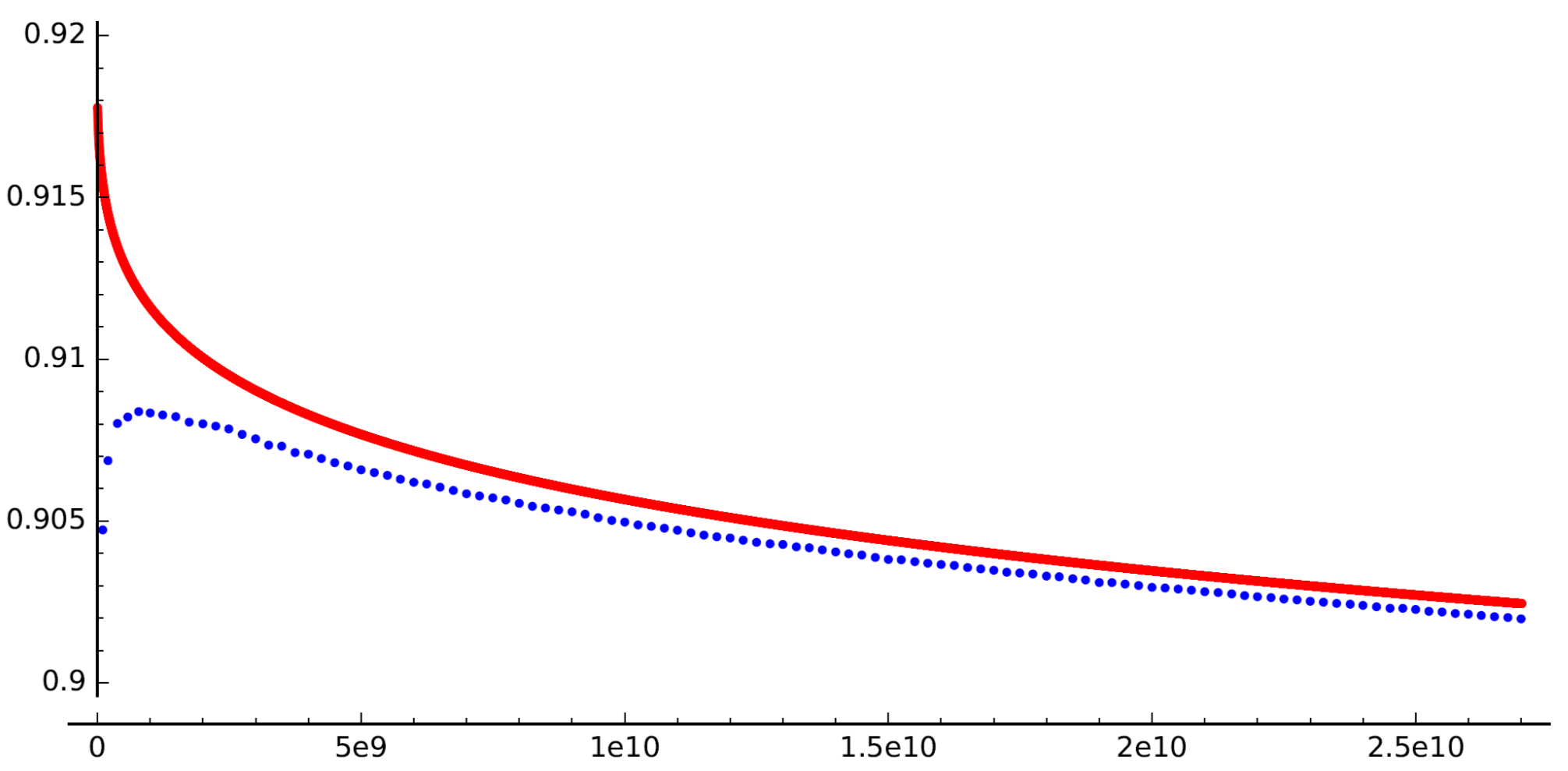}
\caption{Values of $\operatorname{AvgRank}_{\mathcal{E}}(X)$ from the BHKSSW database (blue dots), and the approximation given in Corollary \ref{cor-averank} (in red).} \label{fig-averank}
\end{figure}
\end{center}


\begin{thebibliography}{9}

\bibitem{BHKSSW} J. S. Balakrishnan, W. Ho, N. Kaplan, S. Spicer, W. Stein, J. Weigandt, {\it Databases of elliptic curves ordered by height and distributions of Selmer groups and ranks}, to appear in ANTS XII.

\bibitem{BMSW} B. Bektemirov, B. Mazur, W. Stein, M. Watkins, {\it Average ranks of elliptic curves: tension between data and conjecture}, Bull. (New Series) Amer. Math. Soc., Vol. 44, No. 2, April 2007, 233-254.

\bibitem{magma} W. Bosma, J. Cannon, and C. Playoust, {\it The {M}agma algebra system. {I}. {T}he user language}, J. Symolic Comput., \textbf{24} (1997), 235--265.

\bibitem{brumer} A. Brumer, {\it The average rank of elliptic curves I}, Invent. Math. 109 (1992), no. 3, 445-472.

\bibitem{cassels} J. W. S. Cassels, {\it Arithmetic on curves of genus 1. IV. Proof of the Hauptvermutung}, Journal f\"ur die reine und angewandte Mathematik, 211: 95-112.

\bibitem{cramer} H. Cram\'er, {\it On the order of magnitude of the difference between consecutive primes}, Acta Arith. 2, 23-46.

\bibitem{cremona} J.~E. Cremona, {\it Elliptic curve data}, database available at \url{http://homepages.warwick.ac.uk/~masgaj/ftp/data/}.

\bibitem{cremonamazur} J. E. Cremona, B. Mazur, {\it Visualizing Elements in the Shafarevich–Tate Group}, Experiment. Math. 9:1 (2000), 13-28.

\bibitem{del1} C. Delaunay, {\it Heuristics on Tate-Shafarevitch groups of elliptic curves defined over Q},
Experiment. Math. 10 (2001), no. 2, 191-196.

\bibitem{del2} C. Delaunay, {\it Moments of the orders of Tate-Shafarevich groups}, Int. J. Number Theory 1 (2005), no. 2, 243-264.

\bibitem{dujella} A. Dujella, Website: \url{https://web.math.pmf.unizg.hr/~duje/tors/tors.html}, {\it High rank elliptic curves with prescribed torsion}.

\bibitem{goldfeld} D. Goldfeld, {\it Conjectures on elliptic curves over quadratic fields}, Number theory, Carbondale
1979 (Proc. Southern Illinois Conf., Southern Illinois Univ., Carbondale, Ill., 1979), Lecture
Notes in Math., vol. 751, Springer, Berlin, 1979, pp. 108-118.

\bibitem{granville} A. Granville, {\it Harald Cram\'er and the Distribution of Prime Numbers}, Scand. Actuarial J. 1995; 1: 12-28.

\bibitem{hardy} G.H. Hardy, E. M. Wright, {\it An Introduction to the Theory of Numbers}, 5th ed. Oxford, England: Clarendon Press, 1979.

\bibitem{harron} R. Harron, A. Snowden, {\it Counting elliptic curves with prescribed torsion}, to appear in Journal f\"ur die reine und angewandte Mathematik.

\bibitem{heath1} D. R. Heath-Brown, {\it The size of Selmer groups for the congruent number problem}, Invent. Math.
111 (1993), no. 1, 171-195.

\bibitem{heath2} D. R. Heath-Brown, {\it The size of Selmer groups for the congruent number problem. II}, Invent. Math. 118
(1994), no. 2, 331-370.

\bibitem{kane} D. M. Kane, {\it On the ranks of the 2-Selmer groups of twists of a given elliptic curve}, Algebra Number Theory 7 (2013), 1253-1279.

\bibitem{kenku} M.\,A. Kenku, {\it On the number of $\Q$-isomorphism classes of elliptic curves in each $\Q$-isogeny class}, J. Number Theory \textbf{15} (1982), 199--202.

\bibitem{kubert} D. S. Kubert, {\it Universal bounds on the torsion of elliptic curves}, Compositio Math. 38 (1979), no. 1, 121-128. 

\bibitem{mazur1} B. Mazur, {\it Modular curves and the Eisenstein ideal}, Inst. Hautes \'Etudes Sci. Publ. Math. \textbf{47} (1977), 33--186.

\bibitem{mazur2} B. Mazur, {\it Rational isogenies of prime degree}, Invent. Math. \textbf{44} (1978), 129--162.

\bibitem{ppvm} J. Park, B. Poonen, J. Voight, M. M. Wood, {\it A heuristic for boundedness of ranks of elliptic curves}, submitted.

\bibitem{poonen} B. Poonen, {\it Average rank of elliptic curves
[after Manjul Bhargava and Arul Shankar]}, S\'eminaire Bourbaki, Janvier 2012 / revised June 13, 64\`eme ann\'ee, 2011-2012, no. 1049.

\bibitem{rains}  B. Poonen, E. Rains, {\it Random maximal isotropic subspaces and Selmer groups}, J. Amer. Math. Soc. 25 (2012), 245-269.

\bibitem{silverman} J. H. Silverman, {\it The arithmetic of elliptic curves}, Springer-Verlag, 2nd Edition, New York, 2009.

\bibitem{swin} P. Swinnerton-Dyer, {\it The effect of twisting on the 2-Selmer group}, Math. Proc. Cambridge Philos.
Soc. 145 (2008), no. 3, 513-526.

\bibitem{sage} The SageMath Developers, Sage Mathematics Software (Version 6.9), 2015, \url{http://www.sagemath.org}

\bibitem{watkins} M. Watkins, {\it Some heuristics about elliptic curves}, Experiment. Math. 17 (2008), no. 1, 105-125.

\end{thebibliography}
\end{document}